\documentclass[10pt,draftcls,onecolumn]{IEEEtran}
\usepackage{amsbsy}
\usepackage{floatflt} 

\usepackage{amsmath}
\usepackage{amssymb}
\usepackage{times}
\usepackage{graphicx}
\usepackage{xspace}
\usepackage{paralist} 
\usepackage{setspace} 
\usepackage{xypic}
\xyoption{curve}
\usepackage{latexsym}
\usepackage{theorem}
\usepackage{ifthen}
\usepackage{subfigure}
\usepackage{turnstile}

\usepackage{booktabs}
\newcommand{\mypar}[1]{\vspace{0.03in}\noindent{\bf #1.}}

{\theoremheaderfont{\it} \theorembodyfont{\rmfamily}
\newtheorem{theorem}{Theorem}[section]
\newtheorem{lemma}[theorem]{Lemma}
\newtheorem{definition}[theorem]{Definition}

\newtheorem{corollary}[theorem]{Corollary}

\newtheorem{proposition}[theorem]{Proposition}
\newtheorem{remark}[theorem]{Remark}

}

\renewcommand{\baselinestretch}{1.3}

\newcommand{\sbb}{\mbox{\scriptsize{sb}}}

\def\ln{{\rm ln}}

\title{Moderate Deviations of the Random Riccati Equation}
\author{Soummya Kar and Jos\'e M.~F.~Moura$^{*}$
\thanks{The authors are with the Dept.~Electrical and Computer Engineering, Carnegie Mellon University, Pittsburgh, PA 15213, USA (e-mail:
soummyak@andrew.cmu.edu, moura@ece.cmu.edu, ph: (412)268-6341, fax: (412)268-3890.)}
\thanks{Work supported by NSF under grants \#~ECS-0225449
and~\#~CNS-0428404.}}

\begin{document}
\maketitle \thispagestyle{empty} \maketitle
\vspace*{-1.5cm}
\begin{abstract}We characterize the invariant filtering measures resulting from Kalman filtering with intermittent observations (\cite{Bruno}), where the observation arrival is modeled as a Bernoulli process. In~\cite{Riccati-weakconv}, it was shown that there exists a $\overline{\gamma}^{\mbox{\scriptsize{sb}}}>0$ such that for every observation packet arrival probability $\overline{\gamma}$,  $\overline{\gamma}>\overline{\gamma}^{\mbox{\scriptsize{sb}}}>0$, the sequence of random conditional error covariance matrices converges in distribution to a unique invariant distribution $\mathbb{\mu}^{\overline{\gamma}}$ (independent of the filter initialization.) In this paper, we prove that, for controllable and observable systems, $\overline{\gamma}^{\mbox{\scriptsize{sb}}}=0$ and that, as $\overline{\gamma}\uparrow 1$, the family $\{\mathbb{\mu}^{\overline{\gamma}}\}_{\overline{\gamma}>0}$ of invariant distributions satisfies a moderate deviations principle (MDP) with a good rate function $I$. The rate function $I$ is explicitly identified. In particular, our results show: \begin{inparaenum}[(1)] \item as $\overline{\gamma}\uparrow 1$, the family $\{\mathbb{\mu}^{\overline{\gamma}}\}$ converges weakly (in distribution) to the Dirac measure $\delta_{P^{\ast}}$, where $P^{\ast}$ is the fixed point of the discrete time Riccati operator; \item the probability of a rare event (an event bounded away from $P^{\ast}$) under $\mathbb{\mu}^{\overline{\gamma}}$ decays to zero as a power law of $(1-\overline{\gamma})$ as $\overline{\gamma}\uparrow 1$. The best exponent of such a power law decay is explicitly obtained by solving a deterministic variational problem involving the MDP rate function $I$. These results offer a complete characterization of the family of invariant distributions $\{\mathbb{\mu}^{\overline{\gamma}}\}_{\overline{\gamma}>0}$.\end{inparaenum}  We provide computationally efficient methods for solving the variational problems in question, leading to efficient estimates of probabilities under the invariant measures. The analytical techniques developed in this paper are fairly general and applicable to the analysis of a broader class of iterated function systems. Several intermediate results obtained in the process are of independent interest.
\end{abstract}
\newenvironment{s_itemize}{\begin{list}{$\bullet$}
{\setlength{\rightmargin}{\leftmargin}
\setlength{\itemsep}{0em}
\setlength{\topsep}{0em}
\setlength{\parsep}{0em}}}{\end{list}}

\newenvironment{s_itemize_2}{\begin{list}{$-$}
{\setlength{\rightmargin}{0em}
\setlength{\itemsep}{0em}
\setlength{\topsep}{0em}
\setlength{\parsep}{0em}}}{\end{list}}

\newenvironment{noinds_itemize}{\begin{list}{$\bullet$}
{\setlength{\rightmargin}{0em}
\setlength{\leftmargin}{1.2em}
\setlength{\itemsep}{0em}
\setlength{\topsep}{0em}
\setlength{\parsep}{0em}}}{\end{list}}

\newenvironment{small_ind_itemize}{\begin{list}{$\bullet$}
{\setlength{\rightmargin}{1em}
\setlength{\leftmargin}{2em}
\setlength{\itemsep}{0em}
\setlength{\topsep}{0em}
\setlength{\parsep}{0em}}}{\end{list}}

\def\descriptionnobflabel#1{\hspace\labelsep #1}
\def\descriptionnobf{\list{}{\labelwidth\z@ \itemindent-\leftmargin
 \let\makelabel\descriptionnobflabel}}
\let\enddescriptionnobf\endlist

\def\s_descriptionlabel#1{\hspace\labelsep \bf #1}
\newenvironment{s_description}{\begin{list}{1}
{\setlength{\rightmargin}{\leftmargin}
\setlength{\itemsep}{0em}
\setlength{\topsep}{0em}
\setlength{\parsep}{0em}
\renewcommand{\makelabel}{\s_descriptionlabel}}}{\end{list}}

\def\s_descriptionlabel#1{\hspace\labelsep \bf #1}
\newenvironment{s_description_nomargin}{\begin{list}{1}
{\setlength{\rightmargin}{0em}
\setlength{\itemsep}{0em}
\setlength{\topsep}{0em}
\setlength{\parsep}{0em}
\renewcommand{\makelabel}{\s_descriptionlabel}}}{\end{list}}

\newenvironment{s_description_nobf}{\begin{list}{1}
{\setlength{\rightmargin}{\leftmargin}
\setlength{\itemsep}{0em}
\setlength{\topsep}{0em}
\setlength{\itemindent}{1em}
\setlength{\parsep}{0em}}}{\end{list}}

\newcommand{\bind}[1]{\hspace*{#1}\begin{minipage}[t]{7in}\begin{itemize}}
\newcommand{\eind}{\end{itemize}\end{minipage}\\}
\newenvironment{inditem}
{\hspace*{1cm}\begin{minipage}[t]{6in}\begin{itemize}}
{\end{itemize}\end{minipage}}

\newenvironment{inditemize}{\begin{list}{$\bullet$}
{\setlength{\rightmargin}{\leftmargin}
\setlength{\itemsep}{\itemsep}
\setlength{\topsep}{\topsep}
\setlength{\itemindent}{1em}
\setlength{\parsep}{\parsep}}}{\end{list}}

\newenvironment{dinditemize}{\begin{list}{$-$}
{\setlength{\rightmargin}{\leftmargin}
\setlength{\itemsep}{\itemsep}
\setlength{\topsep}{\topsep}
\setlength{\itemindent}{1em}
\setlength{\parsep}{\parsep}}}{\end{list}}

\newenvironment{nobulletinditemize}{\begin{list}{}
{\setlength{\rightmargin}{\leftmargin}
\setlength{\itemsep}{\itemsep}
\setlength{\topsep}{\topsep}
\setlength{\itemindent}{0em}
\setlength{\parsep}{\parsep}}}{\end{list}}

\newenvironment{nobullets_itemize}{\begin{list}{}
{\setlength{\rightmargin}{0em}
\setlength{\leftmargin}{0.7em}
\setlength{\itemsep}{0.5em}
\setlength{\topsep}{0.5em}
\setlength{\parsep}{\parsep}}}{\end{list}}

\newenvironment{s_itemize_sp}{\begin{list}{$\bullet$}
{\setlength{\rightmargin}{0em}
\setlength{\itemsep}{0em}
\setlength{\topsep}{0em}
\setlength{\parsep}{0em}}}{\end{list}}

\newcounter{ctr}
\newenvironment{indenumerate}{\begin{list}{\thectr.}
{\usecounter{ctr}
\setlength{\rightmargin}{\leftmargin}
\setlength{\itemsep}{\itemsep}
\setlength{\topsep}{\topsep}
\setlength{\itemindent}{1em}
\setlength{\parsep}{\parsep}}}{\end{list}}

\newenvironment{t_enumerate}{\begin{list}{\thectr.}
{\usecounter{ctr}
\setlength{\rightmargin}{\leftmargin}
\setlength{\itemsep}{0em}
\setlength{\topsep}{\topsep}
\setlength{\itemindent}{\itemindent}
\setlength{\parsep}{\parsep}}}{\end{list}}

\newenvironment{s_enumerate_noindent}{\begin{list}{\thectr.}
{\usecounter{ctr}
\setlength{\rightmargin}{0cm}
\setlength{\leftmargin}{0cm}
\setlength{\itemsep}{0em}
\setlength{\topsep}{0em}
\setlength{\itemindent}{0.5cm}
\setlength{\parsep}{0em}}}{\end{list}}

\newenvironment{s_enumerate}{\begin{list}{\thectr.}
{\usecounter{ctr}
\setlength{\rightmargin}{0.3cm}
\setlength{\leftmargin}{0.3cm}
\setlength{\itemsep}{0em}
\setlength{\topsep}{0em}
\setlength{\itemindent}{0.5cm}
\setlength{\parsep}{0em}}}{\end{list}}

\newenvironment{s_enumerate_ltr}{\begin{list}{\theenumii.}
{\usecounter{enumii}
\setlength{\rightmargin}{0.3cm}
\setlength{\leftmargin}{0.3cm}
\setlength{\itemsep}{0em}
\setlength{\topsep}{0em}
\setlength{\itemindent}{0.4cm}
\setlength{\parsep}{0em}}}{\end{list}}

\def\nolineslidehead#1{\newpage\begin{center}\begin{tabular*}{469pt}[t]{c}
\Huge{\parbox[c]{457pt}{\bf #1}} \\ \vspace{7pt} \\
\end{tabular*}\end{center}\yskip}

\def\centerslidehead#1{\newpage\begin{center}{\Huge{\bf #1}}\end{center}\yskip}

\def\slide#1{
 \newpage
 \begin{center}
 {\LARGE\bf {#1}}
 \end{center}
 \vspace{0in}}

\def\black{
}

\def\red{
}

\def\blue{
}

\def\green{
}

\def\royalblue{
}
\def\brickred{
}
\def\yellow{
}
\def\brown{
}
\def\orange{
}
\def\indigo{
}
\def\violet{
}
\def\white{
}
\def\lightgrey{
}
\def\darkgrey{
}
\def\pink{
}

\newenvironment{tightlist}{
    \begin{itemize}
    \itemsep=0pt \parskip=0pt \parsep=0pt }{
    \end{itemize}
}

\newcommand{\boxedfigure}[2]{\vspace{1ex}
                             \fbox{\begin{minipage}[t]{#1} #2 \end{minipage}}
                             \vspace{1ex}}

\newenvironment{s_itemize-2}{\begin{list}{$\rhd$}
{\setlength{\rightmargin}{0em}
\setlength{\itemsep}{0em}
\setlength{\topsep}{0.25cm}
\setlength{\parsep}{0em}}}{\end{list}}

\newenvironment{itemize-1}{\begin{list}{$\bullet$}
{\setlength{\rightmargin}{0em}
\setlength{\itemsep}{0.5cm}
\setlength{\topsep}{0.25cm}
\setlength{\parsep}{0em}}}{\end{list}}

\newenvironment{itemize-2}{\begin{list}{$\rhd$}
{\setlength{\rightmargin}{0em}
\setlength{\itemsep}{0.25cm}
\setlength{\topsep}{0.25cm}
\setlength{\parsep}{0em}}}{\end{list}}

\def\topslidehead#1{\newpage\begin{center}\vspace*{-3cm}\begin{tabular*}{469pt}[t]{c}
\Huge{\parbox[c]{457pt}{\bf #1}} \\ \vspace{7pt} \\
\end{tabular*}\end{center}\yskip}

\newenvironment{bullist}
    {\begin{list}{$\bullet$}
        {\parsep 0pt \itemsep 0pt \setlength{\rightmargin}{\leftmargin}}}%
    {\end{list}}
\newenvironment{dashlist}
    {\begin{list}{$-$}
        {\parsep 0pt \itemsep 0pt \setlength{\rightmargin}{\leftmargin}}}%
    {\end{list}}

\newenvironment{itemize-3}{\begin{list}{-}
{\setlength{\rightmargin}{0em}
\setlength{\itemsep}{0.25cm}
\setlength{\topsep}{0.25cm}
\setlength{\parsep}{0em}}}{\end{list}}

\newenvironment{itemize-4}{\begin{list}{\ }
{\setlength{\rightmargin}{0em}
\setlength{\itemsep}{0.25cm}
\setlength{\topsep}{0.25cm}
\setlength{\parsep}{0em}}}{\end{list}}

\newenvironment{short-itemize}{\begin{list}{$\bullet$}
{\setlength{\rightmargin}{0em}
\setlength{\itemsep}{0.10cm}
\setlength{\topsep}{0.10cm}
\setlength{\parsep}{0em}}}{\end{list}}

\font\sf=cmss10
\newcommand{\Nats}{{\hbox{\sf I\kern-.13em\hbox{N}}}}   
\newcommand{\Reals}{{\hbox{\sf I\kern-.14em\hbox{R}}}}  
\newcommand{\Ints}{{\hbox{\sf Z\kern-.43emZ}}}          
\newcommand{\CC}{{\hbox{\sf C\kern -.48emC}}}           
\newcommand{\QQ}{{\hbox{\sf C\kern -.48emQ}}}           

\newcommand{\lNats}{{\hbox{\sf {\large I\kern-.13em\hbox{N}}}}}   
\newcommand{\lReals}{{\hbox{\sf {\large I\kern-.14em\hbox{R}}}}}  
\newcommand{\lInts}{{\hbox{\sf {\large Z\kern-.43emZ}}}}          
\newcommand{\lCC}{{\hbox{\sf {\large C\kern -.48emC}}}}           
\newcommand{\lQQ}{{\hbox{\sf {\large C\kern -.48emQ}}}}           

\section{Introduction}
\label{introduction}

\subsection{Background and Motivation}
\label{backmot} Kalman filtering with non-classical information pattern has received significant attention in the control and signal processing literature. There has been renewed interest, motivated by increasing real-time networked systems applications. Such networks operate under constrained resources with lack of supervised control centers leading to inherent sources of randomness in the information pattern. For reliable system operation, it is of interest to understand the asymptotic properties of such systems like stability and ergodicity. In~\cite{Riccati-weakconv}, we studied this problem in the context of Kalman filtering with intermittent observations (\cite{Bruno}.) The results in~\cite{Riccati-weakconv} establish an interesting dichotomy for the filtering error process and show, in particular, that stochastic boundedness of the sequence of conditional error covariance matrices (generated by the discrete time random Riccati equation (RRE)) is necessary and sufficient for its ergodicity. In other words, we showed the existence of a critical probability, $\overline{\gamma}^{\mbox{\scriptsize{sb}}}$, such that, if the observation packet arrival probability $\overline{\gamma}$ is greater than $\overline{\gamma}^{\mbox{\scriptsize{sb}}}$, the sequence of random error covariance matrices converges weakly (in distribution) to a unique invariant distribution $\mathbb{\mu}^{\overline{\gamma}}$. We note here that stochastic boundedness is a much weaker condition than moment stability, and, as shown in~\cite{Riccati-weakconv}, convergence to a unique invariant distribution is possible under a packet arrival probability for which moment stability does not hold. In this context, we further note, that our work (\cite{Riccati-weakconv}) provides a sample-path analysis of the RRE, in contrast to moment stability analysis, as is done conventionally in the literature (see, for example,~\cite{ckvs89,wg99,Liu:04,Gupta:05,Xu:05,Minyi:07,kp-fb:07j,Craig:07,xx07} and also~\cite{Riccati-weakconv} for a detailed review of the literature.)

To summarize, the results in~\cite{Riccati-weakconv} showed the existence and uniqueness of an attracting invariant measure $\mathbb{\mu}^{\overline{\gamma}}$ for the RRE, for every $\overline{\gamma}>\overline{\gamma}^{\mbox{\scriptsize{sb}}}$, to which the conditional error covariance matrices converge weakly when operated at packet arrival probability $\overline{\gamma}$. In this paper, we prove that, for observable and controllable systems, $\overline{\gamma}^{\mbox{\scriptsize{sb}}}=0$, and hence $\mathbb{\mu}^{\overline{\gamma}}$ exists and is unique for every $\overline{\gamma}>0$ for such systems\footnote{The fact, that $\overline{\gamma}^{\mbox{\scriptsize{sb}}}=0$, was proved for systems with invertible observation matrices in~\cite{Riccati-weakconv}. The proof for general observable systems is provided in Appendix~\ref{proof_stoch_bound_er} of the present paper.}.
 The main goal of this paper is to undertake the highly nontrivial problem of characterizing the resulting invariant measures $\mathbb{\mu}^{\overline{\gamma}}$. In the non-classical information case, characterization of the steady-state error covariance distribution $\mathbb{\mu}^{\overline{\gamma}}$ is as important as characterizing the deterministic fixed point of the Riccati equation in the classical case, as derived by Kalman~\cite{kalman1960} for discrete time and by Kalman and Bucy~\cite{kalmanbucy1961} for continuous time.

We detail the key contributions of this paper. We show the following for observable and controllable systems.
\begin{s_enumerate_noindent}
\item \mypar{Stochastic boundedness: $\overline{\gamma}^{\mbox{\scriptsize{sb}}}=0$} We prove the result stated in Theorem~9 in~\cite{Riccati-weakconv}, by proving  that, for controllable and observable systems, $\overline{\gamma}^{\mbox{\scriptsize{sb}}}=0$, and so, for any non-zero observation arrival probability $\overline{\gamma}$, the conditional error covariance process is ergodic,  with the unique attracting measure $\mathbb{\mu}^{\overline{\gamma}}$.
\item \mypar{Moderate deviation principle (MDP)}  We show that the family of invariant distributions $\left\{\mathbb{\mu}^{\overline{\gamma}}\right\}_{\overline{\gamma}>0}$ satisfies a moderate deviations principle (MDP) with good rate function $I$ as $\overline{\gamma}\uparrow 1$. An immediate consequence (which is rather intuitive but not natural) is that, as $\overline{\gamma}\uparrow 1$, the family of invariant measures $\left\{\mathbb{\mu}^{\overline{\gamma}}\right\}$ converges weakly to the Dirac mass $\delta_{P^{\ast}}$, where $P^{\ast}$ is the unique fixed point of the deterministic Riccati equation.
\item \mypar{Probability of rare events} The MDP implies that the probabilities of `rare events' (events bounded away from $P^{\ast}$) decay to zero as $\overline{\gamma}\uparrow 1$. A natural question of practical and theoretical interest is the rate at which the probability of such a rare event goes to zero. We show that the probability of such rare events decays as a power law of $(1-\overline{\gamma})$ as $\overline{\gamma}\uparrow 1$.
\item \mypar{Best decay: Variational problem} The best exponent for the power law decay of the probability of rare events depends on the rare event of interest; it can be explicitly characterized in terms of the rate function $I(\cdot)$. Formally, we have the following MDP asymptotics:
\begin{equation}
\label{intro1}
\mathbb{\mu}^{\overline{\gamma}}(\Gamma)\sim (1-\overline{\gamma})^{\inf_{x\in\Gamma}I(x)}
\end{equation}
(this notation is made precise in the paper.) Thus, the exact decay asymptotics of a rare event is obtained by solving a variational problem involving the rate function $I$. Since, the above MDP asymptotics holds for every Borel set $\Gamma$, our result characterizes completely the family of invariant measures $\left\{\mathbb{\mu}^{\overline{\gamma}}\right\}$.
\item \mypar{Estimating the probability of rare events} The estimation of probabilities of rare events reduces to solving deterministic variational problems; this not only characterizes the decay rate of rare events but also gives insight into how such events occur. In Section~\ref{comp}, we show several techniques that can be employed to solve these variational problems efficiently. We emphasize that our analysis of reducing the problem of estimating probabilities of interest to solving variational problems efficiently is much more definitive and relevant than numerically estimating the invariant distributions. A naive numerical approach of simulating the distributions $\left\{\mathbb{\mu}^{\overline{\gamma}}\right\}$ as $\overline{\gamma}\uparrow 1$ becomes meaningless as the rare events of interest become increasingly difficult to observe as $\overline{\gamma}\uparrow 1$ (see also Section~\ref{sub_scalar}.) One may take recourse to sophisticated simulation techniques like importance sampling (see, for example,~\cite{Dupuisimpsamp}), but such approaches require characterization of the distributions in question, which is addressed in this paper.
\end{s_enumerate_noindent}

More broadly, the techniques developed in this paper are fairly general and go beyond the setting of Kalman filtering with intermittent observations. There is a key difference in the MDP arguments used here and conventional methods for analyzing the moderate (or large) deviations for stationary measures of Markov processes, where it is generally assumed that the underlying Markov process is positive recurrent and moderate deviations of stationary measures then follow from that of finite dimensional distributions (see, for example,~\cite{DeuschelStroock},\cite{FreidlinWentzell}.) However, the Markov processes governing the RRE are not, in general, positive recurrent, as is the case for a large class of iterated function systems (\cite{DiaconisFreedman}.) Our analysis proceeds by studying the probability measures induced on the space of random function compositions (strings) and developing its topological properties, as detailed in the paper. Several intermediate results obtained in the process are of independent interest and follow under more general assumptions. Our tools are applicable to the analysis of more complex networked control systems (see, for example,~\cite{RDS-ACC-2010}) and hybrid or switched systems.

We summarize the organization of the paper. Subsection~\ref{notprel} presents notation and preliminaries on moderate deviations. Subsection~\ref{setup} sets up the problem and prior work is briefly reviewed in Subsection~\ref{prior_work}. Several key approximation results are presented in Section~\ref{approx12}, whereas the main results of this paper are stated and discussed in Section~\ref{main_res}. MDP for finite dimensional distributions of the RRE sequence is analyzed in Section~\ref{mod_finite}, whereas Section~\ref{inv_MDP} systematically carries out the steps required to obtain the main results on MDP for stationary distributions, which is completed in Section~\ref{proof_theorem}. In Section~\ref{comp}, we present efficient ways to solve the variational problems involving the rate function. Numerical studies justifying the theoretical results for a scalar system arepresented in Section~\ref{sub_scalar}. Section~\ref{conclusion} concludes the paper.


\subsection{Notation and Preliminaries}
\label{notprel} 
%

Denote by: $\mathbb{R}$, the reals; $\mathbb{R}^{M}$,
the $M$-dimensional Euclidean space; $\mathbb{T}$, the integers; $\mathbb{T}_{+}$, the non-negative integers; $\mathbb{N}$, the natural numbers; and $\mathcal{X}$, a generic space.
For a subset $B\subset \mathcal{X}$,  $\mathbb{I}_{B}:\mathcal{X}\longmapsto\{0,1\}$ is the indicator
function, which is~$1$ when the argument is in~$B$ and zero otherwise; and $\mbox{id}_{\mathcal{X}}$ is the identity function on
$\mathcal{X}$.
A metric space $\mathcal{X}$ with metric $d_{\mathcal{X}}$ is
denoted by the pair $(\mathcal{X},d_{\mathcal{X}})$. The corresponding Borel
algebra is denoted by $\mathcal{B}(\mathcal{X})$. For $x\in\mathcal{X}$, the open ball of radius $\varepsilon>0$ centered at $x$ is denoted by $B_{\varepsilon}(x)$, i.e., $B_{\varepsilon}(x)=\left\{y\in\mathcal{X}~|~d_{\mathcal{X}}(y,x)<\varepsilon\right\}$.
The closure of $B_{\varepsilon}(x)$ is the closed ball of radius $\varepsilon>0$ centered at $x$ and is denoted by $\overline{B}_{\varepsilon}(x)$.
For any set $\Gamma\subset\mathcal{X}$, the open $\varepsilon$-neighborhood of $\Gamma$ is given by
\begin{equation}
\label{prel101}
\Gamma_{\varepsilon}=\left\{y\in\mathcal{X}~|~\inf_{x\in\Gamma}d_{\mathcal{X}}(y,x)<\varepsilon\right\}
\end{equation}
It can be shown that $\Gamma_{\varepsilon}$ is an open set. Similarly, the closed $\varepsilon$-neighborhood of $\Gamma$ is given by
\begin{equation}
\label{prel102}
\overline{\Gamma_{\varepsilon}}=\left\{y\in\mathcal{X}~|~\inf_{x\in\Gamma}d_{\mathcal{X}}(y,x)\leq\varepsilon\right\}
\end{equation}
which is a closed set.
For a set $\Gamma\subset\mathcal{X}$, we denote by $\Gamma^{\circ}$ and $\overline{\Gamma}$ its interior and closure respectively.

\textbf{The Banach space of symmetric matrices}

Let $\mathbb{S}^{N}$ denote the separable Banach
space of symmetric $N\times N$ matrices,
equipped with the induced 2-norm. The subset $\mathbb{S}^{N}_{+}$
of positive semidefinite matrices is a closed, convex, solid,
normal, minihedral cone in $\mathbb{S}^{N}$, with non-empty
interior $\mathbb{S}^{N}_{++}$, the set of positive definite
matrices. The cone $\mathbb{S}^{N}_{+}$ induces a partial order in $\mathbb{S}^{N}$. For $X,Y\in\mathbb{S}^{N}$, we write $X\preceq Y$ ($Y\succeq X$) to denote $Y-X\in\mathbb{S}^{N}_{+}$; $X\prec Y$ to denote $X\preceq Y$ and $X\neq Y$; $X\ll Y$ ($Y\gg X$) to denote $Y-X\in\mathbb{S}^{N}_{++}$.


\textbf{Limit Notations}
Let $h:\mathbb{R}\longmapsto\mathbb{R}$ be a measurable function.

The notation $\lim_{z\rightarrow x}f(z)=y$ implies that for every sequence $\{z_{n}\}_{n\in\mathbb{N}}$ in $\mathbb{R}$ with $\lim_{n\rightarrow\infty}|z_{n}-x|=0$, we have $\lim_{n\rightarrow\infty}|f(z_{n})-y|=0$.
The notation $\lim_{z\uparrow x}f(z)=y$ implies that for every sequence $\{z_{n}\}_{n\in\mathbb{N}}$ in $\mathbb{R}$ with $z_{n}<x$ and $\lim_{n\rightarrow\infty}|z_{n}-x|=0$, we have $\lim_{n\rightarrow\infty}|f(z_{n})-y|=0$.
The notations $\downarrow$ and $\uparrow$ have similar implications when working with limit inferiors and superiors.

\textbf{Probability measures on metric spaces:} Let: $(\mathcal{X},d_{\mathcal{X}})$ a
complete separable metric space $\mathcal{X}$ with
metric~$d_{\mathcal{X}}$;  $\mathbb{B}(\mathcal{X})$ its Borel algebra;
 $B(\mathcal{X})$ the Banach space of real-valued bounded
functions on $\mathcal{X}$, equipped with the sup-norm, i.e.,
$f\in B(\mathcal{X}), \:\|f\|=\sup_{x\in\mathcal{X}}|f(x)|$; and $C_{b}(\mathcal{X})$  the subspace of $B(\mathcal{X})$ of continuous
functions. Let $\mathcal{P}(\mathcal{X})$ be the
set of probability measures on $\mathcal{X}$. For
$\mu\in\mathcal{M}(\mathcal{X})$, we define the support of $\mu$,
$\mbox{supp}(\mu)$, by
\begin{equation}
\label{def_support}
\mbox{supp}(\mu)=\left\{x\in\mathcal{X}\left|\right.\mu(B_{\varepsilon}(x))>0,
\:\:\forall\varepsilon>0\right\}
\end{equation}
It follows that $\mbox{supp}(\mu)$ is a closed set.
The sequence $\{\mu_{t}\}_{t\in\mathbb{T}_{+}}$ in
$\mathcal{P}(\mathcal{X})$ converges weakly to $\mu\in \mathcal{P}(\mathcal{X})$ if
\begin{equation}
\label{WC1a}
\lim_{t\rightarrow\infty}<f,\mu_{t}>\,=\,<f,\mu>,~~\forall~f\in
C_{b}(\mathcal{X})
\end{equation}
Weak convergence is denoted by $\mu_{t}\Longrightarrow\mu$ and is
also referred to as convergence in distribution. The weak topology
on $\mathcal{P}(\mathcal{X})$ generated by weak convergence can be metrized. In particular, e.g., \cite{Jacod-Shiryaev}, one has the
Prohorov metric $d_{p}$ on $\mathcal{P}(\mathcal{X})$, such that
the metric space $\left(\mathcal{P}(\mathcal{X}),d_{p}\right)$ is
complete, separable, and a sequence $\{\mu_{t}\}_{t\in\mathbb{T}_{+}}$ in $\mathcal{P}(\mathcal{X})$
converges weakly to $\mu$ in $\mathcal{P}(\mathcal{X})$ \emph{iff} $\lim_{t\rightarrow\infty}d_{p}(\mu_{t},\mu)=0$.
The distance between two probability measures $\mathbb{\mu}_{1},\mathbb{\mu}_{2}$ in $\mathcal{P}(\mathcal{X})$ is computed as:
\begin{equation}
\label{WC23}
d_{P}\left(\mathbb{\mu}_{1},\mathbb{\mu}_{2}\right)=\inf\left\{\varepsilon>0~|~\mathbb{\mu}_{1}(\mathcal{F})\leq \mathbb{\mu}_{2}(\mathcal{F}_{\varepsilon})+\varepsilon,~~~\forall~\mbox{closed set}~\mathcal{F}\right\}
\end{equation}

\textbf{Moderate Deviations:}
\begin{definition}
\label{moddev100}
Let $\{\mathbb{\mu}^{\overline{\gamma}}\}$ be a family of probability measures on the complete separable metric space $(\mathcal{X},d_{\mathcal{X}})$ indexed by the real-valued parameter $\overline{\gamma}$ taking values in $(0,1)$.
Let $h:(0,1)\longmapsto\mathbb{R}_{+}$ be a non-decreasing function on $(0,1)$ with
\begin{equation}
\label{moddev}
\lim_{\overline{\gamma}\uparrow 1}h(\overline{\gamma})=\infty
\end{equation}
Let $I:\mathcal{X}\longmapsto\overline{\mathbb{R}}_{+}$ be an extended valued lower semicontinuous function.
The family $\{\mathbb{\mu}^{\overline{\gamma}}\}$ is said to satisfy a moderate deviations principle (MDP) with rate function $I(\cdot)$ at scale $h(\overline{\gamma})$ as $\overline{\gamma}\uparrow 1$ if the following holds:
\begin{equation}
\label{moddev1}
\liminf_{\overline{\gamma}\uparrow
1}\frac{1}{h(\overline{\gamma})}\ln\mathbb{\mu}^{\overline{\gamma}}\left(\mathcal{O}\right)\geq -\inf_{x\in\mathcal{O}}I(x),~~~\mbox{for every open set $\mathcal{O}\in\mathcal{X}$}
\end{equation}
\begin{equation}
\label{moddev2}
\limsup_{\overline{\gamma}\uparrow
1}\frac{1}{h(\overline{\gamma})}\ln\mathbb{\mu}^{\overline{\gamma}}\left(\mathcal{F}\right)\leq -\inf_{x\in\mathcal{F}}I(x),~~~\mbox{for every closed set $\mathcal{F}\in\mathcal{X}$}
\end{equation}
\end{definition}
The function $I(\cdot)$ is called the MDP rate function. The lower semicontinuity implies that the level sets of $I(\cdot)$, i.e., sets of the form $\{x\in\mathcal{X}~|~I(x)\leq\alpha\}$ for every $\alpha\in\mathbb{R}_{+}$, are closed. If in addition, the levels sets are compact (for every $\alpha$), $I(\cdot)$ is said to be a good rate function, and the corresponding family  $\{\mathbb{\mu}^{\overline{\gamma}}\}$ is said to satisfy an MDP with good rate function $I(\cdot)$.

It can be shown that the MDP, as stated in~\eqref{moddev1}-\eqref{moddev2}, is equivalent to the following:
\begin{equation}
\label{moddev10000}
-\inf_{x\in\Gamma^{\circ}}I(x)\leq\liminf_{\overline{\gamma}\uparrow
1}\frac{1}{h(\overline{\gamma})}\ln\mathbb{\mu}^{\overline{\gamma}}\left(\Gamma\right)\leq\limsup_{\overline{\gamma}\uparrow
1}\frac{1}{h(\overline{\gamma})}\ln\mathbb{\mu}^{\overline{\gamma}}\left(\Gamma\right)\leq-\inf_{x\in\overline{\Gamma}}I(x)
\end{equation}
for every measurable set $\Gamma$. In other words,~\eqref{moddev10000} holds \emph{iff} \eqref{moddev1}-\eqref{moddev2} hold.

The above formulation of MDP is similar in spirit to the theory of large deviations principle (LDP). In fact, in the above definition, if the scale function $h(\cdot)$ is a polynomial in $\overline{\gamma}$, the family $\{\mathbb{\mu}^{\overline{\gamma}}\}$ is said to satisfy an LDP (see, for example,~\cite{DeuschelStroock,DemboZeitouni}.) This conceptual similarity is manifested in some of the proof techniques developed in the paper having parallels with their counterparts in the theory of LDP.

Before interpreting the consequences of an MDP as defined above, we consider the notion of a \emph{rare event}, which is the central motivation to all MDP (and LDP):
\begin{definition}[Rare Event]:
\label{moddev3} A set $\Gamma\subset\mathcal{B}(\mathcal{X})$ is called a rare event with respect to the family $\{\mathbb{\mu}^{\overline{\gamma}}\}$ of probability measures, if $\lim_{\overline{\gamma}\uparrow 1}\mathbb{\mu}^{\overline{\gamma}}(\Gamma)=0$.
In other words, the event $\Gamma$ becomes increasingly difficult to observe (i.e., it becomes rare) as $\overline{\gamma}\uparrow 1$.
\end{definition}
Once a rare event $\Gamma$ is identified, the next natural question is the rate at which its probability goes to zero under $\mathbb{\mu}^{\overline{\gamma}}$ as $\overline{\gamma}\uparrow 1$. This is answered by an MDP, which also gives a
complete characterization of the family as $\overline{\gamma}\uparrow 1$. Indeed, from Def.~\ref{moddev100} it is not hard to see that, if the family $\{\mathbb{\mu}^{\overline{\gamma}}\}$ satisfies an MDP, we have for every measurable set $\Gamma\in\mathcal{X}$:
\begin{equation}
\label{moddev5}
c_{1}(\overline{\gamma})e^{-h(\overline{\gamma})\inf_{x\in\Gamma^{\circ}}I(x)}\leq\mathbb{\mu}^{\overline{\gamma}}(\Gamma)\leq c_{2}(\overline{\gamma})e^{-h(\overline{\gamma})\inf_{x\in\overline{\Gamma}}I(x)}
\end{equation}
where $c_{1},c_{2}:(0,1)\longmapsto\mathbb{R}_{+}$ are functions, such that, $\lim_{\overline{\gamma}\uparrow 1}c_{i}(\overline{\gamma})=1,~~i=1,2$.
For brevity, we subsequently use the notation
\begin{equation}
\label{moddev7}
\mathbb{\mu}^{\overline{\gamma}}(\Gamma)\sim e^{-h(\overline{\gamma})\inf_{x\in\Gamma}I(x)}
\end{equation}
as a short form of (\ref{moddev5}).

Now assume $\Gamma$ is a rare event and, to avoid unnecessary technicalities, also assume that $\Gamma$ is an $I$-continuity set, i.e., $\inf_{x\in\Gamma^{\circ}}=\inf_{x\in\overline{\Gamma}}=\inf_{x\in\Gamma}$. Then, by (\ref{moddev5}) we must have $\inf_{x\in\Gamma}>0$ (so that the probabilities decay to zero.) Thus, (\ref{moddev5}) implies that the probability of the rare event $\Gamma$ decays exponentially at a scale $h(\overline{\gamma})$ to zero, $\inf_{x\in\Gamma}I(x)$ being best exponent (or rate) of decay. Such a characterization of the best decay rate of a rare event is extremely important in system analysis and, as will be seen in the paper, offers considerable insight into system design apart from providing a complete characterization of the measures $\mathbb{\mu}^{\overline{\gamma}}$.

\section{Problem Formulation}
\label{prob_form} We split the present section into two subsections, Subsection~\ref{setup} briefly summarizing the model of Kalman filtering with
intermittent observations, while in Subsection~\ref{prior_work} we review some results
from~\cite{Riccati-weakconv} on weak convergence of the random
error covariance matrices resulting from the above filtering model.

\subsection{Setup}
\label{setup}


We start by reviewing the model of Kalman filtering with intermittent observations in~\cite{Bruno}. Let
\begin{equation}
\label{sys_model} \mathbf{x}_{t+1}=A\mathbf{x}_{t}+\mathbf{w}_{t}
\end{equation}
\begin{equation}
\label{sys_model1} \mathbf{y}_{t}=C\mathbf{x}_{t}+\mathbf{v}_{t}
\end{equation}
Here $\mathbf{x}_{t}\in\mathbb{R}^{N}$ is the signal (state)
vector, $\mathbf{y}_{t}\in\mathbb{R}^{M}$ is the observation
vector, $\mathbf{w}_{t}\in\mathbb{R}^{N}$ and
$\mathbf{v}_{t}\in\mathbb{R}^{M}$ are Gaussian random vectors with
zero mean and covariance matrices $Q$ and $R$,
respectively. The sequences
$\{\mathbf{w}_{t}\}_{t\in\mathbb{T}_{+}}$ and
$\{\mathbf{v}_{t}\}_{t\in\mathbb{T}_{+}}$ are uncorrelated and
mutually independent. Also, assume that the initial state
$\mathbf{x}_{0}$ is a zero-mean Gaussian vector with covariance
$P_{0}$. Unless otherwise stated, we use the following
standing assumption throughout the paper:

\textbf{Assumption (E)}: The
pair $(C,A)$ is observable and $Q,R$
are positive definite. The assumption $Q\gg 0$ implies the controllability of the pair $(A,Q^{1/2})$. The main results of
the paper require all these assumptions, but several intermediate
results of independent interest hold under less stringent
assumptions. In that case, they are noted explicitly.

The m.m.s.e.~predictor $\widehat{\mathbf{x}}_{t|t-1}$ of the
signal vector $\mathbf{x}_{t}$ given the observations
$\{\mathbf{y}_{s}\}_{0\leq s< t}$ is the conditional mean. It is
recursively implemented by the Kalman filter.
 The sequence of conditional prediction error
covariances, $\left\{P_{t}\right\}_{t\in\mathbb{T}_{+}}$, is then
given by
{\small

\begin{eqnarray}
\label{sys_model3}
P_{t}&=&\mathbb{E}\left[\left(\mathbf{x}_{t}-\widehat{\mathbf{x}}_{t|t-1}\right)
\left(\mathbf{x}_{t}-\widehat{\mathbf{x}}_{t|t-1}\right)^{T}\left|\right.\{\mathbf{y}(s)\}_{0\leq
s<t}\right]
\\
\label{sys_model2}
P_{t+1}&=&AP_{t}A^{T}+Q-AP_{t}C^{T}\left(CP_{t}C^{T}+R\right)^{-1}CP_{t}A^{T}
\end{eqnarray}}
Under the hypothesis of controllability of the pair $(A,Q^{1/2})$ and
observability of the pair $(C,A)$, the deterministic sequence
$\left\{P_{t}\right\}_{t\in\mathbb{T}_{+}}$ converges to a unique
value $P^{\ast}$ (which is a fixed point of the algebraic Riccati
equation~(\ref{sys_model2})) from any initial condition $P_{0}$. \cite{kalman1960}.

This corresponds to the classical perfect observation scenario,
where the estimator has complete knowledge of the observation
packet $\mathbf{y}_{t}$ at every time $t$. With intermittent
observations, the observation packets are dropped randomly (across
the communication channel to the estimator), and the estimator
receives observations at random times. We study the intermittent
observation model considered in~\cite{Bruno}, where the channel
randomness is modeled by a sequence
$\left\{\gamma_{t}\right\}_{t\in\mathbb{T}_{+}}$ of
i.i.d.~Bernoulli random variables with mean $\overline{\gamma}$
(note, $\overline{\gamma}$ then denotes the arrival probability.)
Here, $\gamma_{t}=1$ corresponds to the arrival of the observation
packet $\mathbf{y}_{t}$ at time $t$ to the estimator, whereas a
packet dropout corresponds to $\gamma_{t}=0$. Denote by
$\widetilde{\mathbf{y}}_{t}$ the pair $\widetilde{\mathbf{y}}_{t}=\left(\mathbf{y}_{t}\mathbb{I}_{(\gamma_{t}=1)},\gamma_{t}\right)$.
Under the TCP packet acknowledgement protocol in~\cite{Bruno} (the
estimator knows at each time whether the observation packet
arrived or not), the m.m.s.e.~predictor of the signal is given by:
\begin{equation}
\label{sys_model6}
\widehat{\mathbf{x}}_{t|t-1}=\mathbb{E}\left[\mathbf{x}_{t}\left|
\left\{\widetilde{\mathbf{y}}_{s}\right\}_{0\leq s< t}\right.\right]
\end{equation}
A modified form of the Kalman filter giving a recursive
implementation of the estimator in (\ref{sys_model6}) is
in~\cite{Bruno}. The sequence of conditional prediction error
covariance matrices, $\left\{P_{t}\right\}_{t\in\mathbb{T}_{+}}$,
is updated according to the following random Riccati
equation (RRE):
{\small
\begin{eqnarray}
\label{sys_model8}
P_{t}&=&\mathbb{E}\left[\left(\mathbf{x}_{t}-\widehat{\mathbf{x}}_{t|t-1}\right)
\left(\mathbf{x}_{t}-\widehat{\mathbf{x}}_{t|t-1}\right)^{T}\left|\right.\left\{\widetilde{\mathbf{y}}(s)\right\}_{0\leq
s<t}\right]\\
\label{sys_model7}
P_{t+1}&=&AP_{t}A^{T}+Q-\gamma_{t}AP_{t}C^{T}\left(CP_{t}C^{T}+R\right)^{-1}CP_{t}A^{T}
\end{eqnarray}}
A convenient representation of $P_{t}$ is obtained by defining the functions $f_{0},f_{1}:\mathbb{S}_{+}^{N}\longmapsto\mathbb{S}_{+}^{N}$ as\footnote{$f_{0}$ corresponds to the Lyapunov operator, whereas $f_{1}$ is the Riccati operator.}:
\begin{equation}
\label{map_f0}
f_{0}(X)=AXA^{T}+Q,~~~\forall X\in\mathbb{S}_{+}^{N}
\end{equation}
\begin{equation}
\label{map_f1}
f_{1}(X)=AXA^{T}+Q-\gamma_{t}AXC^{T}\left(CXC^{T}+R\right)^{-1}CXA^{T},~~~\forall X\in\mathbb{S}_{+}^{N}
\end{equation}
We then have for all $t\geq 1$
\begin{equation}
\label{mapf0f1}
P_{t}=f_{\gamma_{t-1}}\circ f_{\gamma_{t-2}}\circ\cdots\circ f_{\gamma_{0}}(P_{0})
\end{equation}
Unlike the classical case, the sequence
$\left\{P_{t}\right\}_{t\in\mathbb{T}_{+}}$ is now random (because
of its dependence on the random sequence
$\left\{\gamma_{t}\right\}_{t\in\mathbb{T}_{+}}$.) Thus, for each
$t$, $P_{t}$ is a random element of $S^{N}_{+}$, and we denote by
$\mathbb{\mu}_{t}^{\overline{\gamma},P_{0}}$ its distribution (the
measure it induces on $S^{N}_{+}$.) The superscripts
$\overline{\gamma}$, $P_{0}$ emphasize the dependence of
$\mathbb{\mu}_{t}^{\overline{\gamma},P_{0}}$ on the packet arrival
probability and the initial condition. We often use the notations $\mathbb{P}^{\overline{\gamma},P_{0}},\mathbb{E}^{\overline{\gamma},P_{0}}$ to denote probability and expectation operators respectively, when the system is operated with observation arrival probability $\overline{\gamma}$ and initial covariance $P_{0}$.

\subsection{Prior work}
\label{prior_work}
The following extends Theorems~9, 10 in~\cite{Riccati-weakconv} on the weak convergence of the RRE sequence.
%

\begin{theorem}
\label{weakconv1}Let assumption~\textbf{(E.1)} hold. Then,
\begin{itemize}
\item
For each $\overline{\gamma}>0$,
there exists a unique invariant distribution
$\mathbb{\mu}^{\overline{\gamma}}$ s.t.~the sequence
$\left\{P_{t}\right\}_{t\in\mathbb{T}_{+}}$ (or  sequence
$\left\{\mathbb{\mu}_{t}^{\overline{\gamma},P_{0}}\right\}_{t\in\mathbb{T}_{+}}$
of measures) converges weakly to
$\mathbb{\mu}^{\overline{\gamma}}$ from any initial condition
$P_{0}\in\mathbb{S}_{+}^{N}$.\\


\item Define the set
$\mathcal{S}\subset\mathbb{S}^{N}_{+}$ by
\begin{equation}
\label{supp_inv1} \mathcal{S}=\left\{f_{i_{1}}\circ
f_{i_{2}}\circ\cdots\circ
f_{i_{s}}\left(P^{\ast}\right)\left|\right.i_{r}\in\{0,1\},\,1\leq
r\leq s,\:s\in\mathbb{T}_{+}\right\}
\end{equation}
Then\footnote{In the definition of $\mathcal{S}$ ((\ref{supp_inv1})), $s$ can take
the value 0, implying $P^{\ast}\in\mathcal{S}$.}, if $0<\overline{\gamma}<1$,
\begin{equation}
\label{supp_inv2}
\mbox{supp}\left(\mathbb{\mu}^{\overline{\gamma}}\right)=\mbox{cl}(\mathcal{S})
\end{equation}
where $\mbox{cl}(\mathcal{S})$ denotes the topological closure of
$\mathcal{S}$ in $\mathbb{S}^{N}_{+}$ and $\mbox{supp}$ denotes the support of a probability measure (\cite{Zaharopol}). In particular, we have
\begin{equation}
\label{main_th1} \mathbb{\mu}^{\overline{\gamma}}
\left(\left\{Y\in\mathbb{S}^{N}_{+}\left|\right.Y\succeq
P^{\ast}\right\}\right)=1
\end{equation}
\end{itemize}
\end{theorem}
 The proof is in Appendix~\ref{app:prooflemmacr_sb}. For a detailed discussion of the above results, the reader is referred to~\cite{Riccati-weakconv}.

\section{Some key approximation results}
\label{approx12}
In this section we present some results on random compositions of Lyapunov and Riccati operators leading to the RRE sequence (Subsection~\ref{notation}.) In Subsection~\ref{key_approx}, we present some key approximation results that are of independent interest and establish several useful properties of the RRE and the classical Riccati operator.


\subsection{Preliminary Results}
\label{notation}The RRE sequence is an iterated function system (see, for example,~\cite{DiaconisFreedman}) comprising of random compositions of Lyapunov and Riccati operators. To understand the system, we study the behavior of such random function compositions, where not only the numerical value of the composition is important, but also the composition pattern is relevant. To formalize this study, we start with the following definitions:

\begin{definition}[String]: Let $P_{0}\in\mathbb{S}_{+}$. A string
$\mathcal{R}$ with initial state $P_{0}$ and length $n\in\mathbb{T}_{+}$ is a $(n+1)$-tuple of the form:
\begin{equation}
\label{def_string} \mathcal{R}=\left(f_{i_{1}}, f_{i_{2}},\cdots
f_{i_{n}},P_{0}\right),~~~i_{1},\cdots,i_{n}\in\{0,1\}
\end{equation}
where $f_{0}$ and $f_{1}$ correspond to the Lyapunov and Riccati updates in eqns.~(\ref{map_f0},\ref{map_f1}). The length of a string $\mathcal{R}$ is denoted by $\mbox{len}(\mathcal{R})$.
The set of all possible strings is denoted by
$\overline{\mathcal{S}}$.
\end{definition}

\begin{remark} Note that a string $\mathcal{R}$ can be of length 0; then it is represented as a 1-tuple, consisting of only the initial condition. We introduce notation here. Let $t_{1},t_{2},\cdots,t_{l}$ be non-negative integers, such that, $\sum_{i=1}^{l}t_{i}=n$ and $i^{k}_{j}\in\{0,1\}$ for $1\leq j\leq t_{k},~~1\leq k\leq l$, such that, for all $k$, $i^{k}_{j}=i^{k}_{1},~1\leq j\leq t_{k}$. Let $\mathcal{R}$ be a string of length $n$ of the form:
\begin{equation}
\label{string_def_remark}
\mathcal{R}=\left(f_{i^{1}_{1}},\cdots,f_{i^{1}_{t_{1}}},\cdots,f_{i^{2}_{1}},\cdots,f_{i^{2}_{t_{2}}},\cdots,f_{i^{l}_{1}},\cdots,f_{i^{l}_{t_{l}}},P_{0}\right)
\end{equation}
where the indices $i^{k}_{j}$ satisfy the relations above. For brevity, we will write $\mathcal{R}$ as:
\begin{equation}
\label{string_def_remark1}
\mathcal{R}=\left(f_{i^{1}_{1}}^{t_{1}},f_{i^{2}_{1}}^{t_{2}},\cdots,f_{i^{l}_{1}}^{t_{l}},P_{0}\right)
\end{equation}
For example, the string $\left(f_{0},f_{1},f_{1},f_{1},f_{0},f_{0},P_{0}\right)$ is written concisely as $\left(f_{0},f_{1}^{3},f_{0}^{2},P_{0}\right)$.
\end{remark}

\begin{definition}[Numerical Value of a String]: To every string $\mathcal{R}$ is associated its numerical
value, denoted by $\mathcal{N}(\mathcal{R})$, which is the
numerical evaluation of the function composition on the initial
state $P_{0}$, i.e., for $\mathcal{R}$ of the form $\mathcal{R}=\left(f_{i_{1}}, f_{i_{2}},\cdots
f_{i_{n}},P_{0}\right),~~i_{1},\cdots,i_{n}\in\{0,1\}$, we have\footnote{For function compositions, we adopt a similar notation to that of strings, namely, for example, we denote the composition $f_{0}\circ f_{1}\circ f_{1}\circ f_{1}\circ f_{0}\circ f_{0}(P_{0})$ by $f_{0}\circ f_{1}^{3}\circ f_{0}^{2}(P_{0})$.}
\begin{equation}
\label{num_string100}
\mathcal{N}(\mathcal{R})=f_{i_{1}}\circ f_{i_{2}}\circ\cdots\circ f_{i_{n}}(P_{0})
\end{equation}
Thus, the numerical value can be viewed as a
function $\mathcal{N}(\cdot)$ from the space
$\overline{\mathcal{S}}$ of strings to $\mathbb{S}_{+}^{N}$. We
abuse notation by denoting $\mathcal{N}(\overline{\mathcal{S}})$
to be the set of numerical values attainable, i.e.,
\begin{equation}
\label{def_numval1}
\mathcal{N}(\overline{\mathcal{S}})=\left\{\mathcal{N}(\mathcal{R})~|~\mathcal{R}\in\overline{\mathcal{S}}\right\}
\end{equation}
\end{definition}

\begin{remark}Note the difference between a string and its
numerical value. Two strings are equal \emph{iff} they comprise of
the same \emph{order} of function compositions applied to the same initial state.
In particular, two strings can be different, even if they evaluate
to the same numerical value.
\end{remark}

\begin{definition}[Concatenated Strings]: Let $n\in\mathbb{N}$ and
$t_{1}\leq t_{2}\leq\cdots\leq t_{n}\in\mathbb{T}_{+}$. Also, if
$t_{n}\geq 1$, choose $i_{1},i_{2},\cdots,i_{t_{n}}\in\{0,1\}$.
Then
\begin{equation}
\label{concat_string1}
\overline{\mathcal{R}}=\left(\left(f_{i_{t_{1}}},f_{i_{t_{1}-1}}\cdots,
f_{i_{1}},P_{0}\right),\left(f_{i_{t_{2}}},\cdots,
f_{i_{t_{1}}},\cdots,f_{i_{1}},P_{0}\right),\cdots,
\left(f_{i_{t_{n}}},\cdots, f_{i_{1}},P_{0}\right)\right)
\end{equation}
is a concatenated string of block length $n$ with initial state
$P_{0}$\footnote{We again adopt the convention that the first
block is simply $P_{0}$ if $t_{1}=0$.}.

We similarly define the numerical value of such a concatenated
string $\overline{\mathcal{R}}$ by
\begin{equation}
\label{concat_string2}
\mathcal{N}(\overline{\mathcal{R}})=\left(f_{i_{t_{1}}}\circ\cdots\circ
f_{i_{1}}(P_{0}),f_{i_{t_{2}}}\circ\cdots\circ
f_{i_{t_{1}}}\circ\cdots f_{i_{1}}(P_{0}),\cdots,
f_{i_{t_{n}}}\circ\cdots\circ
f_{i_{1}}(P_{0})\right)
\end{equation}
and note that
$\mathcal{N}(\overline{\mathcal{R}})\in\bigotimes_{i=1}^{n}\mathbb{S}_{+}^{N}$.

For fixed $P_{0}$, $n$, $t_{1}\leq t_{2}\leq\cdots\leq t_{n}$, the
set of such concatenated strings is denoted by
$\mathcal{S}_{t_{1},\cdots,t_{n}}^{P_{0}}$. The corresponding set
of numerical values is denoted by
$\mathcal{N}(\mathcal{S}_{t_{1},\cdots,t_{n}}^{P_{0}})$.

Finally, for
$\mathbf{X}\in\bigotimes_{i=1}^{n}\mathbb{S}_{+}^{N}$, the set
$\mathcal{S}_{t_{1},\cdots,t_{n}}^{P_{0}}(\mathbf{X})\subset\mathcal{S}_{t_{1},\cdots,t_{n}}^{P_{0}}$
consists of all strings with numerical value $\mathbf{X}$, i.e.,
\begin{equation}
\label{concat_string3}
\mathcal{S}_{t_{1},\cdots,t_{n}}^{P_{0}}(\mathbf{X})=\left\{\overline{\mathcal{R}}\in\mathcal{S}_{t_{1},
\cdots,t_{n}}^{P_{0}}~|~\mathcal{N}\left(\overline{\mathcal{R}}\right)=\mathbf{X}\right\}
\end{equation}

\end{definition}

A rigorous algebra of such strings can be developed, which we
undertake elsewhere. In the following, we present some important
properties of strings to be used later (see Appendix~\ref{proof_approx12} for a proof):
\begin{proposition}
\label{string_prop}
\begin{itemize}
\item[(i)] For $s,t\in\mathbb{T}_{+}$ and $s\leq t$, we have, $\mathcal{N}\left(\mathcal{S}_{s}^{P^{\ast}}\right)\subset\mathcal{N}\left(\mathcal{S}_{t}^{P^{\ast}}\right)$.
In particular, if for some $X\in\mathbb{S}_{+}^{N}$,$t_{0}\in\mathbb{T}_{+}$ and $i_{1},\cdots, i_{t_{0}}\in\{0,1\}$, the string $\mathcal{R}=\left(f_{i_{1}},\cdots,f_{i_{t_{0}}},P^{\ast}\right)$ belongs to $\mathcal{S}_{t_{0}}^{P^{\ast}}(X)$, we have
\begin{equation}
\label{def_StP200}
\left(f_{i_{1}},\cdots,f_{i_{t_{0}}}, f_{1}^{t-t_{0}},P^{\ast}\right)\in\mathcal{S}_{t}^{P^{\ast}}(X)\subset\mathcal{S}^{P^{\ast}}(X),~~~\forall t\geq t_{0}
\end{equation}

\item[(ii)] Fix $n\in\mathbb{N}$ and $t_{1}<\cdots <
t_{n}\in\mathbb{T}_{+}$. We then have, for
$0\leq\overline{\gamma}\leq 1$,
\begin{equation}
\label{concatS1}
\mathbb{P}^{\overline{\gamma},P_{0}}\left(\left(P_{t_{1}},\cdots,P_{t_{n}}\right)\in\mathcal{N}\left(\mathcal{S}_{t_{1},\cdots,t_{n}}^{P_{0}}\right)\right)=1
\end{equation}

\item[(iii)] Let $t\in\mathbb{T}_{+}$ and
$\mathcal{R}\in\mathcal{S}_{t}^{P_{0}}=\left(f_{i_{1}},\cdots,
f_{i_{t}},P_{0}\right)$ be a string. Then, there exists
$\alpha_{P^{0}}\in\mathbb{R}_{+}$, depending on $P^{0}$, such
that,
\begin{equation}
\label{concatS4}
f_{0}^{\pi(\mathcal{R})}\left(\alpha_{P^{0}}I\right)\succeq\mathcal{N}\left(\mathcal{R}\right)
\end{equation}
where
\begin{equation}
\label{concatS5} \pi\left(\mathcal{R}\right)=\left\{
\begin{array}{ll}
                    \sum_{j=1}^{t}\mathbb{I}_{\{0\}}(i_{j}) & \mbox{if $t\geq 1$} \\
                    0 & \mbox{otherwise}
                   \end{array}
          \right.
\end{equation}
counts the number of $f_{0}$'s in $\mathcal{R}$.
\end{itemize}
\end{proposition}

\subsection{Some approximation results}
\label{key_approx} In this subsection we present several
approximation results to be used in the sequel. The results are of
independent interest and establish some useful properties of the
RRE and the classical Riccati operator.

The first concerns uniform convergence properties of the classical
Riccati operator and is used in the sequel to obtain various
tightness estimates required for establishing the MDP. The proof is
provided in Appendix~\ref{proof_approx12}.
\begin{lemma}
\label{unif_conv}For every $\varepsilon>0$, there exists
$t_{\varepsilon}\geq N$, such that, for every
$X\in\mathbb{S}_{+}^{N}$, with $X\succeq P^{\ast}$,
\begin{equation}
\label{unif_conv1}
\left\|f_{1}^{t}\left(X\right)-P^{\ast}\right\|\leq\varepsilon,~~~t\geq
t_{\varepsilon}
\end{equation}
Note, in particular, that $t_{\varepsilon}$ can be chosen
independently of the initial state $X$.
\end{lemma}

The following result can be viewed as a corollary to
Lemma~\ref{unif_conv} and concerns the Lipschitz continuity of
finite compositions of the Riccati operator.
\begin{lemma}
\label{Lip} For fixed $t\in\mathbb{N}$ and
$i_{1},\cdots,i_{t}\in\{0,1\}$, define the function
$g:\mathbb{S}_{+}^{N}\longmapsto\mathbb{S}_{+}^{N}$ by
\begin{equation}
\label{Lip1} g(X)=f_{i_{1}}\circ\cdots\circ
f_{i_{t}}(X),~~~X\in\mathbb{S}_{+}^{N}
\end{equation}
Then $g(\cdot)$ is Lipschitz continuous with some constant
$K_{g}>0$.

Also, for every $\varepsilon_{2}>0$, there exists
$t_{\varepsilon_{2}}$, such that, the function
$f_{1}^{t_{\varepsilon_{2}}}(\cdot)$ is Lipschitz continuous with
constant $K_{f_{1}^{t_{\varepsilon_{2}}}}<\varepsilon_{2}$.
\end{lemma}
\begin{proof}
From (\ref{unif_conv3}) it follows that the function
$f_{1}(\cdot)$ is Lipschitz continuous with constant
$K_{f_{1}}=c_{1}e^{-c_{2}}$, where $c_{1},c_{2}$ are positive
constants defined in Lemma~\ref{unif_conv}. It is also easy to see
that the affine function $f_{0}(\cdot)$ is Lipschitz continuous
with constant $K_{f_{0}}=\alpha^{2}$, where $\alpha$ is the
largest singular value of the matrix $A$.

It then follows that the function $g(\cdot)$ defined above is
Lipschitz continuous, being a finite composition of Lipschitz
continuous functions.

For the second assertion, choose $t_{\varepsilon_{2}}\in\mathbb{N}$,
such that, $c_{1}e^{-c_{2}t_{\varepsilon_{2}}}<\varepsilon_{2}$,
where $c_{1},c_{2}>0$ are defined in Lemma~\ref{unif_conv}, equation~(\ref{unif_conv3}). It then follows from
(\ref{unif_conv3}) that, with the above choice of
$t_{\varepsilon_{2}}$, the function
$f_{1}^{t_{\varepsilon_{2}}}(\cdot)$ is Lipschitz continuous with
constant $K_{f_{1}}^{t_{\varepsilon_{2}}}<\varepsilon_{2}$.
\end{proof}

The following result concerns stochastic boundedness of the random
sequence $\{P_{t}\}_{t\in\mathbb{T}_{+}}$ generated by the RRE. In particular, it completes the
proof of Proposition 8 in~\cite{Riccati-weakconv} by establishing triviality of $\overline{\gamma}^{\sbb}$ for general observable and controllable systems (in~\cite{Riccati-weakconv}, Proposition 8 was proved only for systems with invertible $C$.)
\begin{lemma}
\label{cr_sb} Assume $\left(A,Q^{1/2}\right)$ controllable,
$(C,A)$ observable. Then $\overline{\gamma}^{\sbb}=0$, i.e., the sequence $\{P_{t}\}$ is stochastically bounded
\begin{equation}
\label{ref1000}
\lim_{M\rightarrow\infty}\sup_{t\in\mathbb{T}_{+}}\mathbb{P}^{\overline{\gamma},P_{0}}\left(\left\|P_{t}\right\|>M\right)
= 0
\end{equation}
for every $\overline{\gamma}>0$ and initial covariance state $P_{0}\in\mathbb{S}_{+}^{N}$.
\end{lemma}

The proof is provided in Appendix~\ref{proof_stoch_bound_er} and offers a new insight into the random Riccati equation.
For a discussion on the consequence and significance of the result, the reader is referred to the text following Proposition 8 in~\cite{Riccati-weakconv}. We reemphasize here, that the result
establishes the importance of stochastic boundedness as a
metric for system stability and design as compared to various
notions of moment stability. For example, as shown
in~\cite{Bruno}, the critical probability
for mean stability may be quite large,
depending on the instability of $A$. However, we show that, even in the sub-mean stability regime, the system is
stochastically bounded (for $\overline{\gamma}>0$) and converges to a
unique invariant distribution. Hence, our analysis offers insight into
system design in the sub-mean stability regime.

The following result on limits of real number sequences will be
useful later.
\begin{proposition}
\label{lim_num} For $J\in\mathbb{N}$ and $1\leq i\leq J$, let
$a_{i}:[0,1)\longmapsto [0,1]$ be functions with
\begin{equation}
\label{lim_num1} \lim_{\overline{\gamma}\uparrow 1} -\frac{\ln
a_{i}(\overline{\gamma})}{\ln(1-\overline{\gamma})}=-a_{i}^{\ast},~~~~1\leq
i\leq J
\end{equation}
(we adopt the convention $\ln 0=-\infty$ and $a_{i}^{\ast}$ is
non-negative with $\infty$ as a possible value.  Then
\begin{equation}
\label{lim_num2} \lim_{\overline{\gamma}\uparrow 1}
-\frac{\ln\left(\sum_{i=1}^{J}a_{i}(\overline{\gamma})\right)}{\ln(1-\overline{\gamma})}=-\min_{i\in\{0,\cdots,J\}}\left\{a_{i}^{\ast}\right\}
\end{equation}
\end{proposition}

\section{Main Results and Discussions}
\label{main_res}

We state the main results of the paper in this section whose proofs are provided in Section~\ref{proof_theorem}.

The following result is a first step to understanding the behavior of the family $\{\mathbb{\mu}^{\overline{\gamma}}\}$ of invariant distributions.
\begin{theorem}
\label{thm:convrate}
The family of invariant distributions $\left\{\mathbb{\mu}^{\overline{\gamma}}\right\}$ converges weakly to the Dirac probability measure $\delta_{P^{\ast}}$ as $\overline{\gamma}\uparrow 1$, i.e.,
\begin{equation}
\label{thm:convrate1}
\lim_{\overline{\gamma}\uparrow 1}d_{P}\left(\mathbb{\mu}^{\overline{\gamma}},\delta_{P^{\ast}}\right)=0
\end{equation}
We have the following convergence rate asymptotics:

For every $\varepsilon>0$, we have
\begin{equation}
\label{thm:convrate4}
\limsup_{\overline{\gamma}\uparrow 1}-\frac{\ln\left(\mathbb{\mu}^{\overline{\gamma}}\left(B^{C}_{\varepsilon}(P^{\ast})\right)\right)}{\ln(1-\overline{\gamma})}\leq -1
\end{equation}
\end{theorem}

We discuss the consequences of Theorem~\ref{thm:convrate}. The first assertion states that the family $\{\mathbb{\mu}^{\overline{\gamma}}\}$ converges weakly to the Dirac measure concentrated at $P^{\ast}$, $\delta_{P^{\ast}}$, as $\overline{\gamma}\uparrow 1$. This is quite intuitive, as with $\overline{\gamma}\uparrow 1$, the filtering problem reduces to the classical Kalman filtering setup with deterministic packet arrival (no dropouts), in which case the deterministic sequence of conditional error covariances converges to the fixed point $P^{\ast}$. Thus, as $\overline{\gamma}\uparrow 1$, we expect the RRE sequence to behave more and more similarly to the deterministic $\overline{\gamma}=1$ case leading to the convergence of the measures $\mathbb{\mu}^{\overline{\gamma}}$ to $\delta_{P^{\ast}}$ as $\overline{\gamma}\uparrow 1$. However, from a technical point of view this is not obvious, as the case $\overline{\gamma}=1$ may be a singularity. Theorem~\ref{thm:convrate} rules out this possibility and shows that the family $\{\mathbb{\mu}^{\overline{\gamma}}\}$ viewed as a function of $\overline{\gamma}$ is sufficiently well behaved (continuous) at $\overline{\gamma}=1$.

An immediate consequence of Theorem~\ref{thm:convrate} is the following:
\begin{equation}
\label{mainres1}
\lim_{\overline{\gamma}\uparrow 1}\mathbb{\mu}^{\overline{\gamma}}(\Gamma)=0,~~~\forall~\overline{\Gamma}\cap P^{\ast}=\phi
\end{equation}
Thus, we note that, w.r.t. $\{\mathbb{\mu}^{\overline{\gamma}}\}$, every event $\Gamma$ with $P^{\ast}\notin\overline{\Gamma}$ is a rare event (see Defn.~(\ref{moddev3}).) This is intuitively clear, because as $\overline{\gamma}\uparrow 1$, the measures $\mathbb{\mu}^{\overline{\gamma}}$ become concentrated on arbitrarily small neighborhoods of $P^{\ast}$, making such an event $\Gamma$ very difficult to observe.

Once the rare events are identified, the next step in the characterization of $\{\mathbb{\mu}^{\overline{\gamma}}\}$ is to ascertain the rate at which such rare events go to zero, or the rate at which the family $\{\mathbb{\mu}^{\overline{\gamma}}\}$ converges to $\delta_{P^{\ast}}$ as $\overline{\gamma}\uparrow 1$. The distance between the family $\{\mathbb{\mu}^{\overline{\gamma}}\}$ and its weak limit $\delta_{P^{\ast}}$ as $\overline{\gamma}\uparrow 1$ is important for the design engineer, as it relates the loss in performance when operating at packet arrival probability $\overline{\gamma}<1$. The answer is provided in the second assertion of Theorem~\ref{thm:convrate}, which says that, for every $\varepsilon>0$, the probability of staying away from the $\varepsilon$-neighborhood of $P^{\ast}$ decreases as $(1-\overline{\gamma})$ when $\overline{\gamma}\uparrow 1$, i.e.\footnote{For functions $h(\cdot),g(\cdot)$, the notation $h(\overline{\gamma})=O(g(\overline{\gamma}))$ implies the existence of a constant $c>0$, such that $h(\overline{\gamma})\leq cg(\overline{\gamma})$ for all $\overline{\gamma}\in (0,1)$.},
\begin{equation}
\label{main_res5}
\mathbb{\mu}^{\overline{\gamma}}=O(1-\overline{\gamma}),~~\forall~\varepsilon>0
\end{equation}
Since the above holds for every $\varepsilon>0$, the probability of any rare event vanishes at least as $(1-\overline{\gamma})$. Clearly, the exact rate of going to zero depends on the rare event in question (for example, how `far' it is from $P^{\ast}$, which becomes `typical' as $\overline{\gamma}\uparrow 1$.

A complete characterization of the family $\{\mathbb{\mu}^{\overline{\gamma}}\}$ requires the exact decay rate of rare events as $\overline{\gamma}\uparrow 1$, and this is achieved in the following result that establishes an MDP for the family $\{\mathbb{\mu}^{\overline{\gamma}}\}$ as $\overline{\gamma}\uparrow 1$.

\begin{theorem}
\label{thm:mdp}
Recall in equation~(\ref{concatS5}) the definition of~$\pi(\cdot)$. The family of invariant distributions $\left\{\mathbb{\mu}^{\overline{\gamma}}\right\}$ satisfies an MDP at scale $-\ln(1-\overline{\gamma})$ as $\overline{\gamma}\uparrow 1$ with a good rate function $I(\cdot)$, i.e.,
\begin{equation}
\label{thm:mdp1}
\liminf_{\overline{\gamma}\uparrow 1}-\frac{1}{\ln(1-\overline{\gamma})}\ln\mathbb{\mu}^{\overline{\gamma}}\left(\mathcal{O}\right)\geq -\inf_{X\in\mathcal{O}}I(X),~~~\mbox{for every open set $\mathcal{O}$}
\end{equation}
\begin{equation}
\label{thm:mdp2}
\limsup_{\overline{\gamma}\uparrow 1}-\frac{1}{\ln(1-\overline{\gamma})}\ln\mathbb{\mu}^{\overline{\gamma}}\left(\mathcal{F}\right)\leq -\inf_{X\in\mathcal{F}}I(X),~~~\mbox{for every closed set $\mathcal{F}$}
\end{equation}
where the function $I:\mathbb{S}_{+}^{N}\longmapsto\overline{\mathbb{R}}_{+}$ is given by:
\begin{equation}
\label{thm:mdp100}
I(X)=\inf_{\mathcal{R}\in\mathcal{S}^{P^{\ast}}}\pi(\mathcal{R}),~~~\forall X\in\mathbb{S}_{+}^{N}
\end{equation}
\end{theorem}

Theorem~\ref{thm:mdp} provides a complete understanding of the family of invariant measures $\{\mathbb{\mu}^{\overline{\gamma}}\}$ as $\overline{\gamma}\uparrow 1$. First, it establishes the important qualitative behavior of $\{\mathbb{\mu}^{\overline{\gamma}}\}$, namely, that rare events decay \emph{exactly} as power-laws of $(1-\overline{\gamma})$ as $\overline{\gamma}\uparrow 1$. Also the exact exponent of such a power law decay depends on the particular rare event and is obtained as the solution of an associated variational problem involving the minimization of the rate function $I(\cdot)$. This is relevant for a system designer who can trade-off estimation accuracy with communication required. For example, given a tolerance $M>0$, we may ask the question, at what operating $\overline{\gamma}$ is the probability of lying outside the $M$-neighborhood centered at $P^{\ast}$ less than some $\delta>0$. Using the notation of (\ref{moddev7}), we then have
\begin{equation}
\label{main_res7}
\mathbb{\mu}^{\overline{\gamma}}\left(B_{M}^{C}(P^{\ast})\right)\sim \left(1-\overline{\gamma}\right)^{\inf_{X\in B_{M}^{C}(P^{\ast})}I(X)}
\end{equation}
Thus by computing $\inf_{X\in B_{M}^{C}(P^{\ast})}I(X)$, the designer obtains an estimate of the $\overline{\gamma}$ required to maintain a probability of error less than $\delta$. Thus the estimation of probabilities of rare events reduces to solving deterministic variational problems. As shown in Section~\ref{comp}, several techniques can be employed to solve these variational problems efficiently. Finally, we emphasize that our analysis of reducing the problem of estimating probabilities of interest to solving variational problems efficiently is much more definitive and relevant than numerically estimating the invariant distributions. A naive numerical approach of simulating the distributions $\{\mathbb{\mu}^{\overline{\gamma}}\}$ as $\overline{\gamma}\uparrow 1$ becomes meaningless here as the rare events, which are of interest, become increasingly difficult to simulate as $\overline{\gamma}\uparrow 1$. One may take recourse to sophisticated simulation techniques like importance sampling (see, for example,~\cite{Dupuisimpsamp}), but such approaches require characterization of the distributions in question, which is addressed in this paper.

\section{Moderate Deviations (MD) for finite-dimensional distributions}
\label{mod_finite} In this section, we establish MD for finite
dimensional distributions of the process
$\{P_{t}\}_{t\in\mathbb{T}_{+}}$ as $\overline{\gamma}\uparrow 1$. We start by setting notation.
%
 Fix $n\in\mathbb{N}$ and $t_{1}<\cdots < t_{n}\in\mathbb{T}_{+}$
and recall the sets $\mathcal{S}_{t_{1},\cdots,t_{n}}^{P_{0}}$ and
$\mathcal{S}_{t_{1},\cdots,t_{n}}(\mathbf{X})$ for
$\mathbf{X}\in\bigotimes_{i=1}^{n}\mathbb{S}_{+}^{N}$.
 As noted in Proposition~\ref{string_prop}, the random object
$\left(P_{t_{1}},\cdots,P_{t_{n}}\right)$ takes only a finite
number of values and for all $\overline{\gamma}$
\begin{equation}
\label{mod_finite10}
\mathbb{P}^{\overline{\gamma},P_{0}}\left(\left(P_{t_{1}},\cdots,P_{t_{n}}\right)\in\mathcal{N}\left(\mathcal{S}_{t_{1},\cdots,t_{n}}^{P_{0}}\right)\right)=1
\end{equation}

We generalize the definition of the functional
$\pi:\mathcal{S}^{P_{0}}_{t}\longmapsto\mathbb{Z}_{+}$
(eqn.~\ref{concatS5}) to strings in
$\mathcal{S}^{P_{0}}_{t_{1},\cdots, t_{n}}$ by
\begin{equation}
\label{mod_finite11}
\pi\left(\overline{\mathcal{R}}\right)=\left\{
\begin{array}{ll}
                    \sum_{j=1}^{t_{n}}\mathbb{I}_{\{0\}}(i_{j}) & \mbox{if $t\geq 1$} \\
                    0 & \mbox{otherwise}
                   \end{array}
          \right.
\end{equation}
Thus $\pi(\cdot)$ counts the number of $f_{0}$'s in the string
$\overline{\mathcal{R}}$.
%
%
 Also, for $\mathbf{X}\in\bigotimes_{i=1}^{n}\mathbb{S}_{+}^{N}$,
define
\begin{equation}
\label{mod_finite12}
\ell(\mathbf{X})=\min_{\overline{\mathcal{R}}\in\mathcal{S}_{t_{1},\cdots,t_{n}}^{P_{0}}(\mathbf{X})}\pi(\overline{\mathcal{R}})
\end{equation}
(We adopt the convention that the minimum of an empty set is
$\infty$.)

The following result shows that the function
$\ell:\bigotimes_{i=1}^{n}\mathbb{S}_{+}^{N}\longmapsto\mathbb{R}_{+}$
is a good rate function on
$\bigotimes_{i=1}^{n}\mathbb{S}_{+}^{N}$.
\begin{proposition}
\label{mod_fin_good} The function
$\ell:\bigotimes_{i=1}^{n}\mathbb{S}_{+}^{N}\longmapsto\mathbb{R}_{+}$
in (\ref{mod_finite12}) is a good rate function on
$\bigotimes_{i=1}^{n}\mathbb{S}_{+}^{N}$.
\end{proposition}
\begin{proof} Clearly,  $\ell(\cdot)>0$.
Its level sets are compact because its effective domain $\mathcal{D}_{\ell}$ is finite, where
\begin{equation}
\label{mod_fin_good1}
\mathcal{D}_{\ell}=\left\{\mathbf{X}\in\bigotimes_{i=1}^{n}\mathbb{S}_{+}^{N}~|~\ell(\mathbf{X})<\infty\right\}
\end{equation}
\end{proof}

We then have the following result giving the MD for the family
(as $\overline{\gamma}\uparrow 1$) of finite dimensional
distributions $(P_{t_{1}},\cdots,P_{t_{n}})$. The proof is provided in Appendix~\ref{proof_thm:fin}.
\begin{theorem}
\label{thm:fin} Fix $n\in\mathbb{N}$ and $t_{1}<\cdots <
t_{n}\in\mathbb{T}_{+}$ and let
$(P_{t_{1}},\cdots,P_{t_{n}})^{\overline{\gamma}}$ be the family
of of finite dimensional distributions indexed by
$\overline{\gamma}$, starting from the same initial state $P_{0}$.
Then for every
$B\in\mathcal{B}\left(\bigotimes_{i=1}^{n}\mathbb{S}_{+}^{N}\right)$,
we have
\begin{equation}
\label{thm:fin1} \lim_{\overline{\gamma}\uparrow
1}\frac{1}{\ln\left(1-\overline{\gamma}\right)}\ln\left(\mathbb{P}^{\overline{\gamma},P_{0}}\left((P_{t_{1}},\cdots,P_{t_{n}})^{\overline{\gamma}}\in
B\right)\right)=\inf_{\mathbf{X}\in
B}I^{t_{1},\cdots,t_{n}}_{P_{0}}\left(\mathbf{X}\right)
\end{equation}
where
\begin{equation}
\label{thm:fin2}
I^{t_{1},\cdots,t_{n}}_{P_{0}}\left(\mathbf{X}\right)=\left\{
\begin{array}{ll}
                    \ell(\mathbf{X}) & \mbox{if $\mathbf{X}\in\mathcal{S}_{t_{1},\cdots,t_{n}}^{P_{0}}$} \\
                    \infty & \mbox{otherwise}
                   \end{array}
          \right.
\end{equation}
\end{theorem}

\begin{remark}
\label{rem_mod_fin} The MDP for finite dimensional distributions
gives insight into transient system behavior. More importantly,
it offers considerable insight in guessing the proper rate
function governing the MDP for the family of invariant
distributions. A naive approach to the MDP rate
function for the family of invariant distributions is to view it as a
suitable `limit' of fixed time rate functions $I^{t}(\cdot)$,
provided the latter converges in an appropriate sense to a limit
$I(\cdot)$ as $t\rightarrow\infty$. The intuition being the weak
convergence of the random sequence
$\{P_{t}\}_{t\in\mathbb{T}_{+}}$ as $t\rightarrow\infty$. However,
in general, the limit of fixed time rate functions may not be the
rate function governing the MDP for the family of invariant rate
functions. One needs to verify rigorously that the guessed rate
function obtained intuitively is the actual rate function
governing the required MDP. This is precisely the way (at least
implicitly) we establish the MDP rate function of the family of
invariant distributions in Section~\ref{inv_MDP}.

\end{remark}


\section{MDP for invariant distributions}
\label{inv_MDP} This section constitutes the key technical part of the paper. Apart from establishing the main ingredients for proving the MDP results in Section~\ref{main_res}, the results are of independent interest and form a basis for understanding the characteristics of stationary measures resulting from \emph{stable} iterated function systems in general. As suggested in
Remark~\ref{rem_mod_fin}, an intuitive guess for the MDP rate function of the family $\{\mathbb{\mu}^{\overline{\gamma}}\}$ as
$\overline{\gamma}\uparrow 1$ is $I(\cdot)$. However, apriori, it is not even obvious whether $I(\cdot)$ is lower semicontinuous to qualify as a rate function. Hence, we start by defining the lower semicontinuous regularization $I_{L}$ of $I$ and establish some of its properties in Subsection~\ref{inv_MDP_rate}. At this point, it is not clear whether $I=I_{L}$ (i.e., $I$ is lower semicontinuous) and, hence, we set to establish the MDP for the family $\{\mathbb{\mu}^{\overline{\gamma}}\}$ with $I_{L}$ as a candidate rate function. The major technical lemmas of this section are presented in Subsections~\ref{inv_MDP_lower},\ref{inv_MDP_upper}, where we establish the MDP lower and upper bounds respectively w.r.t. the proposed rate function $I_{L}$. Our approach is fairly general and is not necessarily restricted to the particular filtering problem considered here.

%

\subsection{A rate function}
\label{inv_MDP_rate}Recall: $I:\mathbb{S}_{+}^{N}\longmapsto\overline{\mathbb{R}}_{+}$ by
\begin{equation}
\label{inv_MDP_rate1}
I(X)=\inf_{\mathcal{R}\in\mathcal{S}^{P^{\ast}}}\pi(\mathcal{R}),~~~\forall X\in\mathbb{S}_{+}^{N}
\end{equation}
where, as usual, we adopt the convention that the infimum of an empty set is $\infty$. Note that, for $X\in\mathbb{S}_{+}^{N}$,
\begin{equation}
\label{inv_MDP_rate2}
I(X)=\inf_{t\in\mathbb{T}_{+}}I_{t}^{P^{\ast}}(X)
\end{equation}
and can be thought of as a natural generalization of the marginal rate functions $I_{t}^{P^{\ast}}(\cdot)$ for all $t$.

However, the function $I(\cdot)$ is not generally lower semicontinuous (as will be seen later) and hence does not qualify as a rate function. A candidate rate function for the family of invariant distributions can be the lower semicontinuous regularization of $I(\cdot)$, defined as
\begin{equation}
\label{inv_MDP_rate3}
I_{L}(X)=\lim_{\varepsilon\rightarrow 0}\inf_{Y\in B_{\varepsilon}(X)}I(Y),~~~~\forall X\in\mathbb{S}_{+}^{N}
\end{equation}
(note if $I(\cdot)$ is lower semicontinuous then $I(\cdot)=I_{L}(\cdot)$.)

The following proposition states some easily verifiable properties of $I_{L}:\mathbb{S}_{+}^{N}\longmapsto\overline{\mathbb{R}}_{+}$, whose proof is provided in Appendix~\ref{proof_inv_MDP}.
\begin{proposition}
\label{inv_MDP_rate_prop}
\begin{itemize}
\item [(i)] The function $I_{L}(\cdot)$ is a good rate function on $\mathbb{S}_{+}^{N}$.

\item [(ii)] For every $X\in\mathbb{S}_{+}^{N}$, we have
\begin{equation}
\label{inv_MDP_rate_prop1}
I_{L}(X)=\lim_{\varepsilon\rightarrow 0}\inf_{Y\in \overline{B_{\varepsilon}(X)}}I(Y)
\end{equation}

\item [(iii)] For any non-empty set $\Gamma\in\mathcal{B}(\mathbb{S}_{+}^{N})$ we have
\begin{equation}
\label{inv_MDP_rate_prop200}
\inf_{X\in\Gamma}I_{L}(X)\leq\inf_{X\in\Gamma}I(X)
\end{equation}
In addition, if $\Gamma$ is open, the reverse inequality holds and we have
\begin{equation}
\label{inv_MDP_rate_prop201}
\inf_{X\in\Gamma}I_{L}(X)=\inf_{X\in\Gamma}I(X)
\end{equation}

\item [(iv)] Let $K\subset\mathbb{S}_{+}^{N}$ be a non-empty compact set. We have
\begin{equation}
\label{add_1}
\lim_{\varepsilon\rightarrow 0}\inf_{Y\in\overline{K_{\varepsilon}}}I_{L}(Y)=\inf_{Y\in K}I_{L}(Y)
\end{equation}

\end{itemize}
\end{proposition}

\subsection{The MDP lower bound}
\label{inv_MDP_lower} The following result establishes the MDP
lower bound for the sequence
$\{\mathbb{\mu}^{\overline{\gamma}}\}$ of invariant distributions
as $\overline{\gamma}\uparrow 1$.

\begin{lemma}
\label{inv_MDP_lower_lemma} Let
$\Gamma\in\mathcal{B}\left(\mathbb{S}_{+}^{N}\right)$. Then the
following lower bound holds:
\begin{equation}
\label{lower1} \liminf_{\overline{\gamma}\uparrow 1}-\frac{1}{\ln(1-\overline{\gamma})}\ln\mathbb{\mu}^{\overline{\gamma}}\left(\Gamma\right)\geq-\inf_{X\in
\Gamma^{\circ}}I_{L}(X)
\end{equation}
\end{lemma}

\begin{proof}
Let $P_{0}$ be an arbitrary initial state and
$\left\{P^{\overline{\gamma},P_{0}}_{t}\right\}_{t\in\mathbb{T}_{+}}$
be the sequence generated by the RRE for $\overline{\gamma}\in
(0,1)$. It was shown in Theorem 9 in~\cite{Riccati-weakconv}, that,
for such $\overline{\gamma}$, the sequence
$\left\{P^{\overline{\gamma},P_{0}}_{t}\right\}_{t\in\mathbb{T}_{+}}$
converges weakly to an invariant distribution
$\mathbb{\mu}^{\overline{\gamma}}$, i.e.,
\begin{equation}
\label{lower2}
\liminf_{t\rightarrow\infty}\mathbb{P}^{\overline{\gamma},P_{0}}\left(P_{t}^{\overline{\gamma},P_{0}}\in\mathcal{O}\right)\geq\mathbb{\mu}^{\overline{\gamma}}\left(\mathcal{O}\right),~~~\forall~\mbox{open
set $\mathcal{O}\subset\mathbb{S}_{+}^{N}$}
\end{equation}
\begin{equation}
\label{lower3}
\limsup_{t\rightarrow\infty}\mathbb{P}^{\overline{\gamma},P_{0}}\left(P_{t}^{\overline{\gamma},P_{0}}\in\mathcal{F}\right)\leq\mathbb{\mu}^{\overline{\gamma}}\left(\mathcal{F}\right),~~~\forall~\mbox{closed
set $\mathcal{F}\subset\mathbb{S}_{+}^{N}$}
\end{equation}
Now consider the measurable set
$\Gamma\in\mathcal{B}\left(\mathbb{S}_{+}^{N}\right)$ and let
$X\in\Gamma^{\circ}\cap\mathcal{D}_{I}$, where $\mathcal{D}_{I}$
is the effective domain of $I(\cdot)$. Then, there exists
$\varepsilon>0$, sufficiently small, such that the closed ball
$\overline{B}_{\varepsilon}(X)\in\Gamma$. From (\ref{lower3})
it then follows
\begin{equation}
\label{lower4}
\mathbb{\mu}^{\overline{\gamma}}\left(\Gamma\right)\geq\mathbb{\mu}^{\overline{\gamma}}\left(\overline{B}_{\varepsilon}(X)\right)\geq\limsup_{t\rightarrow\infty}\mathbb{P}^{\overline{\gamma},P_{0}}\left(P_{t}^{\overline{\gamma},P_{0}}\in\mathcal{F}\right)
\end{equation}
We now set to estimate the R.H.S. of (\ref{lower4}).

To this end, recall the nonempty set $\mathcal{S}^{P^{\ast}}(X)$
of all strings of finite (but arbitrary) length with initial state
$P^{\ast}$ and numerical value $X$. For some
$t_{0}\in\mathbb{T}_{+}$ and $i_{1},\cdots,i_{t_{0}}\in\{0,1\}$,
let the string $\mathcal{R}=f_{i_{1}}\circ\cdots\circ
f_{i_{t_{0}}}(P^{\ast})\in\mathcal{S}^{P^{\ast}}(X)$. Define the
function $g:\mathbb{S}_{+}^{N}\longmapsto\mathbb{S}^{N}_{+}$ by
\begin{equation}
\label{lower5} g(Y)=f_{i_{1}}\circ\cdots\circ
f_{i_{t_{0}}}(Y),~~\forall Y\in\mathbb{S}^{N}_{+}
\end{equation}
The function $g(\cdot)$ is continuous (being the composition of
continuous functions) and hence there exists $\varepsilon_{1}>0$,
such that
\begin{equation}
\label{lower6}
\left\|g(Y)-g(P^{\ast})\right\|\leq\varepsilon,~~~\forall
Y\in\overline{B}_{\varepsilon_{1}}(P^{\ast})
\end{equation}
Also, by Lemma~\ref{unif_conv}, there exists
$t_{\varepsilon_{1}}\in\mathbb{T}_{+}$, such that
\begin{equation}
\label{lower7}
\left\|f_{1}^{t}(Y)-P^{\ast}\right\|\leq\varepsilon_{1},~~~\forall
t\geq t_{\varepsilon_{1}},~~Y\in\mathbb{S}_{+}^{N}
\end{equation}
It then follows from eqns.~(\ref{lower6},\ref{lower7}), that, for
any $t\in\mathbb{T}_{+}$, such that $t\geq
t_{0}+t_{\varepsilon_{1}}$ and any string
$\mathcal{R}_{1}\in\mathcal{S}_{t}^{P_{0}}$ of the form
\begin{equation}
\label{lower8} \mathcal{R}_{1}=\left(f_{i_{1}},\cdots,
f_{i_{t_{0}}}, f_{1}^{t_{\varepsilon_{1}}},
f_{j_{1}}\cdots, f_{j_{t-t_{0}-t_{\varepsilon_{1}}}}, P_{0}\right)
\end{equation}
where $j_{1},\cdots,j_{t-t_{0}-t_{\varepsilon_{1}}}\in\{0,1\}$, we
have
\begin{equation}
\label{lower9}
\mathcal{N}\left(\mathcal{R}_{1}\right)\in\overline{B}_{\varepsilon}(X)
\end{equation}
Indeed
\begin{eqnarray}
\label{lower10}
\left\|\mathcal{N}\left(\mathcal{R}_{1}\right)-X\right\| & = &
\left\|\mathcal{N}\left(\mathcal{R}_{1}\right)-\mathcal{N}\left(\mathcal{R}\right)\right\|\nonumber
\\ & = & \left\|g\left(f_{1}^{t_{\varepsilon_{1}}}\left(f_{j_{1}}\cdots\circ
f_{j_{t-t_{0}-t_{\varepsilon_{1}}}}(P_{0})\right)\right)-g\left(P^{\ast}\right)\right\|\nonumber
\\ & \leq & \varepsilon
\end{eqnarray}
where the last step follows from the fact, that,
\begin{equation}
\label{lower11}
\left\|f_{1}^{t_{\varepsilon_{1}}}\left(f_{j_{1}}\circ\cdots\circ
f_{j_{t-t_{0}-t_{\varepsilon_{1}}}}(P_{0})\right)-P^{\ast}\right\|\leq\varepsilon_{1}
\end{equation}
by (\ref{lower7}).

Now for $\mathcal{R}$ (defined above), $t\geq
t_{0}+t_{\varepsilon_{1}}$, consider the set of strings
\begin{equation}
\label{lower12} \mathcal{R}_{t}=\left\{\left.\left(f_{i_{1}},\cdots,
f_{i_{t_{0}}}, f_{1}^{t_{\varepsilon_{1}}},
f_{j_{1}}\cdots,
f_{j_{t-t_{0}-t_{\varepsilon_{1}}}}, P_{0}\right)\right|j_{1},\cdots,j_{t-t_{0}-t_{\varepsilon_{1}}}\in\{0,1\}\right\}
\end{equation}
(in other words, the indices
$j_{1},\cdots,j_{t-t_{0}-t_{\varepsilon_{1}}}$ can be arbitrary.)

It then follows from (\ref{lower9}) above,
\begin{equation}
\label{lower13}
\mathcal{N}\left(\mathcal{R}_{2}\right)\in\overline{B}_{\varepsilon}(X),~~~\forall
\mathcal{R}_{2}\in\mathcal{R}_{t}
\end{equation}
From the iterative construction of the sequence
$\{P_{t}^{\overline{\gamma},P_{0}}\}_{t\in\mathbb{T}_{+}}$ it is
then obvious, for $t\geq t_{0}+t_{\varepsilon_{1}}$,
\begin{eqnarray}
\label{lower14}
\mathbb{P}^{\overline{\gamma},P_{0}}\left(P_{t}^{\overline{\gamma},P_{0}}\in\overline{B}_{\varepsilon}(X)\right)
& \geq &
\mathbb{P}^{\overline{\gamma},P_{0}}\left(P_{t}^{\overline{\gamma},P_{0}}\in\mathcal{N}\left(\mathcal{R}_{t}\right)\right)\nonumber
\\ & = &
\sum_{j_{1},\cdots,j_{t-t_{0}+t_{\varepsilon_{1}}}\in\{0,1\}}\left[\left(\prod_{k=1}^{t_{0}}(1-\overline{\gamma})^{1-i_{k}}\overline{\gamma}^{i_{k}}\right)\overline{\gamma}^{t_{\varepsilon_{1}}}\left(\prod_{k=1}^{t-t_{0}-t_{\varepsilon_{1}}}(1-\overline{\gamma})^{1-j_{k}}\overline{\gamma}^{j_{k}}\right)\right]\nonumber
\\ & = &
(1-\overline{\gamma})^{\pi(\mathcal{R})}\overline{\gamma}^{t_{0}-\pi(\mathcal{R})}\overline{\gamma}^{t_{\varepsilon_{1}}}\nonumber
\\ & = & (1-\overline{\gamma})^{\pi(\mathcal{R})}\overline{\gamma}^{t_{0}+t_{\varepsilon_{1}}-\pi(\mathcal{R})}
\end{eqnarray}
From eqns.~(\ref{lower4},\ref{lower14}) we have
\begin{equation}
\label{lower15}
\mathbb{\mu}^{\overline{\gamma}}\left(\Gamma\right)\geq
(1-\overline{\gamma})^{\pi(\mathcal{R})}\overline{\gamma}^{t_{0}+t_{\varepsilon_{1}}-\pi(\mathcal{R})}
\end{equation}
Noting that $\ln(1-\overline{\gamma})<0$ and passing to the limit
in the above we have
\begin{equation}
\label{lower16} \liminf_{\overline{\gamma}\uparrow 1}-\frac{1}{\ln(1-\overline{\gamma})}\ln\mathbb{\mu}^{\overline{\gamma}}\left(\Gamma\right)\geq
-\pi(\mathcal{R})
\end{equation}
Since the above holds for all
$\mathcal{R}\in\mathcal{S}^{P^{\ast}}(X)$, we have
\begin{equation}
\label{lower17} \liminf_{\overline{\gamma}\uparrow 1}-\frac{1}{\ln(1-\overline{\gamma})}\ln\mathbb{\mu}^{\overline{\gamma}}\left(\Gamma\right)
\geq
\sup_{\mathcal{R}\in\mathcal{S}^{P^{\ast}}(X)}\left(-\pi(\mathcal{R})\right)= -\inf_{\mathcal{R}\in\mathcal{S}^{P^{\ast}}(X)}\left(\pi(\mathcal{R})\right)= -I(X)
\end{equation}
The above holds for all $X\in\Gamma^{\circ}\cap\mathcal{D}_{I}$
and hence
\begin{equation}
\label{lower18} \liminf_{\overline{\gamma}\uparrow 1}-\frac{1}{\ln(1-\overline{\gamma})}\ln\mathbb{\mu}^{\overline{\gamma}}\left(\Gamma\right)
\geq
-\inf_{X\in\Gamma^{\circ}\cap\mathcal{D}_{I}}I(X)= -\inf_{X\in\Gamma^{\circ}}I(X)
\end{equation}
where the last step follows from the fact that, for
$X\notin\mathcal{D}_{I}$, $I(X)=\infty$. To establish the Lemma, it suffices to show the R.H.S. of (\ref{lower18}) satisfies
\begin{equation}
\label{lower19}
-\inf_{X\in\Gamma^{\circ}}I(X)= -\inf_{X\in\Gamma^{\circ}}I_{L}(X)
\end{equation}
which follows from Proposition~\ref{inv_MDP_rate_prop} Assertion (iii), since $\Gamma^{\circ}$ is an open set.

\end{proof}

We extract the following result for later use, which follows from the arguments in the proof of Lemma~\ref{inv_MDP_lower} culminating to (\ref{lower15}).

\begin{corollary}
\label{cor:lower} Let $\mathcal{O}\subset\mathbb{S}_{+}^{N}$ be an open set and $\mathcal{R}\subset\mathcal{S}^{P^{\ast}}$ be a string, such that, $\mathcal{N}(\mathcal{R})\in\mathcal{O}$. Then, there exists a positive integer $t_{\mathcal{R},\mathcal{O}}$ (depending on $\mathcal{R}$ and $\mathcal{O}$), such that,
\begin{equation}
\label{cor:lower1}
\mathbb{\mu}^{\overline{\gamma}}\left(\mathcal{O}\right)\geq (1-\overline{\gamma})^{\pi(\mathcal{R})}\overline{\gamma}^{t_{\mathcal{R},\mathcal{O}}-\pi(\mathcal{R})},~~~\forall \overline{\gamma}\in [0,1]
\end{equation}
\end{corollary}

\subsection{The MDP upper bound}
\label{inv_MDP_upper} In this subsection, we establish the MDP
upper bound for the family of invariant measures as
$\overline{\gamma}\uparrow 1$. The proof is carried out in essentially three stages. First, we establish the upper bound for compact sets and this is done in Lemma~\ref{upper}. Then, in Lemma~\ref{tight} we prove a tightness result on the family of invariant distributions and finally establish the MDP upper bound for closed sets in Lemma~\ref{MDP_upper}.

We start with the following result on topological properties
of strings. We need the following definition.
\begin{definition}[Truncated String]
\label{def_trunc} Let the string $\mathcal{R}$ be given by
\begin{equation}
\label{def_trunc1}
\mathcal{R}=\left(f_{i_{1}},\cdots, f_{i_{t}}, P_{0}\right)
\end{equation}
where $t\in\mathbb{T}_{+}$, $i_{1},\cdots,i_{t}\in\{0,1\}$ and $P_{0}\in\mathbb{S}_{+}^{N}$. Then for $s\leq t$, the truncated string $\mathcal{R}^{s}$ of length $s$ is given by
\begin{equation}
\label{def_trunc2}
\mathcal{R}^{s}=\left(f_{i_{1}},\cdots, f_{i_{s}},P_{0}\right)
\end{equation}
\end{definition}

\begin{lemma}
\label{top_string}Let $\mathcal{F}\in\mathbb{S}_{+}^{N}$ be a
closed set. Define the set of strings
$\mathcal{U}\subset\mathcal{S}^{P^{\ast}}$ by
\begin{equation}
\label{top_string1000}
\mathcal{U}(\mathcal{F})=\left\{\mathcal{R}\in\mathcal{S}^{P^{\ast}}~|~\mathcal{N}(\mathcal{R})\in\mathcal{F}\right\}
\end{equation}
and let
\begin{equation}
\label{top_string1}
\ell(\mathcal{F})=\inf_{\mathcal{R}\in\mathcal{U}(\mathcal{F})}\pi(\mathcal{R})
\end{equation}
(we adopt the convention that the infimum of an empty set is $\infty$.)
Then, if $\ell(\mathcal{F})<\infty$, there exists $t_{\mathcal{F}}\in\mathbb{T}_{+}$ sufficiently large, such that, for all $\mathcal{R}\in\mathcal{U}(\mathcal{F})$ with $\mbox{len}(\mathcal{R})\geq t_{\mathcal{F}}$, we have
\begin{equation}
\label{top_string4} \pi(\mathcal{R}^{t_{\mathcal{F}}})\geq\ell(\mathcal{F})
\end{equation}
\end{lemma}
For a proof see Appendix~\ref{proof_inv_MDP}.
We present the following remark.
\begin{remark}
\label{rem_top_string}
\begin{itemize}
\item [(i)] It follows from the definitions that
\begin{equation}
\label{top_string5}
\mathcal{U}(\mathcal{F})=\cup_{X\in\mathcal{F}}\mathcal{S}^{P^{\ast}}(X)
\end{equation}
and hence
\begin{equation}
\label{top_string6}
\ell(\mathcal{F})=\inf_{X\in\mathcal{F}}\inf_{\mathcal{R}\in\mathcal{S}^{P^{\ast}}(X)}\pi(\mathcal{R})=\inf_{X\in\mathcal{F}}I(X)
\end{equation}

\item [(ii)] If $\ell(\mathcal{F})<\infty$, i.e., the set $\mathcal{U}(\mathcal{F})$ is non-empty, the infimum in (\ref{top_string1}) is attained, i.e., there exists $\mathcal{R}^{\ast}\in\mathcal{U}(\mathcal{F})$, such that,
    \begin{equation}
    \label{top_string7}
    \pi(\mathcal{R}^{\ast})=\ell(\mathcal{F})
    \end{equation}
    This follows from the fact that the function $\pi(\cdot)$ takes only a countable number of values. Similarly, in the case $\ell(\mathcal{F})<\infty$, if we define $X^{\ast}=\mathcal{N}(\mathcal{R}^{\ast})$, then
    \begin{equation}
    \label{top_string8}
    I(X^{\ast})=\inf_{X\in\mathcal{F}}I(X)
    \end{equation}
 %

\end{itemize}
\end{remark}

We now prove the MDP upper bound for the family $\{\mathbb{\mu}^{\overline{\gamma}}\}$ as $\overline{\gamma}\uparrow 1$ over compact sets.

\begin{lemma}
\label{upper}
Let $K\in \mathcal{B}(\mathbb{S}_{+}^{N})$ be a compact set. Then the following upper bound holds:
\begin{equation}
\label{upper1}
\limsup_{\overline{\gamma}\uparrow 1}-\frac{1}{\ln(1-\overline{\gamma})}\ln\mathbb{\mu}^{\overline{\gamma}}\left(K\right)\leq-\inf_{X\in
K}I_{L}(X)
\end{equation}
\end{lemma}
\begin{proof}
For every $\varepsilon>0$, define the $\varepsilon$-closure $\overline{K_{\varepsilon}}$ and the $\varepsilon$-neighborhood $K_{\varepsilon}$ of $K$ by
\begin{equation}
\label{upper2}
\overline{K_{\varepsilon}}=\left\{X\in\mathbb{S}_{+}^{N}~|~\inf_{Y\in K}\left\|X-Y\right\|\leq\varepsilon\right\}
\end{equation}
\begin{equation}
\label{upper3}
K_{\varepsilon}=\left\{X\in\mathbb{S}_{+}^{N}~|~\inf_{Y\in  K}\left\|X-Y\right\|<\varepsilon\right\}
\end{equation}
Since $K_{\varepsilon}$ is open, we have by the weak convergence of the sequence $\{P_{t}^{\overline{\gamma},P^{\ast}}\}_{t\in\mathbb{T}_{+}}$ to $\mathbb{\mu}^{\overline{\gamma}}$
\begin{equation}
\label{upper4}
\liminf_{t\rightarrow\infty}\mathbb{P}^{\overline{\gamma},P^{\ast}}\left(P_{t}^{\overline{\gamma},P^{\ast}}\in   K_{\varepsilon}\right)\geq\mathbb{\mu}^{\overline{\gamma}}\left(K_{\varepsilon}\right)
\end{equation}
which in turn implies
\begin{equation}
\label{upper5}
\liminf_{t\rightarrow\infty}\mathbb{P}^{\overline{\gamma},P^{\ast}}\left(P_{t}^{\overline{\gamma},P^{\ast}}\in\overline{K_{\varepsilon}}\right)\geq\mathbb{\mu}^{\overline{\gamma}}\left(K\right)
\end{equation}
We now estimate the probabilities on the L.H.S. of (\ref{upper5}). Since $\overline{K_{\varepsilon}}$ is closed, the results of Lemma~\ref{top_string} apply and recall the objects $\mathcal{U}(\cdot)$ and $\ell(\cdot)$ defined for any closed set $\mathcal{F}$ as:
\begin{equation}
\label{upper6}
\mathcal{U}(\mathcal{F})=\left\{\mathcal{R}\in\mathcal{S}^{P^{\ast}}~|~\mathcal{N}(\mathcal{R})\in\mathcal{F}\right\}
\end{equation}
\begin{equation}
\label{upper7}
\ell(\mathcal{F})=\inf_{\mathcal{R}\in\mathcal{U}(\mathcal{F})}\pi(\mathcal{R})
\end{equation}
with the convention that $\ell(\mathcal{F})=\infty$ if $\mathcal{U}(\mathcal{F})$ is empty. Also, for every $t\in\mathbb{T}_{+}$ and closed set $\mathcal{F}$ define the sets:
\begin{equation}
\label{upper8}
\mathcal{U}^{t}(\mathcal{F})=\mathcal{U}(\mathcal{F})\cap\mathcal{S}_{t}^{P^{\ast}}
\end{equation}
To establish the Lemma we may consider two cases, as to whether $\ell(K)<\infty$ (i.e, $\mathcal{K}$ is non-empty) or not. We first consider the non-trivial case $\ell(K)<\infty$.

To this end, consider fixed $\varepsilon>0$ and from Proposition~\ref{string_prop} Assertion (ii) it is easy to see that
\begin{equation}
\label{upper9}
\mathbb{P}^{\overline{\gamma},P^{\ast}}\left(P_{t}^{\overline{\gamma},P^{\ast}}\in\overline{K_{\varepsilon}}\right)=\mathbb{P}^{\overline{\gamma},P^{\ast}}\left(P_{t}^{\overline{\gamma},P^{\ast}}\in\mathcal{N}\left(\mathcal{U}^{t}(\overline{K_{\varepsilon}})\right)\right)=1
\end{equation}
Since $K\subset\overline{K_{\varepsilon}}$ and $\ell({K})<\infty$, we have $\ell(\overline{K_{\varepsilon}})<\infty$. The fact that $\overline{K_{\varepsilon}}$ is closed and Lemma~\ref{top_string} imply there exists $t_{\overline{K_{\varepsilon}}}\in\mathbb{T}_{+}$, such that, for every string $\mathcal{R}\in\mathcal{U}(\overline{K_{\varepsilon}})$ with $\mbox{len}(\mathcal{R})\geq t_{\overline{K_{\varepsilon}}}$, we have $\pi\left(\mathcal{R}^{t_{\overline{K_{\varepsilon}}}}\right)\geq\ell(K_{\varepsilon})$.
In other words, we have for all $t\geq t_{\overline{K_{\varepsilon}}}$,
\begin{equation}
\label{upper11}
\pi\left(\mathcal{R}^{t_{\overline{K_{\varepsilon}}}}\right)\geq\ell(K_{\varepsilon}),~~~\forall\mathcal{R}\in\mathcal{U}^{t}(\overline{K_{\varepsilon}})
\end{equation}
Now consider $t\geq t_{\overline{K_{\varepsilon}}}$ and define the set of strings $\mathcal{J}_{t}^{P^{\ast}}$ by
\begin{equation}
\label{upper12}
\mathcal{J}_{t}^{P^{\ast}}=\left\{\left.\mathcal{R}\in\mathcal{S}_{t}^{P^{\ast}}\right|\pi\left(\mathcal{R}^{t_{\overline{K_{\varepsilon}}}}\right)\geq\ell(K_{\varepsilon})\right\}
\end{equation}
In other words, the set $\mathcal{J}_{t}^{P^{\ast}}$ consists of all strings $\mathcal{R}$ of length $t$, such that there are at least $\ell(\overline{K_{\varepsilon}})$ occurrences of $f_{0}$'s in the truncated string $\mathcal{R}^{t_{\overline{K_{\varepsilon}}}}$.
The following inclusion is then obvious for $t\geq t_{\overline{K_{\varepsilon}}}$:
\begin{equation}
\label{upper13}
\mathcal{U}^{t}(\overline{K_{\varepsilon}})\subset\mathcal{J}_{t}^{P^{\ast}}\subset\mathcal{S}_{t}^{P^{\ast}}
\end{equation}
By the Markovian dynamics of the RRE, it is clear, that for $t\geq t_{\overline{K_{\varepsilon}}}$
\begin{equation}
\label{upper14}
\mathbb{P}^{\overline{\gamma},P^{\ast}}\left(P_{t}^{\overline{\gamma},P^{\ast}}\in\mathcal{N}(\mathcal{J}_{t}^{P^{\ast}})\right) = \sum_{\mathcal{R}\in\mathcal{J}_{t}^{P^{\ast}}}(1-\overline{\gamma})^{\pi(\mathcal{R})}\overline{\gamma}^{t-\pi(\mathcal{R})}\leq {t_{\overline{K_{\varepsilon}}}\choose\ell(\overline{K_{\varepsilon}})}(1-\overline{\gamma})^{\ell(\overline{K_{\varepsilon}})}
\end{equation}
We then have from eqns.~(\ref{upper9},\ref{upper13})
\begin{equation}
\label{upper15}\mathbb{\mu}^{\overline{\gamma}}(K) \leq  \liminf_{t\rightarrow\infty}\mathbb{P}^{\overline{\gamma},P^{\ast}}\left(P_{t}^{\overline{\gamma},P^{\ast}}\in\mathcal{N}(\mathcal{U}_{t}(\overline{K_{\varepsilon}}))\right)\leq  \liminf_{t\rightarrow\infty}\mathbb{P}^{\overline{\gamma},P^{\ast}}\left(P_{t}^{\overline{\gamma},P^{\ast}}\in\mathcal{N}(\mathcal{J}_{t}^{P^{\ast}})\right)\leq {t_{\overline{K_{\varepsilon}}}\choose\ell(\overline{K_{\varepsilon}})}(1-\overline{\gamma})^{\ell(\overline{K_{\varepsilon}})}
\end{equation}
Taking the log-limits on both sides and noting that $\ln(1-\overline{\gamma})$ is negative, $t_{\overline{K_{\varepsilon}}}$ is independent of $\overline{\gamma}$ we have
\begin{equation}
\label{upper16}
\limsup_{\overline{\gamma}\uparrow 1}-\frac{1}{\ln(1-\overline{\gamma})}\ln\mathbb{\mu}^{\overline{\gamma}}\left(K\right)\leq-\ln\left({t_{\overline{K_{\varepsilon}}}\choose\ell(\overline{K_{\varepsilon}})}\right)\lim_{\overline{\gamma}\uparrow 1}\frac{1}{\ln(1-\overline{\gamma})}-\lim_{\overline{\gamma}\uparrow 1}\ell(\overline{K_{\varepsilon}})= -\ell(\overline{K_{\varepsilon}})
\end{equation}
Taking the limit on both sides as $\varepsilon\rightarrow 0$ we have
\begin{equation}
\label{upper17}
\limsup_{\overline{\gamma}\uparrow 1}-\frac{1}{\ln(1-\overline{\gamma})}\ln\mathbb{\mu}^{\overline{\gamma}}\left(K\right)\leq-\lim_{\varepsilon\rightarrow 0}\ell(\overline{K_{\varepsilon}})
\end{equation}
From Lemma~\ref{inv_MDP_rate_prop} (Assertion (iii)) we have for all $\varepsilon>0$
\begin{equation}
\label{upper18}
\ell(\overline{K_{\varepsilon}})=\inf_{Y\in\overline{K_{\varepsilon}}}I(Y)\geq\inf_{Y\in\overline{K_{\varepsilon}}}I_{L}(Y)
\end{equation}
Taking the limit and using Lemma~\ref{inv_MDP_rate_prop} (Assertion (iv)) we have
\begin{equation}
\label{upper19}
\lim_{\varepsilon\rightarrow 0}\ell(\overline{K_{\varepsilon}})\geq\lim_{\varepsilon\rightarrow 0}\inf_{Y\in\overline{K_{\varepsilon}}}I_{L}(Y)=\inf_{Y\in K}I_{L}(Y)
\end{equation}
The Lemma then follows from eqns.~(\ref{upper17},\ref{upper19}).
\end{proof}

The following tightness result enables us to extend the upper bound from compact sets to arbitrary closed sets (see Appendix~\ref{proof_inv_MDP} for a proof.)
\begin{lemma}
\label{tight}
The family of invariant distributions $\left\{\mathbb{\mu}^{\overline{\gamma}}\right\}$ satisfies the following tightness property: For every $a>0$, there exists a compact set $K_{a}\subset\mathbb{S}_{+}^{N}$, such that
\begin{equation}
\label{tight1}
\limsup_{\overline{\gamma}\uparrow 1}-\frac{1}{\ln(1-\overline{\gamma})}\mathbb{\mu}^{\overline{\gamma}}(K_{a}^{C})\leq -a
\end{equation}
\end{lemma}

We now complete the proof of the MDP upper bound for arbitrary closed sets by the upper bound for compact sets and the tightness estimate obtained in Lemma~\ref{tight}. This has parallels with the theory of large deviations, where one establishes the LDP upper bound first for compact sets. The tightness analogue of Lemma~\ref{tight} here is called \emph{exponential tightness} in the context of LDP. It can be shown that the LDP upper bound for closed sets follows from that of compact sets and exponential tightness, see, for example,~\cite{DeuschelStroock}. Here, although we are concerned with a MDP, the proof philosophy is related, i.e., we first establish the MDP upper bound for compact sets and then use the tightness estimate of Lemma~\ref{tight} to extend it to arbitrary closed sets.

\begin{lemma}
\label{MDP_upper}
Let $\mathcal{F}\in \mathcal{B}(\mathbb{S}_{+}^{N})$ be a closed set. Then the following upper bound holds:
\begin{equation}
\label{MDP_upper1}
\limsup_{\overline{\gamma}\uparrow 1}-\frac{1}{\ln(1-\overline{\gamma})}\ln\mathbb{\mu}^{\overline{\gamma}}\left(\mathcal{F}\right)\leq-\inf_{X\in
\mathcal{F}}I_{L}(X)
\end{equation}
\end{lemma}
\begin{proof}
Let $a>0$ be an arbitrary positive number. By the tightness estimate in Lemma~\ref{tight}, there exists a compact set $K_{a}\subset\mathbb{S}_{+}^{N}$, such that,
\begin{equation}
\label{MDP_upper2}
\limsup_{\overline{\gamma}\uparrow 1}-\frac{1}{\ln(1-\overline{\gamma})}\mathbb{\mu}^{\overline{\gamma}}(K_{a}^{C})\leq -a
\end{equation}
The set $\mathcal{F}\cap K_{a}$ is compact, being the intersection of a closed and a compact set, and hence the MDP upper bound holds from Lemma~\ref{upper}, i.e., we have
\begin{equation}
\label{MDP_upper3}
\limsup_{\overline{\gamma}\uparrow 1}-\frac{1}{\ln(1-\overline{\gamma})}\ln\mathbb{\mu}^{\overline{\gamma}}\left(\mathcal{F}\cap K_{a}\right)\leq-\inf_{X\in
\mathcal{F}\cap K_{a}}I_{L}(X)
\end{equation}
To estimate the probabilities $\mathbb{\mu}^{\overline{\gamma}}(\mathcal{F})$, we use the decomposition:
\begin{equation}
\label{MDP_upper4}
\mathbb{\mu}^{\overline{\gamma}}(\mathcal{F})=\mathbb{\mu}^{\overline{\gamma}}\left(\mathcal{F}\cap K_{a}\right)+\mathbb{\mu}^{\overline{\gamma}}\left(\mathcal{F}\cap K_{a}^{C}\right)\leq \mathbb{\mu}^{\overline{\gamma}}\left(\mathcal{F}\cap K_{a}\right)+ \mathbb{\mu}^{\overline{\gamma}}\left(K_{a}^{C}\right)
\end{equation}
By Lemma~\ref{lim_num} we then have
\begin{equation}
\label{MDP_upper5}
\limsup_{\overline{\gamma}\uparrow 1}-\frac{1}{\ln(1-\overline{\gamma})}\ln\mathbb{\mu}^{\overline{\gamma}}\left(\mathcal{F}\right)\leq\max\left(\limsup_{\overline{\gamma}\uparrow 1}-\frac{1}{\ln(1-\overline{\gamma})}\ln\mathbb{\mu}^{\overline{\gamma}}\left(\mathcal{F}\cap K_{a}^{C}\right),\limsup_{\overline{\gamma}\uparrow 1}-\frac{1}{\ln(1-\overline{\gamma})}\ln\mathbb{\mu}^{\overline{\gamma}}\left(K_{a}^{C}\right)\right)
\end{equation}
From eqns.~(\ref{upper2},\ref{upper3}) we then have
\begin{eqnarray}
\label{MDP_upper6}
\limsup_{\overline{\gamma}\uparrow 1}-\frac{1}{\ln(1-\overline{\gamma})}\ln\mathbb{\mu}^{\overline{\gamma}}\left(\mathcal{F}\right) & \leq & \max\left(-\inf_{X\in\mathcal{F}\cap K_{a}}I_{L}(X),-a\right)\nonumber \\ & \leq & \max\left(-\inf_{X\in\mathcal{F}}I_{L}(X),-a\right)\nonumber \\ & = & -\min\left(\inf_{X\in\mathcal{F}}I_{L}(X),a\right)
\end{eqnarray}
Since (\ref{MDP_upper6}) holds for all $a\in\mathbb{R}_{+}$, passing to the limit as $a\rightarrow\infty$ on both sides we obtain
\begin{equation}
\label{MDP_upper7}
\limsup_{\overline{\gamma}\uparrow 1}-\frac{1}{\ln(1-\overline{\gamma})}\ln\mathbb{\mu}^{\overline{\gamma}}\left(\mathcal{F}\right)\leq -\inf_{X\in\mathcal{F}}I_{L}(X)
\end{equation}
This establishes the MDP upper bound for arbitrary closed sets.
\end{proof}

\section{Proofs of Theorems}
\label{proof_theorem}
\begin{proof}[Theorem~\ref{thm:mdp}]
The MDP lower and upper bounds obtained in Lemma~\ref{inv_MDP_lower_lemma} and Lemma~\ref{MDP_upper} respectively show that the family $\left\{\mathbb{\mu}^{\overline{\gamma}}\right\}$ satisfies an MDP at scale $-\ln(1-\overline{\gamma})$ as $\overline{\gamma}\uparrow 1$ with a good rate function $I_{L}(\cdot)$. To complete the proof of Theorem~\ref{thm:mdp} it suffices to show that $I(\cdot)=I_{L}(\cdot)$, i.e., the function $I(\cdot)$ is lower semicontinuous. We now show that
\begin{equation}
\label{thm:mdp3}
I(X)=I_{L}(X),~~~\forall X\in\mathbb{S}_{+}^{N}
\end{equation}
Clearly, if $I_{L}(X)=\infty$, the claim in (\ref{thm:mdp3}) follows from (\ref{lower22}). We thus consider the case $I_{L}(X)<\infty$. Since
\begin{equation}
\label{thm:mdp4}
I_{L}(X)=\lim_{\varepsilon\rightarrow 0}\inf_{Y\in B_{\varepsilon}(X)}I(X)
\end{equation}
and the integer-valued quantity $\inf_{Y\in B_{\varepsilon}(X)}I(X)$ is non-decreasing w.r.t. $\varepsilon$, there exists $\varepsilon_{0}>0$, such that
\begin{equation}
\label{thm:mdp5}
\inf_{Y\in B_{\varepsilon}(X)}I(X)=I_{L}(X),~~~\forall \varepsilon\leq\varepsilon_{0}
\end{equation}
The infimum above is achieved for every $\varepsilon>0$, and we conclude that there exists a sequence $\{X_{n}\}_{n\in\mathbb{N}}$, such that
\begin{equation}
\label{thm:mdp6}
X_{n}\in \overline{B_{\varepsilon_{0}}}(X),~~~\lim_{n\rightarrow\infty}X_{n}=X,~~~I(X_{n})=I_{L}(X)
\end{equation}
Recall the set of strings
\begin{equation}
\label{thm:mdp7}
\mathcal{U}(\overline{B_{\varepsilon_{0}}}(X))=\{\mathcal{R}\in\mathcal{S}^{P^{\ast}}~|~\mathcal{N}(\mathcal{R})\in \overline{B_{\varepsilon_{0}}}(X)\}
\end{equation}
We then have
\begin{equation}
\label{thm:mdp8}
\ell(\overline{B_{\varepsilon_{0}}}(X))=\inf_{Y\in\overline{B_{\varepsilon_{0}}}(X)}I(X)=I_{L}(X)
\end{equation}
Since $\overline{B_{\varepsilon_{0}}}(X)$ is closed, by Lemma~\ref{top_string}, there exists $t_{0}\in\mathbb{T}_{+}$, such that, for $\mathcal{R}\in\mathcal{U}(\overline{B}_{\varepsilon_{0}}(X))$ with $\mbox{len}(\mathcal{R})\geq t_{0}$, we have
\begin{equation}
\label{thm:mdp9}
\pi\left(\mathcal{R}^{t_{0}}\right)\geq\ell(\overline{B}_{\varepsilon_{0}})=I_{L}(X)
\end{equation}
By the existence of $\{X_{n}\}$, there exists a sequence $\{\mathcal{R}_{n}\}$ of strings in $\mathcal{U}(\overline{B}_{\varepsilon_{0}})$, such that
\begin{equation}
\label{thm:mdp10}
\mathcal{N}\left(\mathcal{R}_{n}\right)=X_{n},~~~\pi\left(\mathcal{R}_{n}\right)=I_{L}(X)
\end{equation}
Note that, without loss of generality, we can assume that $\mbox{len}(\mathcal{R}_{n})=t_{0}$ for all $n$. Indeed, if $\mbox{len}(\mathcal{R}_{n})<t_{0}$, we can modify $\mathcal{R}_{n}$ by appending the requisite number of $f_{1}$'s at the right end, still satisfying (\ref{thm:mdp10}). On the other hand, if $\mbox{len}(\mathcal{R}_{n})>t_{0}$, we note that $\mathcal{R}_{n}$ must be of the form
\begin{equation}
\label{thm:mdp11}
\mathcal{R}_{n}=\left(f_{i_{1}},\cdots, f_{i_{t_{0}}}, f_{1}^{\mbox{len}(\mathcal{R}_{n})-t_{0}}, P^{\ast}\right)
\end{equation}
where the truncated string
\begin{equation}
\label{thm:mdp12}
\mathcal{R}_{n}^{t_{0}}=f_{i_{1}}\circ\cdots\circ f_{i_{t_{0}}}(P^{\ast})
\end{equation}
satisfies
\begin{equation}
\label{thm:mdp13}
\mathcal{N}\left(\mathcal{R}_{n}^{t_{0}}\right)=X_{n},~~~\pi\left(\mathcal{R}_{n}^{t_{0}}\right)=I_{L}(X)
\end{equation}
The second inequality in (\ref{thm:mdp13}) follows from (\ref{thm:mdp9}), whereas (\ref{thm:mdp12}) follows from the fact that $\pi(\mathcal{R}_{n})=I_{L}(X)$ implying
\begin{equation}
\label{thm:mdp14}
\pi(\mathcal{R}_{n})=\pi(\mathcal{R}_{n}^{t_{0}})
\end{equation}
Thus the right end of $\mathcal{R}_{n}$ does not contain any $f_{0}$, explaining the form in (\ref{thm:mdp12}). The key conclusion of the above discussion is that, if $\mbox{len}(\mathcal{R}_{n})>t_{0}$, we may consider the truncated string $\mathcal{R}_{n}^{t_{0}}$ instead, which also satisfies (\ref{thm:mdp10}).

We thus assume that the sequence $\{\mathcal{R}_{n}\}$ with the properties in (\ref{thm:mdp10}) satisfy:
\begin{equation}
\label{thm:mdp15}
\mbox{len}(\mathcal{R}_{n})=t_{0},~~~\forall n
\end{equation}
The number of distinct strings in the sequence $\{\mathcal{R}_{n}\}$ is at most $2^{t_{0}}$ (in fact, lesser than that, because of the constraint $\pi(\mathcal{R}_{n})=I_{L}(X)$) and, hence, at least one pattern is repeated infinitely often in the sequence $\{\mathcal{R}_{n}\}$, i.e., there exists a string $\mathcal{R}^{\ast}$, such that,
\begin{equation}
\label{thm:mdp16}
\mbox{len}(\mathcal{R}^{\ast})=t_{0},~~~\pi\left(\mathcal{R}^{\ast}\right)=I_{L}(X)
\end{equation}
and a subsequence $\{\mathcal{R}_{n_{k}}\}_{k\in\mathbb{N}}$ of $\{\mathcal{R}_{n}\}$, such that,
\begin{equation}
\label{thm:mdp17}
\mathcal{R}_{n_{k}}=\mathcal{R}^{\ast},~~~\forall k\in\mathbb{N}
\end{equation}
The corresponding subsequence $\{X_{n_{k}}\}$ of numerical values then satisfy
\begin{equation}
\label{thm:mdp18}
X_{n_{k}}=\mathcal{N}(\mathcal{R}_{n_{k}})=\mathcal{N}(\mathcal{R}^{\ast}),~~~\forall k\in\mathbb{N}
\end{equation}
and hence
\begin{equation}
\label{thm:mdp19}
X=\lim_{k\rightarrow\infty}X_{n_{k}}=\mathcal{N}(\mathcal{R}^{\ast})
\end{equation}
Thus the string $\mathcal{R}^{\ast}\in\mathcal{S}^{P^{\ast}}(X)$ and hence
\begin{equation}
\label{thm:mdp20}
I(X)=\inf_{\mathcal{R}\in\mathcal{S}^{P^{\ast}}(X)}\pi(\mathcal{R})\leq\pi\left(\mathcal{R}^{\ast}\right)=I_{L}(X)
\end{equation}
The other inequality $I_{L}(X)\leq I(X)$ is obvious and hence we conclude from (\ref{thm:mdp20}) that
\begin{equation}
\label{thm:21}
I_{L}(X)=I(X)
\end{equation}
This completes the proof of Theorem~\ref{thm:mdp}.
\end{proof}

%

\begin{proof}[Theorem~\ref{thm:convrate}]
Recall
\begin{equation}
\label{thm:convrate3}
d_{P}\left(\mathbb{\mu}^{\overline{\gamma}},\delta_{P^{\ast}}\right)=\inf\left\{\varepsilon>0~|~\delta_{P^{\ast}}(\mathcal{F})\leq \mathbb{\mu}^{\overline{\gamma}}(\mathcal{F}_{\varepsilon})+\varepsilon,~~~\forall~\mbox{closed set}~\mathcal{F}\right\}
\end{equation}
Define the class of sets
\begin{equation}
\label{thm:convrate5}
\mathcal{C}=\{\mathcal{F}~|~\mathcal{F}~\mbox{is closed and}~P^{\ast}\in\mathcal{F}\}
\end{equation}
Then the following equivalence is straight forward:
\begin{equation}
\label{thm:convrate6}
d_{P}\left(\mathbb{\mu}^{\overline{\gamma}},\delta_{P^{\ast}}\right)=\inf\left\{\varepsilon>0~|~ \mathbb{\mu}^{\overline{\gamma}}(\mathcal{F}_{\varepsilon})+\varepsilon\geq 1,~~~\forall~\mathcal{F}\in\mathcal{C}\right\}
\end{equation}
Now consider $0<\varepsilon<1$, small enough. Then there exists $\varepsilon_{0}>0$, such that for every $\mathcal{F}\in\mathcal{C}$, we have $B_{\varepsilon_{0}}(P^{\ast})\subset\mathcal{F}_{\varepsilon}$.
(Note that the constant $\varepsilon_{0}$ can be chosen independently of $\mathcal{F}$, but depends on $\varepsilon$.)
The string $\mathcal{R}=P^{\ast}$ belongs to $B_{\varepsilon_{0}}(P^{\ast})$ and hence by Corollary~\ref{cor:lower}, there exists an integer $t_{0}>0$, such that,
\begin{equation}
\label{thm:convrate8}
\mathbb{\mu}^{\overline{\gamma}}\left(B_{\varepsilon_{0}}(P^{\ast})\right)\geq (1-\overline{\gamma})^{\pi(\mathcal{R})}\overline{\gamma}^{t_{0}-\pi(\mathcal{R})}=\overline{\gamma}^{t_{0}}
\end{equation}
Thus for all $\mathcal{F}\in\mathcal{C}$
\begin{equation}
\label{thm:convrate9}
\mathbb{\mu}^{\overline{\gamma}}\left(\mathcal{F}_{\varepsilon})\right)\geq\mathbb{\mu}^{\overline{\gamma}}\left(B_{\varepsilon_{0}}(P^{\ast})\right)\geq\overline{\gamma}^{t_{0}}
\end{equation}
Then for $\overline{\gamma}\geq (1-\varepsilon)^{1/t_{0}}$ we have for all $\mathcal{F}$
\begin{equation}
\label{thm:convrate10}
\mathcal{F}_{\varepsilon}+\varepsilon\geq \overline{\gamma}^{t_{0}}+\varepsilon\geq 1
\end{equation}
It then follows from (\ref{thm:convrate6}) that
\begin{equation}
\label{thm:convrate11}
d_{P}\left(\mathbb{\mu}^{\overline{\gamma}},\delta_{P^{\ast}}\right)\leq\varepsilon,~~~\overline{\gamma}\geq (1-\varepsilon)^{1/t_{0}}
\end{equation}
Hence
\begin{equation}
\label{thm:convrate12}
\limsup_{\overline{\gamma}\uparrow 1}d_{P}\left(\mathbb{\mu}^{\overline{\gamma}},\delta_{P^{\ast}}\right)\leq\varepsilon
\end{equation}
Since $\varepsilon>0$ is arbitrary, by passing to the limit as $\varepsilon\rightarrow 0$, we have the weak convergence
\begin{equation}
\label{thm:convrate13}
\lim_{\overline{\gamma}\uparrow 1}d_{P}\left(\mathbb{\mu}^{\overline{\gamma}},\delta_{P^{\ast}}\right)=0
\end{equation}

For the second assertion we note that for $\varepsilon>0$, the closed set $B_{\varepsilon}^{C}(P^{\ast})$ does not contain $P^{\ast}$. Hence $\ell(B_{\varepsilon}^{C}(P^{\ast}))\geq 1$.
The claim in (\ref{thm:convrate4}) then follows from the MDP upper bound for closed sets (Lemma~\ref{MDP_upper}.)
\end{proof}

\section{Computations with the rate function}
\label{comp} A complete characterization of probabilities under the invariant distributions are obtained in Theorem~\ref{thm:mdp}, which shows that the probability of `rare events'\footnote{The term rare event in this context refers to a Borel set (event) bounded away from $P^{\ast}$.} decays as powers of $\ln(1-\overline{\gamma})$ as $\overline{\gamma}\uparrow 1$. The best exponent of this power law decay is characterized by the MDP rate function $I(\cdot)$. As Theorem~\ref{thm:mdp} shows, the best decay exponent of such a rare event can be computed as the infimum of the $I(\cdot)$ over that set. This reduces the problem of estimating probabilities under the invariant distributions to solutions of a related variational problem, namely, minimizing $I(\cdot)$ over Borel sets in $\mathbb{S}_{+}^{N}$. From a numerical analysis point of view, one may simulate the function $I(\cdot)$ and use a look-up table to numerically solve the associated variational problems. In this section, we show that, for a class of events of interest, the variational problem of computing the best decay exponent can be simplified to a great extent. In particular, we are interested in estimating probabilities of the form $\mathbb{\mu}^{\overline{\gamma}}(B_{\varepsilon}^{C}(P^{\ast}))$ for every $\varepsilon>0$. Theorem~\ref{thm:convrate} shows that such probabilities decay to zero at least as $(1-\overline{\gamma})$, i.e., for every $\varepsilon>0$, we have
\begin{equation}
\label{comp1}
\mathbb{\mu}^{\overline{\gamma}}(B_{\varepsilon}^{C}(P^{\ast}))\equiv O
(1-\overline{\gamma})
\end{equation}
However, this upper bound becomes loose as $\varepsilon$ increases and the best exponent of decay, i.e., the best power of $(1-\overline{\gamma})$ in (\ref{comp1}) can be computed precisely by solving the variational problem in Theorem~\ref{thm:mdp}. The main result of this section shows that, for a particular class of systems, this computation can be extremely simplified and for general systems we present bounds on the decay exponent that are simple to obtain in contrast to solving the full-fledged variational problem. For a system designer, the probabilities of the form in (\ref{comp1}) are of interest and a technique to obtain them efficiently is of much use.
 We start by setting some notation.

Define the function $\iota:\mathbb{R}_{+}\longmapsto\mathbb{T}_{+}$ by
\begin{equation}
\label{comp2}
\iota(M)=\inf\left\{k\in\mathbb{T}_{+}~|~\left\|f_{0}^{k}(P^{\ast})-P^{\ast}\right\|\geq M\right\}
\end{equation}
Also, define $\iota_{+}:\mathbb{R}_{+}\longmapsto\mathbb{T}_{+}$ by
\begin{equation}
\label{comp3}
\iota_{+}(M)=\inf\left\{k\in\mathbb{T}_{+}~|~\left\|f_{0}^{k}(P^{\ast})-P^{\ast}\right\|> M\right\}
\end{equation}
We note that $\iota(\cdot)$ is a non-decreasing right continuous function, and we have for all $M>0$
\begin{equation}
\label{comp4}
\iota_{+}(M)\leq\iota(M)+1,~~~\lim_{U<M:U\rightarrow M}\iota(U)=\iota_{+}(M)
\end{equation}
Also recall:
\begin{equation}
\label{comp5}
B_{M}^{C}(P^{\ast})=\{X\in\mathbb{S}_{+}^{N}~|~\left\|X-P^{\ast}\right\|\geq M\}
\end{equation}
\begin{equation}
\label{comp6}
\overline{B}_{M}^{C}(P^{\ast})=\{X\in\mathbb{S}_{+}^{N}~|~\left\|X-P^{\ast}\right\|> M\}
\end{equation}
\begin{definition}[Class $\mathcal{A}$ systems]: Let $(A,Q,C,R)$ be a system satisfying the assumption \textbf{(E.1)}. Then the system is called a class $\mathcal{A}$ system if
\begin{equation}
\label{comp7}\mathcal{S}^{-}\supset\{X\in\mathbb{S}_{+}^{N}~|~X\succeq f_{0}(P^{\ast})\}
\end{equation}
where $S^{-}$ is defined in (\ref{prop_string3:1}).
\end{definition}

We then have the following MDP asymptotics for class $\mathcal{A}$ systems (see Appendix~\ref{proof_compMDP} for a proof):
\begin{lemma}
\label{compMDP}
Let $(A,Q,C,R)$ be a class $\mathcal{A}$ system.
Then we have for all $M>0$
\begin{equation}
\label{compMDP2}
\limsup_{\overline{\gamma}\uparrow 1}-\frac{1}{\ln(1-\overline{\gamma})}\ln\mathbb{\mu}^{\overline{\gamma}}\left(B_{M}^{C}(P^{\ast})\right)\leq-\iota(M)
\end{equation}
and
\begin{equation}
\label{compMDP3}
\liminf_{\overline{\gamma}\uparrow 1}-\frac{1}{\ln(1-\overline{\gamma})}\ln\mathbb{\mu}^{\overline{\gamma}}\left(\overline{B}_{M}^{C}(P^{\ast})\right)\geq-\iota_{+}(M)
\end{equation}
\end{lemma}

\begin{remark}
\label{rm:compMDP}
Lemma~\ref{compMDP} shows that for class $\mathcal{A}$ systems the variational problem of computing the best decay exponent can be simplified to a great extent. In particular, rather than generating the set of all possible strings $\mathcal{S}^{P^{\ast}}$ (that grows exponentially with the length of the strings) one can systematically look into strings of the form $f_{0}^{k}(P^{\ast}),~k\in\mathbb{T}_{+}$ and obtain the decay exponent. The next natural question is whether there exists a suitable characterization of class $\mathcal{A}$ systems. Determining whether a given system is class $\mathcal{A}$ is numerically simple, as one only needs to check the condition in (\ref{comp7}). From a theoretical point of view, it would be relevant to offer a characterization of class $\mathcal{A}$ systems through properties of the system matrices. A detailed study to that end will be a digression from the main theme of the paper, and we intend to do it elsewhere. However, we show that scalar systems are included in class $\mathcal{A}$, thereby confirming that it is not empty.
\begin{proposition}
\label{scalA} Let $(A,Q,C,R)$ be a scalar system satisfying~\textbf{(E.1)} (i.e., $Q,R>0$ and $C\neq 0$.) Then, $(A,Q,C,R)$ belongs to class $\mathcal{A}$.
\end{proposition}
\begin{proof} Define the function $g:\mathbb{R}_{+}\longmapsto\mathbb{R}_{+}$ by
\begin{equation}
\label{scalA1} g(X)=f_{1}(X)-X,~~~\forall X\geq 0
\end{equation}
where $f_{1}$ is the scalar Riccati operator. Under the assumptions, $f_{1}$ has only one fixed point in $\mathbb{R}_{+}$, $P^{\ast}$, the steady state solution. Thus, in the domain of interest,
\begin{equation}
\label{scalA2} g(X)=0~~~\mbox{\emph{iff}}~~~X=P^{\ast}
\end{equation}
We note that $g(0)>0$, and, by~\cite{Bucy-Riccati}, there exists $\alpha_{P^{\ast}}\geq P^{\ast}$ sufficiently large, such that,
\begin{equation}
\label{scalA3}
g(X)=f_{1}(X)-X<0,~~~\forall X> \alpha_{P^{\ast}}
\end{equation}
We now claim that
\begin{equation}
\label{scalA4}
g(X)>0,~~0\leq X<P^{\ast}~~~\mbox{and}~~~g(X)<0,~~X>P^{\ast}
\end{equation}
Indeed, if this was not true, then by (\ref{scalA3}) this would imply the existence of an interval in $\mathbb{R}_{+}$ not containing $P^{\ast}$, such that the sign of $g(\cdot)$ changes over this interval. This in turn would imply from the continuity of $g(\cdot)$, the existence of another solution to the equation $g(X)=0$ on $\mathbb{R}_{+}$ other than $P^{\ast}$. Clearly, this contradicts with the hypothesis (see (\ref{scalA2}) and, hence, the claim in (\ref{scalA4}) holds.

Since $f_{0}(P^{\ast})>P^{\ast}$, it then follows from (\ref{scalA4}) that, for $X\geq f_{0}(P^{\ast})$,
\begin{equation}
\label{scalA5}
f_{1}(X)-X = g(X)<0
\end{equation}
thus showing that scalar systems belong to class $\mathcal{A}$.
\end{proof}
\end{remark}

As shown by Lemma~\ref{compMDP}, for class $\mathcal{A}$ systems the computation of the decay exponent of rare events can be greatly simplified. For general systems such a simplification may not be possible, however, the solution of the variational problem in Theorem~\ref{thm:mdp} can still be made more efficient rather than searching haphazardly over the set of strings $\mathcal{S}^{P^{\ast}}$. The following proposition outlines a simple algorithm for solving the variational problems of interest leading to the decay exponents of rare events in general systems:
\begin{proposition}
\label{compgen}
Let $\Gamma\subset\mathbb{S}_{+}^{N}$ and define
\begin{equation}
\label{compgen1}
k_{\Gamma}=\inf\{k\geq 0~|~\mathcal{N}\left(f_{0}^{k}(P^{\ast})\right)\in\Gamma\}
\end{equation}
Define the set
\begin{equation}
\label{compgen2}
\mathcal{J}_{\Gamma}=\{\mathcal{R}\in\mathcal{S}^{P^{\ast}}~|~\mathcal{N}(\mathcal{R})\in\Gamma~~\mbox{and}~~\mbox{len}(\mathcal{R})\leq k_{\Gamma}\}
\end{equation}
Then
\begin{equation}
\label{compgen3}
\inf_{X\in\Gamma}I(X)=\inf_{\mathcal{R}\in\mathcal{J}_{\Gamma}}\pi(\mathcal{R})
\end{equation}
\end{proposition}
\begin{proof} The proof is straight-forward and follows from the fact, that, $\inf_{X\in\Gamma}I(X)\leq k_{\Gamma}$
and so it suffices to look at strings of length $k_{\Gamma}$ at most.
\end{proof}
\begin{remark}
\label{rm:compgen} The conclusion of Proposition~\ref{compgen} is that, to find the decay exponent of a rare event $\Gamma$, one may compute $k_{\Gamma}$ according to (\ref{compgen1}) and then the minimizing string can be found in the set $\mathcal{J}_{\Gamma}$ as constructed above.
\end{remark}

\section{A Scalar Example}
\label{sub_scalar} We present a numerical study to demonstrate the efficiency of our approach over extensive Monte-Carlo type simulations to estimate the decay rate of rare events.

Consider a scalar system with parameters: $A=\sqrt{2},C=Q=R=1$. Solving the algebraic Riccati equation, $X=f_{1}(X)$, we obtain $P^{\ast}=1+\sqrt{2}$. By Proposition~\ref{scalA}, we note that the system is of class $\mathcal{A}$. By Lemma~\ref{compMDP}, we then have for $M>0$,
\begin{equation}
\label{sc_ex1}
-\iota_{+}(M)\liminf_{\overline{\gamma}\uparrow 1}-\frac{1}{\ln(1-\overline{\gamma})}\ln\mathbb{\mu}^{\overline{\gamma}}\left(B_{M}^{C}(P^{\ast})\right)\leq\limsup_{\overline{\gamma}\uparrow 1}-\frac{1}{\ln(1-\overline{\gamma})}\ln\mathbb{\mu}^{\overline{\gamma}}\left(B_{M}^{C}(P^{\ast})\right)\leq-\iota(M)
\end{equation}
(note, we use the alternative MDP representation,~\eqref{moddev10000}.)

Now choose $M_{1}=40-P^{\ast}$ and we estimate the decay rate of the rare event $B_{M_{1}}^{C}(P^{\ast})$ as $\overline{\gamma}\rightarrow 1$. Using the definitions and that
\begin{equation}
\label{sc_ex2000}
P^{\ast}<f_{0}^{3}(P^{\ast})<M_{1}+P^{\ast}<f_{0}^{4}(P^{\ast})
\end{equation}
we note that
\begin{equation}
\label{sc_ex2}
\iota(M_{1})=\iota_{+}(M_{1})=4
\end{equation}
By~\eqref{sc_ex1}, $B_{M_{1}}^{C}(P^{\ast})$ is an $I$-continuity set (see text following~\eqref{moddev7}), hence the limit in~\eqref{sc_ex1} exists and we have
\begin{equation}
\label{sc_ex3}
\lim_{\overline{\gamma}\uparrow 1}\frac{1}{\ln(1-\overline{\gamma})}\ln\mathbb{\mu}^{\overline{\gamma}}\left(B_{M_{1}}^{C}(P^{\ast})\right)=4
\end{equation}
Our theory then predicts, that, for $\overline{\gamma}$ close to 1,
\begin{equation}
\label{sc_ex4}
\mathbb{\mu}^{\overline{\gamma}}\left(B_{M_{1}}^{C}(P^{\ast})\right)\sim (1-\overline{\gamma})^{4}
\end{equation}

\begin{figure}[ptb]
\begin{center}
\includegraphics[height=2.5in, width=2.5in]{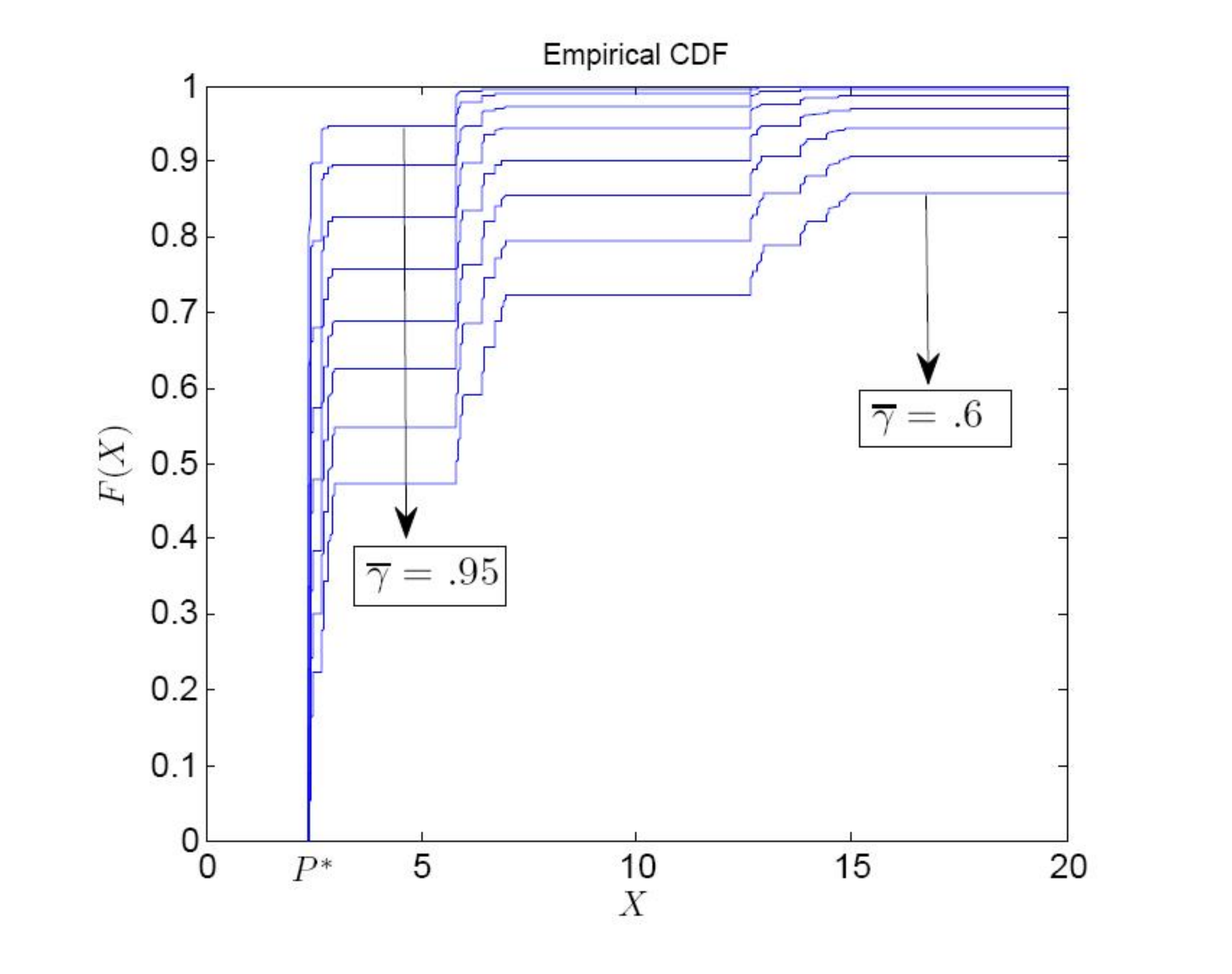}
\includegraphics[height=2.5in, width=2.5in]{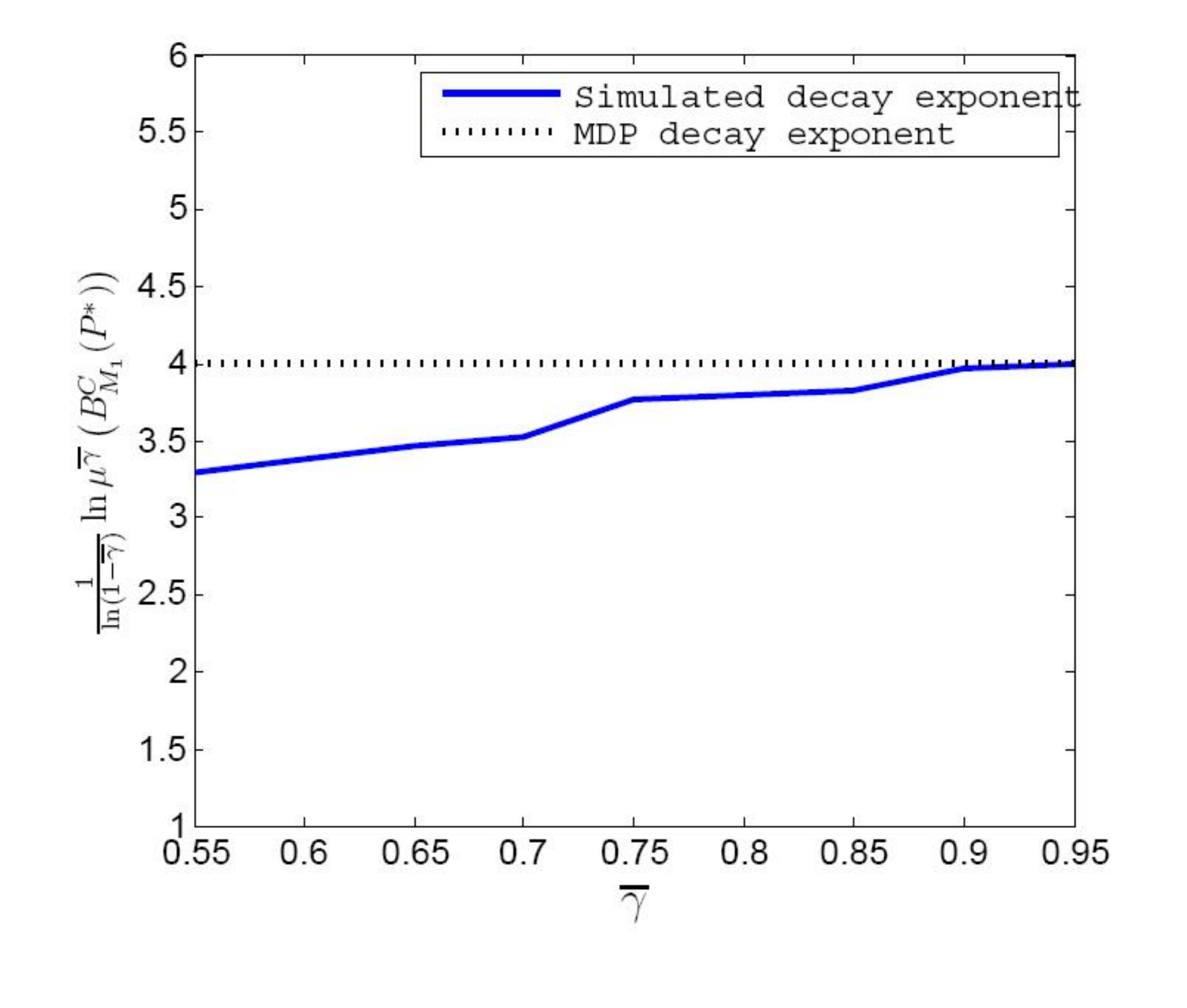}
\caption{Left: Weak convergence of (empirical) measures $\mathbb{\mu}^{\overline{\gamma}}$ to $\delta_{P^{\ast}}$ as $\overline{\gamma}\rightarrow 1$. Right: Decay exponent of probability of the rare event $B_{M_{1}}^{C}(P^{\ast})$.}
\label{fig_1}
\label{fig_M_1}
\end{center}
\end{figure}

We now estimate the empirical decay rate of the event through extensive numerical simulations. We simulate different values of $\overline{\gamma}$, in the range $[.55,1)$ with a step size of .05. For each such value of $\overline{\gamma}$, we obtained $10^{4}$ samples from the invariant measure $\mathbb{\mu}^{\overline{\gamma}}$ (this is needed as the event $\left(B_{M_{1}}^{C}(P^{\ast})\right)$ becomes increasingly difficult to observe as $\overline{\gamma}$ approaches 1.) Obtaining a sample from $\mathbb{\mu}^{\overline{\gamma}}$ is also numerically intensive, and we iterated the RRE 100 times to make sure the random covariance converged in distribution to $\mathbb{\mu}^{\overline{\gamma}}$. Thus, a total of $9\times 10^{4}\times 10^{2}$ computations were involved. The resulting empirical cumulative distributions functions (cdf) are plotted in Fig.~\ref{fig_1} (on the left). We note that, as $\overline{\gamma}\rightarrow 1$, the empirical invariant measures converge to $\delta_{P^{\ast}}$, thus verifying Theorem~\ref{thm:convrate}.

To obtain the empirical decay rate of $B_{M_{1}}^{C}(P^{\ast})$, for each $\overline{\gamma}$, we numerically estimate the quantity $\frac{\ln\mathbb{\mu}^{\overline{\gamma}}\left(B_{M_{1}}^{C}(P^{\ast})\right)}{\ln(1-\overline{\gamma})}$ from the empirical cdfs obtained above. This is plotted in Fig.~\ref{fig_M_1} (on the right) as a function of $\overline{\gamma}$ (the solid line). The result agrees with our theoretical findings: (1) qualitatively, the rare event decays as a power law of $(1-\overline{\gamma})$; (2) the best decay exponent approaches 4 (which is theoretically established in~\eqref{sc_ex3}) as $\overline{\gamma}$ gets closer to 1; (3) even for $\overline{\gamma}$ much less than 1, the empirical decay rate is close to 4, justifying~\eqref{sc_ex4}.

This example truly demonstrates the relevance of our theoretical findings. Even for the simple scalar case, a modest numerical estimation of the probabilities of rare events required computations of the order $10^{7}$, whereas, our theoretical findings provide the best decay exponent by solving a much simpler variational problem.

\section{Conclusions}
\label{conclusion} The paper studies the RRE arising from the problem of Kalman filtering with intermittent observations. We show that for every $\overline{\gamma}>0$ the conditional mean-squared error process is ergodic and the resulting family of invariant distributions $\{\mathbb{\mu}^{\overline{\gamma}}\}$ (as they converge weakly to $\delta_{P^{\ast}}$ as $\overline{\gamma}\uparrow 1$) satisfies a MDP with good rate function $I$. The rate function $I$ is completely characterized and the asymptotic decay rate of rare events characterized as solutions of deterministic variational problems. The intermediate results obtained are of independent interest and our methods are fairly general to be applicable to the analysis of more complex networked control systems (see, for example,~\cite{RDS-ACC-2010}) and hybrid or switched systems.

\section{Acknowledgements}
Soummya Kar wishes to thank Kavita Ramanan of the Department of Mathematical Sciences, Carnegie Mellon University, for insightful discussions on large deviations of Markov processes.

\appendices

\renewcommand{\baselinestretch}{1}

\section{Proofs of Proposition~\ref{string_prop}, Lemma~\ref{unif_conv}}
\label{proof_approx12}
{\small

\textbf{Proof of Proposition~\ref{string_prop}}

\begin{proof}
Assertion (i) follows from the fact, that,
$f_{1}\left(P^{\ast}\right)=P^{\ast}$, whereas Assertion (ii) is
obvious from the iterated construction of the sequence
$\{P_{t}\}_{t\in\mathbb{T}_{+}}$. We prove Assertion (iii) now.


Following Bucy (\cite{Bucy-Riccati}), we define the set
\begin{equation}
\label{prop_string3:1}
S^{-}=\left\{X\in\mathbb{S}_{+}^{N}~|~f_{1}(X)\preceq X\right\}
\end{equation}
Under the assumptions of observability and controllability, it can
be shown (see~\cite{Bucy-Riccati}), there exists
$\beta\in\mathbb{R}_{+}$, such that,
\begin{equation}
\label{prop_string3:2}
\left\{X\in\mathbb{S}_{+}^{N}~|~X\succeq\beta I\right\}\subset
S^{-}
\end{equation}
Now choose $\alpha_{P^{0}}\geq\beta$, such that,
$P^{0}\preceq\alpha_{P^{0}}I$. Then $\alpha_{P^{0}}\in S^{-}$ and
it follows from the order-preserving property of the operators
$f_{0}$ and $f_{1}$, that,
\begin{equation}
\label{prop_string3:3}
\mathcal{N}\left(\mathcal{R}\right)\preceq\mathcal{N}\left(\mathcal{R}_{1}\right)
\end{equation}
where $\mathcal{R}_{1}=\left(f_{i_{1}},\cdots,
f_{i_{t}}, \alpha_{P^{0}}I\right)$.
Note that the claim is trivial for $t=0$, so assume $t\geq 1$. Let
$j_{1}\leq\cdots\leq j_{\pi(\mathcal{R})}$ be the indices
corresponding to $f_{0}$s in $\mathcal{R}$ from left to right.
Consider the last segment of $\mathcal{R}_{1}$, i.e., the string
$\left(f_{0},
f_{1}^{t-j_{\pi(\mathcal{R})}}, \alpha_{P^{0}}I\right)$. We
then have
\begin{equation}
\label{prop_string3:7} f_{0}\circ
f_{1}^{t-j_{\pi(\mathcal{R})}}\left(\alpha_{P^{0}}I\right)\preceq f_{0}\left(\alpha_{P^{0}}I\right)
\end{equation}
This follows from the fact, that, $f_{1}\left(\alpha_{P^{0}}I\right)\preceq
\alpha_{P^{0}}I$,
which implies the sequence
$\{f_{1}^{s}\left(\alpha_{P^{0}}I\right)\}_{s\geq 0}$ is
decreasing and we have $f_{1}^{t-j_{\pi(\mathcal{R})}}\left(\alpha_{P^{0}}I\right)\preceq
\alpha_{P^{0}}I$.
Since $\alpha_{P^{0}}\geq\beta$ and $A$ is unstable, we have $f_{0}\left(\alpha_{P^{0}}I\right)\succeq\alpha_{P^{0}}I\succeq\beta
I$
and $f_{0}\left(\alpha_{P^{0}}I\right)\subset S^{-}$. In
particular, we have
\begin{equation}
\label{prop_string3:11}
\mathcal{N}\left(\mathcal{R}_{1}\right)\preceq
f_{i_{1}}\circ\cdots
f_{i_{j_{\pi(\mathcal{R})}}-1}\left(f_{0}\left(\alpha_{P^{0}}I\right)\right)
\end{equation}
using the order-preserving property of the functions
$f_{0},f_{1}$. In a similar way, we can repeat the above argument
inductively starting with the string on the R.H.S. of
(\ref{prop_string3:11}) and arrive at $\mathcal{N}\left(\mathcal{R}\right)\preceq f_{0}^{\pi(\mathcal{R})}\left(\alpha_{P^{0}}I\right)$.
The claim then follows from (\ref{prop_string3:3}.)

\end{proof}

\textbf{Proof of Lemma~\ref{unif_conv}}

\begin{proof} We use the following result on uniform convergence
over compact sets of the Riccati iterates to $P^{\ast}$. Let
$X_{1},X_{2}\in\mathbb{S}^{N}_{+}$. Then it can be shown (see
Theorem 7.5 in~\cite{Jazwinski})\footnote{The result
in~\cite{Jazwinski} applies to time-variant system matrices under
the assumptions of uniform complete observability and uniform
complete controllability, which reduces to the observability and
controllability for the time-invariant system matrices considered
in this paper.} that there exist constants $c_{1},c_{2}>0$, such
that,
\begin{equation}
\label{unif_conv3}
\left\|f_{1}^{t}(X_{1})-f_{1}^{t}(X_{2})\right\|\leq
c_{1}e^{-c_{2}t}\left\|X_{1}-X_{2}\right\|
\end{equation}
Taking $X_{1}=X$, $X_{2}=P^{\ast}$ and noting that
$f_{1}(P^{\ast})=P^{\ast}$ we have from the above
\begin{equation}
\label{unif_conv4} \left\|f_{1}^{t}(X)-P^{\ast}\right\|\leq
c_{1}e^{-c_{2}t}\left\|X_{1}-P^{\ast}\right\|
\end{equation}
Thus for every compact subset $K\in\mathbb{S}_{+}^{N}$, there
exists $\acute{t}_{\varepsilon}$, depending on $K$, such that,
\begin{equation}
\label{unif_conv5}
\left\|f_{1}^{t}(X)-P^{\ast}\right\|\leq\varepsilon,~~~X\in
K,~~t\geq \acute{t}_{\varepsilon}
\end{equation}
To prove the Lemma, we need to transform the above uniform
convergence over compact sets to uniform convergence over the
entire space $\mathbb{S}_{+}^{N}$. To this end, we note that for a
observable and controllable system, the following uniform
boundedness of Riccati iterates holds from arbitrary initial state
$X\in\mathbb{S}_{+}^{N}$ (see, for example, Lemma 7.1
in~\cite{Jazwinski}):
\begin{equation}
\label{unif_conv6} f^{t}(X)\leq \alpha_{1}I,~~~t\geq N,~~\forall
X\in\mathbb{S}^{N}_{+}
\end{equation}
where $\alpha_{1}\in\mathbb{R}_{+}$ is a sufficiently large
number.

Now consider the compact set
\begin{equation}
\label{unif_conv7}
K_{\alpha}=\left\{X\in\mathbb{S}_{+}^{N}~|~\left\|X\right\|\leq
\alpha\right\}
\end{equation}
From (\ref{unif_conv6}) we have
\begin{equation}
\label{unif_conv8} f^{t}(X)\in K_{\alpha},~~~t\geq N,~~\forall
X\in\mathbb{S}^{N}_{+}
\end{equation}
Now from (\ref{unif_conv5}) choose $\acute{t}_{\varepsilon}$,
such that,
\begin{equation}
\label{unif_conv9}
\left\|f_{1}^{t}(X)-P^{\ast}\right\|\leq\varepsilon,~~~X\in
K_{\alpha},~~t\geq \acute{t}_{\varepsilon}
\end{equation}
Then defining $t_{\varepsilon}=\acute{t}_{\varepsilon}+N$
we have from eqns.~(\ref{unif_conv8},\ref{unif_conv9})
\begin{equation}
\label{unif_conv11}
\left\|f_{1}^{t}\left(X\right)-P^{\ast}\right\|\leq\varepsilon,~~~t\geq
t_{\varepsilon}
\end{equation}
and the result follows.
\end{proof}
}

\section{Stochastic boundedness of error covariances}
\label{proof_stoch_bound_er}
{\small

\textbf{Proof of Lemma~\ref{cr_sb}}
\label{app:prooflemmacr_sb}
\begin{proof} We consider the case of unstable $A$. For stable $A$, the proposition is trivial and follows from the fact, that, the unconditional variance of the state sequence reaches a steady state (hence bounded), and a suboptimal estimate $\widehat{x}_{t}\equiv 0$ for all $t$ leads to pathwise boundedness of the corresponding error covariance. In fact, in this case even $\overline{\gamma}=0$ leads to stochastic boundedness of the sequence $\{P_{t}\}$ from every initial condition.

The proof follows the same line of arguments used in
Proposition 6 of~\cite{Riccati-weakconv}, where the above claim
was established for the case of invertible $C$. The key
ingredients used there consisted of uniformly bounding the Riccati
operator (in case $C$ is invertible) and then estimating the
probability that the random sequence
$\{P_{t}\}_{t\in\mathbb{T}_{+}}$ exceeds a particular range by
relating it to the length of the random time intervals between
packet arrivals.

In the general case, as shown below, instead of bounding the
one-step Riccati iterates, we bound $N$-step Riccati iterates for
controllable and observable systems and then repeat the arguments
in~\cite{Riccati-weakconv} with more generality.

To this end, we start with the following result on boundedness of
$N$-step Riccati iterates. It can be shown (see Lemma 7.1
in~\cite{Jazwinski}) that for controllable and observable
systems\footnote{The result in~\cite{Jazwinski} applies to
time-variant system matrices under the assumptions of uniform
complete observability and uniform complete controllability, which
reduces to the observability and controllability for the
time-invariant system matrices considered in this paper.} the
following holds:
\begin{equation}
\label{cr_sb2} f_{1}^{t}\left(X\right)\preceq\kappa I,~~~t\geq
N,~~X\in\mathbb{S}_{N}^{+}
\end{equation}
where $\alpha_{2}\in\mathbb{R}_{+}$ is sufficiently large. In
other words, the above states that an application of the Riccati
operator more than $N$ times in succession leads to a covariance
bounded above by a specific constant, irrespective of the initial
state.

Now, for $M\in\mathbb{T}_{+}$ and sufficiently large, define
\begin{equation}
\label{cr_sb3}
k(M)=\min\left\{k\in\mathbb{T}_{+}~|~\kappa_{1}\alpha^{2k}+\left\|Q\right\|\frac{\alpha^{2k}-1}{\alpha^{2}-1}\geq
M\right\}
\end{equation}
where $\alpha=\left\|A\right\|$ and $\kappa_{1}=\max\{\kappa,\left\|P_{0}\right\|\}$.
Since $A$ is unstable ($\alpha>1$), it follows that
$k(M)\rightarrow\infty$ as $N\rightarrow\infty$. To estimate the
probability
$\mathbb{P}^{\overline{\gamma},P_{0}}\left(\left\|P_{t}\right\|>M\right)$
for $t\in\mathbb{T}_{+}$, define the random time $\widetilde{t}$
by
\begin{equation}
\label{cr_sb5} \widetilde{t}=\max\left\{0<s\leq
t~|~\gamma_{s-r}=1,~~1\leq r\leq N\right\}
\end{equation}
where the maximum of an empty set is taken to be zero. Thus, if
$\widetilde{t}\neq 0$,\footnote{Note that, if not zero,
$\widetilde{t}\geq N$.} it denotes the time closest to $t$, such
that, there were $N$ successive packet arrivals in the time
interval $[\widetilde{t}-N,\widetilde{t}-1]$. Then, using the
above arguments, we have
\begin{equation}
\label{cr_sb6} \left\|P_{\widetilde{t}}\right\|\leq\kappa_{1}
\end{equation}
Indeed, if $\widetilde{t}=0$, then $P_{\widetilde{t}}=P_{0}$ and
(\ref{cr_sb6}) holds by the definition of $\kappa_{1}$. On
the contrary, if $\widetilde{t}>0$, we have by (\ref{cr_sb2})
\begin{equation}
\label{cr_sb7}
\left\|P_{\widetilde{t}}\right\|=\left\|f_{1}^{N}\left(P_{\widetilde{t}-N}\right)\right\|\leq\kappa\leq\kappa_{1}
\end{equation}
We then have
\begin{equation}
\label{cr_sb8} \left\|P_{t}\right\| =
\left\|f_{\gamma_{t}-1}\circ\cdots\circ
f_{\gamma_{\widetilde{t}}}\left(P_{\widetilde{t}}\right)\right\|\leq
\left\|f_{0}^{t-\widetilde{t}}\left(P_{\widetilde{t}}\right)\right\|\leq
\kappa_{1}\alpha^{2(t-\widetilde{t})}+\left\|Q\right\|\sum_{k=1}^{t-\widetilde{t}}\alpha^{2k}=
\kappa_{1}\alpha^{2(t-\widetilde{t})}+\left\|Q\right\|\frac{\alpha^{2(t-\widetilde{t})-1}}{\alpha^{2}-1}
\end{equation}
where we have used the fact, that, $f_{1}(X)\preceq
f_{0}(X),~~~\forall~X\in\mathbb{S}_{+}^{N}$.
It then follows from the above and (\ref{cr_sb3}), that,
\begin{equation}
\label{cr_sb10}
\mathbb{P}^{\overline{\gamma},P_{0}}\left(\left\|P_{t}\right\|>M\right)\leq\mathbb{P}_{\overline{\gamma},P_{0}}\left(t-\widetilde{t}\geq
k(M)\right)
\end{equation}
We now estimate the probability
$\mathbb{P}_{\overline{\gamma},P_{0}}\left(t-\widetilde{t}=
k\right)$.

First, consider the case $\widetilde{t}\neq 0$. On the event
$\widetilde{t}\neq 0$, it is not hard to see that the following
events are equal:
\begin{equation}
\label{cr_sb11}
\left\{t-\widetilde{t}=k\right\}=\left\{\gamma_{t-k-r}=1,~~1\leq
r\leq N\right\}\bigcap_{s=t-k+1}^{t}\left\{\gamma_{s-r}=1,~~1\leq
r\leq N\right\}^{c}
\end{equation}
It then follows by elementary manipulations and the independence
of packet arrivals
\begin{eqnarray}
\label{cr_sb12}
\mathbb{P}_{\overline{\gamma},P_{0}}\left(t-\widetilde{t}=
k\right) & = &
\mathbb{P}_{\overline{\gamma},P_{0}}\left(\left\{\gamma_{t-k-r}=1,~~1\leq
r\leq N\right\}\bigcap_{s=t-k+1}^{t}\left\{\gamma_{s-r}=1,~~1\leq
r\leq N\right\}^{c}\right)\nonumber \\ & \leq &
\mathbb{P}_{\overline{\gamma},P_{0}}\left(\left\{\gamma_{t-k-r}=1,~~1\leq
r\leq
N\right\}\bigcap_{i=1}^{\lfloor\frac{k}{N}\rfloor}\left\{\gamma_{t-k+(i-1)N-1+r}=1,~~1\leq
r\leq N\right\}^{c}\right)\nonumber \\ & = &
\mathbb{P}_{\overline{\gamma},P_{0}}\left(\left\{\gamma_{t-k-r}=1,~~1\leq
r\leq
N\right\}\right)\prod_{i=1}^{\lfloor\frac{k}{N}\rfloor}\mathbb{P}_{\overline{\gamma},P_{0}}\left(\left\{\gamma_{t-k+(i-1)N-1+r}=1,~~1\leq
r\leq N\right\}^{c}\right)\nonumber \\ & = &
\overline{\gamma}^{N}\prod_{i=1}^{\lfloor\frac{k}{N}\rfloor}\left(1-\overline{\gamma}^{N}\right)\nonumber
\\ & \leq & \left(1-\overline{\gamma}^{N}\right)^{\lfloor\frac{k}{N}\rfloor}
\end{eqnarray}
On the event $\widetilde{t}=0$, using a similar set of arguments,
we can show
\begin{equation}
\label{cr_sb13}
\mathbb{P}_{\overline{\gamma},P_{0}}\left(t-\widetilde{t}=
k\right)\leq\left(1-\overline{\gamma}^{N}\right)^{\lfloor\frac{k}{N}\rfloor}
\end{equation}
We thus have the upper bound (possibly loose, but sufficient for
our purpose)
\begin{equation}
\label{cr_sb1490}
\mathbb{P}_{\overline{\gamma},P_{0}}\left(t-\widetilde{t}\geq
k(M)\right) =
\sum_{k=k(M)}^{\infty}\mathbb{P}_{\overline{\gamma},P_{0}}\left(t-\widetilde{t}=
k\right) \leq
\sum_{k=k(M)}^{\infty}\left(1-\overline{\gamma}^{N}\right)^{\lfloor\frac{k}{N}\rfloor}\leq \sum_{k=k(M)}^{\infty}\left(1-\overline{\gamma}^{N}\right)^{\frac{k}{N}-1}
\end{equation}
Rearranging and summing the geometric series above we have
\begin{equation}
\label{cr_sb14}
\mathbb{P}_{\overline{\gamma},P_{0}}\left(t-\widetilde{t}\geq
k(M)\right)\leq
\frac{1}{1-\overline{\gamma}^{N}}\sum_{k=k(M)}^{\infty}\left[\left(1-\overline{\gamma}^{N}\right)^{1/N}\right]^{k}=
\frac{1}{1-\overline{\gamma}^{N}}\frac{\left[\left(1-\overline{\gamma}^{N}\right)^{1/N}\right]^{k(M)}}{1-\left(1-\overline{\gamma}^{N}\right)^{1/N}}
\end{equation}
From eqns.~(\ref{cr_sb10},\ref{cr_sb14}) we have for all $t$ and
sufficiently large $M$
\begin{equation}
\label{cr_sb15}
\mathbb{P}^{\overline{\gamma},P_{0}}\left(\left\|P_{t}\right\|>M\right)\leq\frac{1}{1-\overline{\gamma}^{N}}\frac{\left[\left(1-\overline{\gamma}^{N}\right)^{1/N}\right]^{k(M)}}{1-\left(1-\overline{\gamma}^{N}\right)^{1/N}}
\end{equation}
Since $\overline{\gamma}>0$ and $k(M)\rightarrow\infty$ as
$M\rightarrow\infty$, it follows from the above
\begin{equation}
\lim_{M\rightarrow\infty}\sup_{t\in\mathbb{T}_{+}}\mathbb{P}^{\overline{\gamma},P_{0}}\left(\left\|P_{t}\right\|>M\right)
= 0
\end{equation}
Thus $\{P_{t}\}_{t\in\mathbb{T}_{+}}$ is s.b. (for all initial
conditions $P_{0}$) for every $\overline{\gamma}>0$ and hence, the
Lemma follows.

\end{proof}
}

\section{Proof of Theorem~\ref{thm:fin}}
\label{proof_thm:fin}
{\small

\begin{proof}
For
$\mathbf{X}=\{X_{1},\cdots,X_{n}\}\in\bigotimes_{i=1}^{n}\mathbb{S}_{+}^{N}$,
we have
\begin{equation}
\label{thm:fin3}
\mathbb{P}^{\overline{\gamma},P_{0}}\left((P_{t_{1}},\cdots,P_{t_{n}})^{\overline{\gamma}}
=\mathbf{X}\right) =
 \mathbb{P}^{\overline{\gamma},P_{0}}\left(P_{t_{1}}=X_{1}\right)\prod_{i=1}^{n-1}\mathbb{P}^{\overline{\gamma},P_{0}}\left(P_{t_{i+1}}=X_{i+1}~|~P_{t_{i}}=X_{i}\right)
\end{equation}
which follows from the Markov property. Clearly, from the above we
have
\begin{equation}
\label{thm:fin4}
\mathbb{P}^{\overline{\gamma},P_{0}}\left((P_{t_{1}},\cdots,P_{t_{n}})^{\overline{\gamma}}
=\mathbf{X}\right) =
0,~~~\mathbf{X}\notin\mathcal{N}\left(\mathcal{S}^{P_{0}}_{t_{1},\cdots,t_{n}}\right)
\end{equation}
Also, if
$\mathbf{X}\in\mathcal{N}\left(\mathcal{S}^{P_{0}}_{t_{1},\cdots,t_{n}}\right)$,
it follows from (\ref{thm:fin3}) through simple manipulations
and the independence of the packet dropouts
\begin{equation}
\label{thm:fin5}
\mathbb{P}^{\overline{\gamma},P_{0}}\left((P_{t_{1}},\cdots,P_{t_{n}})^{\overline{\gamma}}
=\mathbf{X}\right) =
\sum_{\overline{\mathcal{R}}\in\mathcal{S}^{P_{0}}_{t_{1},\cdots,t_{n}}(\mathbf{X})}(1-\overline{\gamma})^{\pi(\overline{\mathcal{R}})}\overline{\gamma}^{t_{n}-\pi(\overline{\mathcal{R}})}
\end{equation}
By manipulating each term on the R.H.S. above, we have
\begin{eqnarray}
\label{thm:fin6} \lim_{\overline{\gamma}\uparrow
1}\frac{1}{\ln\left(1-\overline{\gamma}\right)}\ln\left(
(1-\overline{\gamma})^{\pi(\overline{\mathcal{R}})}\overline{\gamma}^{t_{n}-\pi(\overline{\mathcal{R}})}\right)
& = & \lim_{\overline{\gamma}\uparrow
1}\frac{1}{\ln\left(1-\overline{\gamma}\right)}\left[\pi(\overline{\mathcal{R}})\ln(1-\overline{\gamma})+\left(t_{n}-\pi(\overline{\mathcal{R}})\right)\ln\overline{\gamma}\right]\nonumber
\\ & = &
\pi(\overline{\mathcal{R}})+\left(t_{n}-\pi(\overline{\mathcal{R}})\right)\lim_{\overline{\gamma}\uparrow
1}\frac{\ln\overline{\gamma}}{\ln\left(1-\overline{\gamma}\right)}\nonumber
\\ & = & \pi(\overline{\mathcal{R}})
\end{eqnarray}
It then follows from
Proposition~\ref{lim_num},eqns.~(\ref{thm:fin5},\ref{thm:fin6})
that
\begin{equation}
\label{thm:fin7}
\mathbb{P}^{\overline{\gamma},P_{0}}\left((P_{t_{1}},\cdots,P_{t_{n}})^{\overline{\gamma}}
=\mathbf{X}\right) =
\min_{\overline{\mathcal{R}}\in\mathcal{S}^{P_{0}}_{t_{1},\cdots,t_{n}}(\mathbf{X})}\pi(\overline{\mathcal{R}})= \ell(\mathbf{X})
\end{equation}
(the case
$\mathbf{X}\notin\mathcal{N}\left(\mathcal{S}^{P_{0}}_{t_{1},\cdots,t_{n}}\right)$
is absorbed above by using the convention, that, the minimum of an
empty set is $\infty$.)

Now consider
$B\in\mathcal{B}\left(\bigotimes_{i=1}^{n}\mathbb{S}_{+}^{N}\right)$.
If
$B\cap\mathcal{N}\left(\mathcal{S}^{P_{0}}_{t_{1},\cdots,t_{n}}\right)=\phi$,
then the claim is obvious. Hence, assume
$B\cap\mathcal{N}\left(\mathcal{S}^{P_{0}}_{t_{1},\cdots,t_{n}}\right)\neq\phi$
(this intersection is necessarily finite) and note that
\begin{equation}
\label{thm:fin8}
\mathbb{P}^{\overline{\gamma},P_{0}}\left((P_{t_{1}},\cdots,P_{t_{n}})^{\overline{\gamma}}
\in B\right) = \sum_{\mathbf{X}\in
B\cap\mathcal{N}\left(\mathcal{S}^{P_{0}}_{t_{1},\cdots,t_{n}}\right)}\mathbb{P}^{\overline{\gamma},P_{0}}\left((P_{t_{1}},\cdots,P_{t_{n}})^{\overline{\gamma}}
=\mathbf{X}\right)
\end{equation}
It then follows by Proposition~\ref{lim_num}
\begin{eqnarray}
\label{thm:fin9} \lim_{\overline{\gamma}\uparrow
1}\frac{1}{\ln\left(1-\overline{\gamma}\right)}\ln\left(\mathbb{P}^{\overline{\gamma},P_{0}}\left((P_{t_{1}},\cdots,P_{t_{n}})^{\overline{\gamma}}\in
B\right)\right) & = & \min_{\mathbf{X}\in
B\cap\mathcal{N}\left(\mathcal{S}^{P_{0}}_{t_{1},\cdots,t_{n}}\right)}\lim_{\overline{\gamma}\uparrow
1}\frac{1}{\ln\left(1-\overline{\gamma}\right)}\ln\left(\mathbb{P}^{\overline{\gamma},P_{0}}\left((P_{t_{1}},\cdots,P_{t_{n}})^{\overline{\gamma}}
=\mathbf{X}\right)\right)\nonumber \\ & = & \min_{\mathbf{X}\in
B\cap\mathcal{N}\left(\mathcal{S}^{P_{0}}_{t_{1},\cdots,t_{n}}\right)}\ell(\mathbf{X})\nonumber
\\ & = & \inf_{\mathbf{X}\in B}I^{t_{1},\cdots,t_{n}}(\mathbf{X})
\end{eqnarray}
which establishes the Theorem.
\end{proof}
}

\section{Proofs of Results in Section~\ref{inv_MDP}}
\label{proof_inv_MDP}
{\small

\textbf{Proof of Proposition~\ref{inv_MDP_rate_prop}}

\begin{proof}That $I_{L}(\cdot)$ is lower semicontinuous follows from its definition (see, for example,~\cite{FengKurtz}.) Now for $a\in\mathbb{R}_{+}$ consider the level set $K_{a}=\{X\in\mathbb{S}_{+}^{N}~|~I_{L}(X)\leq a\}$. By lower semicontinuity we know that $K_{a}$ is closed. We now show that $K_{a}$ is bounded and hence compact. To this end, we note that for all $b\in\mathbb{R}_{+}$
\begin{equation}
\label{inv_MDP_rate_prop2}
\left\{Y\in\mathbb{S}_{+}^{N}~|~I(Y)\leq b\right\}\subset\left\{Y\in\mathbb{S}_{+}^{N}~|~Y\preceq f_{0}^{\lceil b\rceil}\left(\alpha_{P^{\ast}}I\right)\right\}
\end{equation}
for some constant $\alpha_{P^{\ast}}\in\mathbb{R}_{+}$, which can be chosen independent of $b$.
Indeed, $I(Y)\leq b$ implies that $\mathcal{S}^{P^{\ast}}(Y)$ is non-empty and
\begin{equation}
\label{inv_MDP_rate_prop3}
 I(Y)=\inf_{\mathcal{R}\in\mathcal{S}^{P^{\ast}}(Y)}\pi(\mathcal{R})\leq b
\end{equation}
Since $\pi(\cdot)$ takes on integral values only, the infimum above is attained and there exists $\mathcal{R}\in\mathcal{S}^{P^{\ast}}(Y)$ with $\pi(\mathcal{R})\leq b$. Then, from Proposition~\ref{string_prop}, there exists $\alpha_{P^{\ast}}\in\mathbb{R}_{+}$ (depending on $P^{\ast}$ only), such that
\begin{equation}
\label{inv_MDP_rate_prop4}
Y=\mathcal{N}(\mathcal{R})\preceq f_{0}^{\pi(\mathcal{R})}(\alpha_{P^{\ast}}I)\preceq f_{0}^{\lceil b\rceil}(\alpha_{P^{\ast}}I)
\end{equation}
This verifies the claim in (\ref{inv_MDP_rate_prop2}).

Now consider $\varepsilon_{1}>0$. For $X\in\mathbb{S}_{+}^{N}$, the sequence $\inf_{Y\in B_{\varepsilon}(X)}I(Y)$ is non-decreasing w.r.t. $\varepsilon$ and hence $X\in K_{a}$ implies
\begin{equation}
\label{inv_MDP_rate_prop5}
\inf_{Y\in
B_{\varepsilon_{1}}(X)}I(Y)\leq a
\end{equation}
Since $I(\cdot)$ takes on integral values the infimum is attained and there exists $Y(X)\in\mathbb{S}_{+}^{N}$, such that, $I(Y(X))\leq a$. From (\ref{inv_MDP_rate_prop2}) we then have
\begin{equation}
\label{inv_MDP_rate_prop6}
Y(X)\preceq f_{0}^{\lceil a\rceil}(\alpha_{P^{\ast}}I)~\Longrightarrow~\left\|Y(X)\right\|\leq\left\|f_{0}^{\lceil a\rceil}(\alpha_{P^{\ast}}I)\right\|
\end{equation}
Since $Y(X)\in B_{\varepsilon_{1}}(X)$ we have
\begin{equation}
\label{inv_MDP_rate_prop7}
\left\|X\right\|\leq \left\|Y(X)\right\|+\varepsilon_{1} \leq\left\|f_{0}^{\lceil a\rceil}(\alpha_{P^{\ast}}I)\right\|+\varepsilon_{1}
\end{equation}
We thus note that
\begin{equation}
\label{inv_MDP_rate_prop8}
K_{a}\subset\left\{Z\in\mathbb{S}_{+}^{N}~|~\left\|Z\right\|\leq \left\|f_{0}^{\lceil a\rceil}(\alpha_{P^{\ast}}I)\right\|+\varepsilon_{1}\right\}
\end{equation}
which verifies the boundedness of $K_{a}$. Hence the level sets $K_{a}$ are closed and bounded for all $a\in\mathbb{R}_{+}$, establishing the goodness of $I_{L}(\cdot)$.

For Assertion (ii), we note that for arbitrary $\varepsilon>0$,
\begin{equation}
\label{inv_MDP_rate_prop202}
\inf_{Y\in \overline{B_{\varepsilon}(X)}}I(Y)\leq\inf_{Y\in B_{\varepsilon}(X)}I(Y)\leq\inf_{Y\in \overline{B_{\varepsilon/2}(X)}}I(X)
\end{equation}
The assertion then follows by passing to the limit as $\varepsilon\rightarrow 0$ on each side.

For Assertion (iii), note that, in general we have for arbitrary $\varepsilon>0$, $I(X)\geq\inf_{Y\in B_{\varepsilon}(X)}I(X)$
and by passing to the limit it follows
\begin{equation}
\label{lower21}
I(X)\geq\lim_{\varepsilon\rightarrow 0}\inf_{Y\in B_{\varepsilon}(X)}I(X)=I_{L}(X)
\end{equation}
This immediately gives for any set $\Gamma\in\mathcal{B}(\mathbb{S}_{+}^{N})$
\begin{equation}
\label{lower22}
\inf_{X\in\Gamma}I(X)\geq\inf_{X\in\Gamma}I_{L}(X)
\end{equation}
For the reverse inequality when $\Gamma$ is open, consider $X\in\Gamma$. Then there exists $\varepsilon_{1}>0$ (depending on $X$) such that, for every $0<\varepsilon<\varepsilon_{1}$, the open ball $B_{\varepsilon}(X)\in\Gamma$. It then follows
\begin{equation}
\label{lower23}
\inf_{Y\in B_{\varepsilon}(X)}I(Y)\geq\inf_{Y\in\Gamma}I(Y),~~~0<\varepsilon<\varepsilon_{1}
\end{equation}
Taking the limit on both sides we have
\begin{equation}
\label{lower24}
\inf_{Y\in\Gamma}I(Y)\leq\lim_{\varepsilon\rightarrow 0}\inf_{Y\in B_{\varepsilon}(X)}I(Y)=I_{L}(X)
\end{equation}
Thus for every $X\in\Gamma$, we have
\begin{equation}
\label{lower25}
I_{L}(X)\geq \inf_{Y\in\Gamma}I(Y)
\end{equation}
Taking the infimum over $X\in \Gamma$ on the L.H.S. gives the required inequality
\begin{equation}
\label{lower26}
\inf_{X\in\Gamma}I_{L}(X)\geq \inf_{X\in\Gamma}I(X)
\end{equation}
and the result follows.

We now prove Assertion (iv). By the notion of limits involving continuous arguments, it suffices to show that for every sequence $\{\varepsilon_{n}\}_{n\in\mathbb{N}}$ with
\begin{equation}
\label{add_3}
\lim_{n\rightarrow\infty}\varepsilon_{n}=0,~~~~0<\varepsilon_{n}\leq 1~\forall n
\end{equation}
we have
\begin{equation}
\label{add_2}
\lim_{n\rightarrow\infty}\inf_{Y\in\overline{K_{\varepsilon_{n}}}}I_{L}(Y)=\inf_{Y\in K}I_{L}(Y)
\end{equation}
To this end consider such a sequence $\{\varepsilon_{n}\}$ and assume on the contrary, that (\ref{add_2}) is not satisfied. Since $K\subset\overline{K_{\varepsilon_{n}}}$, we clearly have
\begin{equation}
\label{add_5}
\lim_{n\rightarrow\infty}\inf_{Y\in\overline{K_{\varepsilon_{n}}}}I_{L}(Y)\leq\inf_{Y\in K}I_{L}(Y)
\end{equation}
Thus the hypothesis that (\ref{add_2}) is not satisfied implies
\begin{equation}
\label{add_6}
\lim_{n\rightarrow\infty}\inf_{Y\in\overline{K_{\varepsilon_{n}}}}I_{L}(Y)<\inf_{Y\in K}I_{L}(Y)
\end{equation}
We note that the sets $\overline{K_{\varepsilon_{n}}}$ are compact for all $n$ and since a lower semicontinuous function attains its minimum over a compact set, for every $n\in\mathbb{N}$, there exists $X_{n}\in\overline{K_{\varepsilon_{n}}}$, such that
\begin{equation}
\label{add_4}
\inf_{Y\in\overline{K_{\varepsilon_{n}}}}I_{L}(Y)=I_{L}(X_{n})
\end{equation}
Similarly, there exists $Y^{\ast}\in K$, such that,
\begin{equation}
\label{add_7}
\inf_{Y\in K}I_{L}(Y)=I_{L}(Y^{\ast})
\end{equation}
We note that $\{I_{L}(X_{n})\}$ is a non decreasing sequence (hence the limit exists) and
\begin{equation}
\label{add_9}
\lim_{n\rightarrow\infty}I_{L}(X_{n})<I_{L}(Y^{\ast})
\end{equation}
Also, given $Z\in\mathbb{S}_{+}^{N}$, it follows from the continuity of the metric, that the function $d^{Z}:\mathbb{S}_{+}^{N}\longmapsto\mathbb{R}_{+}$ given by
\begin{equation}
\label{add_8}
d^{Z}(Y)=\left\|Y-Z\right\|,~~~\forall Y\in\mathbb{S}_{+}^{N}
\end{equation}
is continuous and hence attains its minimum over a compact set. Thus for every $n\in\mathbb{N}$, there exists $Y_{n}\in K$, such that
\begin{equation}
\label{add_10}
\inf_{Y\in K}\left\|X_{n}-Y\right\|=\left\|X_{n}-Y_{n}\right\|\leq\varepsilon_{n}
\end{equation}
where the last inequality follows from the fact that $X_{n}\in\overline{K_{\varepsilon_{n}}}$.

We note that the sequence $\{Y_{n}\}$ belongs to the compact set $K$ and hence there exists a subsequence $\{Y_{n_{k}}\}_{k\in\mathbb{N}}$ in $K$, which converges to some $\acute{Y}\in K$, i.e.,
\begin{equation}
\label{add_11}
\lim_{k\rightarrow\infty}Y_{n_{k}}=\acute{Y}
\end{equation}
Now consider the sequence $\{X_{n_{k}}\}_{k\in\mathbb{N}}$. We then have
\begin{equation}
\label{add_12}
\left\|X_{n_{k}}-\acute{Y}\right\|\leq\left\|X_{n_{k}}-Y_{n_{k}}\right\|+\left\|Y_{n_{k}}-\acute{Y}\right\|\leq\varepsilon_{n_{k}}+\left\|Y_{n_{k}}-\acute{Y}\right\|
\end{equation}
Taking the limit as $k\rightarrow\infty$ we obtain
\begin{equation}
\label{add_13}
\lim_{k\rightarrow\infty}X_{n_{k}}=\acute{Y}
\end{equation}
We then have from the lower semicontinuity of $I_{L}(\cdot)$ and (\ref{add_9})
\begin{equation}
\label{add_14}
I_{L}(\acute{Y})\leq\liminf_{k\rightarrow\infty}I_{L}(X_{n_{k}})<I_{L}(Y^{\ast})
\end{equation}
This contradicts the fact that $Y^{\ast}$ is the minimizer of $I_{L}(\cdot)$ over $K$ and we conclude that (\ref{add_2}) holds. Hence the result follows.

\end{proof}

\textbf{Proof of Lemma~\ref{top_string}}

\begin{proof}The case $\ell(\mathcal{F})=0$ is trivial as the assertion follows by choosing an arbitrary positive $t_{\mathcal{F}}$.

We consider the case $\ell(\mathcal{F})\geq 1$. We use an inductive argument and it suffices to show that for every $1\leq i\leq \ell(\mathcal{F})$ there exists positive $t_{\mathcal{F}}^{i}\in\mathbb{T}_{+}$, such that, for $\mathcal{R}\in\mathcal{U}(\mathcal{F})$ with $\mbox{len}(\mathcal{R})\geq t_{\mathcal{F}}^{i}$, we have
\begin{equation}
\label{top_string10}
\pi\left(\mathcal{R}^{t_{\mathcal{F}}^{i}}\right)\geq i
\end{equation}
We start with $i=1$. Assume on the contrary, that there is no such $t_{\mathcal{F}}^{1}\in\mathbb{T}_{+}$ with the above postulated properties. Since $\mathcal{U}(\mathcal{F})$ is non-empty, by Proposition~\ref{string_prop} Assertion (i), there exists $t_{0}\in\mathbb{T}_{+}$, such that
\begin{equation}
\label{top_string11}
\mathcal{S}_{t}^{P^{\ast}}\cap\mathcal{U}(\mathcal{F})\neq \phi
\end{equation}
Thus the non-existence of $t_{\mathcal{F}}^{1}$, implies, that for every $t\geq t_{0}$, there exists a string $\mathcal{R}_{t}\in\mathcal{U}(\mathcal{F})$, with $\mbox{len}(\mathcal{R}_{t})\geq t$, such that, $\pi\left(\mathcal{R}_{t}^{t}\right)=0$.
Such a string $\mathcal{R}_{t}$ is then necessarily of the form:
\begin{equation}
\label{top_string13}
\mathcal{R}=\left(f_{1}^{t}, f_{i_{1}}, f_{i_{2}},\cdots, f_{i_{\mbox{len}(\mathcal{R}_{t})-t}}, P^{\ast}\right)
\end{equation}
where $i_{1},\cdots,i_{\mbox{len}(\mathcal{R}_{t})-t}\in\{0,1\}$. Thus denoting
\begin{equation}
\label{top_string14}
X_{t}=f_{i_{1}}\circ f_{i_{2}}\circ\cdots\circ f_{i_{\mbox{len}(\mathcal{R}_{t})-t}}(P^{\ast})
\end{equation}
we note that
\begin{equation}
\label{top_string15}
\mathcal{N}\left(\mathcal{R}_{t}\right)=f_{1}^{t}(X_{t})
\end{equation}

Now consider the sequence $\left\{\mathcal{R}_{t}\right\}_{t\geq t_{0}}$ of such strings as $t\rightarrow\infty$. Let $\varepsilon>0$ be an arbitrary positive number. By Proposition~\ref{unif_conv}, the uniform convergence of Riccati iterates implies there exists $t_{\varepsilon}\geq N$, such that, for every $X\in\mathbb{S}_{+}^{N}$,
\begin{equation}
\label{top_string16}
\left\|f_{1}^{t}(X)-P^{\ast}\right\|\leq \varepsilon,~~~t\geq t_{\varepsilon}
\end{equation}
(we emphasize that the constant $t_{\varepsilon}$ can be chosen independently of $X$.)
Then defining $\acute{t}_{\varepsilon}=\max(t_{0},t_{\varepsilon})$, it follows from (\ref{top_string15}) that for $t\geq \acute{t}_{\varepsilon}$
\begin{equation}
\label{top_string17}
\left\|\mathcal{N}\left(\mathcal{R}_{t}\right)-P^{\ast}\right\|=\left\|f_{1}^{t}(X_{t})-P^{\ast}\right\|\leq\varepsilon
\end{equation}
Since $\varepsilon>0$ above was arbitrary, it follows that the sequence $\left\{\mathcal{N}\left(\mathcal{R}_{t}\right)\right\}_{t\geq t_{0}}$ of numerical values converges to $P^{\ast}$ as $t\rightarrow\infty$. However, by construction, the sequence $\left\{\mathcal{N}\left(\mathcal{R}_{t}\right)\right\}_{t\geq t_{0}}$ belongs to the set $\mathcal{F}$ and we conclude that $P^{\ast}$ is a limit point of the set $\mathcal{F}$. Since $\mathcal{F}$ is closed, we have $P^{\ast}\in\mathcal{F}$. It then follows
\begin{equation}
\label{top_string18}
\left\{\mathcal{R}\in\mathcal{S}^{P^{\ast}}~|~\mathcal{N}(\mathcal{R})=P^{\ast}\right\}\subset\mathcal{U}(\mathcal{F})
\end{equation}
Hence, in particular, the string $f_{1}(P^{\ast})\in\mathcal{U}(\mathcal{F})$. The fact, that $\pi\left(f_{1}(P^{\ast})\right)=0$ then contradicts the hypothesis $\ell(\mathcal{F})\geq 1$.

Thus by contradiction we establish that if $\ell(\mathcal{F})\geq 1$, there exists $t_{\mathcal{F}}^{1}$ satisfying the properties in (\ref{top_string10}) for $i=1$. Note that, if $\ell(\mathcal{F})=1$, this step completes the proof of the Lemma. In the general case, i.e., to establish (\ref{top_string10}) for all $1\leq i\leq\ell(\mathcal{F})$ we need the following additional steps.

We can now assume $\ell(\mathcal{F})\geq 2$. We assume on the contrary that the claim in (\ref{top_string10}) does not hold for all $1\leq i\leq\ell(\mathcal{F})$. By the previous arguments, the claim clearly holds for $i=1$. Then, let $1\leq k< \ell(\mathcal{F})$ be the largest integer such that the claim in (\ref{top_string10}) holds for all $1\leq i\leq k$. The hypothesis $k<\ell(\mathcal{F})$ implies there exists no $t_{\mathcal{F}}^{k+1}\in\mathbb{T}_{+}$ satisfying the claim in (\ref{top_string10}) for $i=k+1$. Since the claim holds for $i=k$, there exists $t_{\mathcal{F}}^{k}\in\mathbb{T}_{+}$, such that for all $\mathcal{R}\in\mathcal{U}(\mathcal{F})$ with $\mbox{len}(\mathcal{R})\geq t_{\mathcal{F}}^{k}$, we have $\pi\left(\mathcal{R}^{t_{\mathcal{F}}^{k}}\right)\geq k$.
The non-existence of $t_{\mathcal{F}}^{k+1}$ and (\ref{top_string11}) implies, that for every $t\geq t_{0}$, there exists a string $\mathcal{R}_{t}\in\mathcal{U}(\mathcal{F})$, with $\mbox{len}(\mathcal{R}_{t})\geq t$, such that,
\begin{equation}
\label{top_string20}
\pi\left(\mathcal{R}_{t}^{t}\right)<k+1
\end{equation}
We now study the structure of the strings $\mathcal{R}_{t}$ for sufficiently large $t$. To this end, define $\acute{t}_{0}=\max(t_{0},t_{\mathcal{F}}^{k})$. Then by the existence of $t_{\mathcal{F}}^{k}$ and (\ref{top_string20}) it follows that $\pi\left(\mathcal{R}_{t}^{t}\right)=k$ for $t\geq\acute{t}_{0}$.
Hence for $t\geq\acute{t}_{0}$, $\mathcal{R}_{t}$ is necessarily of the form:
\begin{equation}
\label{top_string22}
\mathcal{R}_{t}=\left(f_{i_{1}},\cdots, f_{i_{t_{\mathcal{F}}^{k}}}, f_{1}^{t-t_{\mathcal{F}}^{k}}, f_{j_{1}}, f_{j_{2}},\cdots, f_{j_{\mbox{len}(\mathcal{R}_{t})-t}}, P^{\ast}\right)
\end{equation}
where $i_{1},\cdots,i_{t_{\mathcal{F}}^{k}}\in\{0,1\}$, such that, $\pi\left(\mathcal{R}^{t_{\mathcal{F}}^{k}}\right)=k$ and $j_{1},\cdots,j_{\mbox{len}(\mathcal{R}_{t})-t}\in\{0,1\}$.

Now consider the sequence $\left\{\mathcal{R}_{t}\right\}_{t\geq\acute{t}_{0}}$ and define the set $\mathcal{J}$ by $\mathcal{J}=\{\mathcal{R}_{t},~~t\geq\acute{t}_{0}\}$.
Also, define $\mathcal{J}_{1}=\{\mathcal{R}\in\mathcal{S}_{t_{\mathcal{F}}^{k}}^{P^{\ast}}~|~\pi(\mathcal{R})=k\}$.
Consider the function $\Lambda^{t_{\mathcal{F}}^{k}}:\mathcal{J}\longmapsto\mathcal{J}_{1}$ by
\begin{equation}
\label{top_string25}
\Lambda^{t_{\mathcal{F}}^{k}}(\mathcal{R})=\mathcal{R}^{t_{\mathcal{F}}^{k}},~~~\forall~\mathcal{R}\in\mathcal{J}
\end{equation}
Since the cardinality of $\mathcal{J}_{1}$ is finite and $\mathcal{J}$ is countably infinite, there exists $\mathcal{R}^{\ast}\in\mathcal{J}_{1}$, such that,
the set $\left(\Lambda^{t_{\mathcal{F}}^{k}}\right)^{-1}(\{\mathcal{R}^{\ast}\})$ is countably infinite. This in turn implies, that we can extract a subsequence $\left\{\mathcal{R}_{t_{m}}\right\}_{m\geq 0}$ from the sequence $\left\{\mathcal{R}_{t}\right\}_{t\geq\acute{t}_{0}}$, such that,
\begin{equation}
\label{top_string26}
\mathcal{R}_{t_{m}}^{t_{\mathcal{F}}^{k}}=\mathcal{R}^{\ast},~~~\forall m\geq 0
\end{equation}
In other words, if $\mathcal{R}^{\ast}$ is represented by $\mathcal{R}^{\ast}=\left(f_{i^{\ast}_{1}},\cdots, f_{i^{\ast}_{t_{\mathcal{F}}^{k}}}, P^{\ast}\right)$
for some fixed $i^{\ast}_{1},\cdots,i^{\ast}_{t_{\mathcal{F}}^{k}}\in\{0,1\}$, for every $m$, the string $\mathcal{R}_{t_{m}}$ is of the form:
\begin{equation}
\label{top_string28}
\mathcal{R}_{t_{m}}=\left(f_{i^{\ast}_{1}},\cdots, f_{i^{\ast}_{t_{\mathcal{F}}^{k}}}, f_{1}^{t_{m}-t_{\mathcal{F}}^{k}}, f_{j_{1}}, f_{j_{2}},\cdots, f_{j_{\mbox{len}(\mathcal{R}_{t_{m}})-t}}, P^{\ast}\right)
\end{equation}
where $j_{1},\cdots,j_{\mbox{len}(\mathcal{R}_{t_{m}})-t}\in\{0,1\}$ are arbitrary.
Denoting
\begin{equation}
\label{top_string29}
X_{m}=f_{j_{1}}\circ f_{j_{2}}\circ\cdots\circ f_{j_{\mbox{len}(\mathcal{R}_{t_{m}})-t}}(P^{\ast}),~~~\forall m
\end{equation}
we have
\begin{equation}
\label{top_string30}
\mathcal{N}(\mathcal{R}_{t_{m}})=f_{i^{\ast}_{1}}\circ\cdots\circ f_{i^{\ast}_{t_{\mathcal{F}}^{k}}}\left(f_{1}^{t_{m}-t_{\mathcal{F}}^{k}}(X_{m})\right)
\end{equation}
Since $t_{m}\rightarrow\infty$ as $m\rightarrow\infty$, using Proposition~\ref{unif_conv} in a similar way, we have
\begin{equation}
\label{top_string31}
\lim_{m\rightarrow\infty}f_{1}^{t_{m}-t_{\mathcal{F}}^{k}}(X_{m})=P^{\ast}
\end{equation}
Noting that the function $f_{i^{\ast}_{1}}\circ\cdots\circ f_{i^{\ast}_{t_{\mathcal{F}}^{k}}}:\mathbb{S}_{+}^{N}\longmapsto\mathbb{S}_{+}^{N}$ is continuous (being the composition of continuous functions), we then have
\begin{eqnarray}
\label{top_string32}
\lim_{m\rightarrow\infty}\mathcal{N}(\mathcal{R}_{t_{m}}) & = & \lim_{m\rightarrow\infty}f_{i^{\ast}_{1}}\circ\cdots\circ f_{i^{\ast}_{t_{\mathcal{F}}^{k}}}\left(f_{1}^{t_{m}-t_{\mathcal{F}}^{k}}(X_{m})\right)\nonumber \\ & = & f_{i^{\ast}_{1}}\circ\cdots\circ f_{i^{\ast}_{t_{\mathcal{F}}^{k}}}\left(\lim_{m\rightarrow\infty}f_{1}^{t_{m}-t_{\mathcal{F}}^{k}}(X_{m})\right)\nonumber \\ & = & f_{i^{\ast}_{1}}\circ\cdots\circ f_{i^{\ast}_{t_{\mathcal{F}}^{k}}}\left(P^{\ast}\right)\nonumber \\ & = & \mathcal{N}\left(\mathcal{R}^{\ast}\right)
\end{eqnarray}
Thus the sequence $\left\{\mathcal{N}\left(\mathcal{R}_{t_{m}}\right)\right\}_{m\geq 0}$ in $\mathcal{F}$ converges to $\mathcal{N}\left(\mathcal{R}^{\ast}\right)$ as $m\rightarrow\infty$. Hence, $\mathcal{N}\left(\mathcal{R}^{\ast}\right)\in\mathcal{F}$ as $\mathcal{F}$ is closed and $\mathcal{N}\left(\mathcal{R}^{\ast}\right)$ is a limit point of $\mathcal{F}$. This in turn implies $\mathcal{R}^{\ast}\in\mathcal{U}(\mathcal{F})$. Since $\pi(\mathcal{R}^{\ast})=k$, this contradicts the hypothesis $k<\ell(\mathcal{F})$ and the claim in (\ref{top_string10}) holds for all $1\leq i\leq \ell(\mathcal{F})$. This establishes the Lemma.

\end{proof}

\textbf{Proof of Lemma~\ref{tight}}

\begin{proof}Let $a>0$ be arbitrary and $z\in\mathbb{N}$ be such that $z\geq a$. From Proposition~\ref{string_prop} Assertion (iii), there exists $\alpha_{P^{\ast}}\in\mathbb{R}_{+}$, such that,
\begin{equation}
\label{tight2}
\mathcal{N}(\mathcal{R})\preceq f_{0}^{\pi(\mathcal{R})}(P^{\ast}),~~~\forall \mathcal{R}\in\mathcal{S}^{P^{\ast}}
\end{equation}
Define $b\in\mathbb{R}_{+}$ such that $\left\|f_{0}^{z}(P^{\ast})\right\|<b$
and consider the compact set $K_{a}=\left\{X\in\mathbb{S}_{+}^{N}~|~\left\|X\right\|\leq b\right\}$.
Also define the closed set, $\mathcal{F}_{b}$, by $\mathcal{F}_{b}=\left\{X\in\mathbb{S}_{+}^{N}~|~\left\|X\right\|\geq b\right\}$.
As per Lemma~\ref{top_string}, define the set $\mathcal{U}(\mathcal{F}_{b})$ as
\begin{equation}
\label{tight6}
\mathcal{U}(\mathcal{F}_{b})=\left\{\mathcal{R}\in\mathcal{S}^{P^{\ast}}~|~\mathcal{N}(\mathcal{R})\in\mathcal{F}_{b}\right\}
\end{equation}
We then have the following inclusion:
\begin{equation}
\label{tight6-b}
\mathcal{U}(\mathcal{F}_{b})\subset\left\{\mathcal{R}\in\mathcal{S}^{P^{\ast}}~|~\pi(\mathcal{R})\geq z\right\}
\end{equation}
and hence
\begin{equation}
\label{tight7}
\ell(\mathcal{F}_{b})=\inf_{\mathcal{R}\in\mathcal{U}(\mathcal{F}_{b})}\pi(\mathcal{R})\geq z
\end{equation}
Since $\mathcal{F}_{b}$ is closed, by Lemma~\ref{top_string} there exists $t_{\mathcal{F}_{b}}\in\mathbb{T}_{+}$, such that,
\begin{equation}
\label{tight8}
\pi(\mathcal{R}^{t_{\mathcal{F}_{b}}})\geq z,~~~\forall \mathcal{R}\in\mathcal{U}(\mathcal{F}_{b})
\end{equation}
To estimate the probability $\mathbb{\mu}^{\overline{\gamma}}(K_{a}^{C})$, we now follow a similar set of arguments as used in Lemma~\ref{upper}. First note that we have by weak convergence:
\begin{equation}
\label{tight9}
\mathbb{\mu}^{\overline{\gamma}}(K_{a}^{C})\leq\liminf_{t\rightarrow\infty}\mathbb{P}^{\overline{\gamma},P^{\ast}}\left(P_{t}^{\overline{\gamma},P^{\ast}}\in K_{a}^{C}\right)\leq\liminf_{t\rightarrow\infty}\mathbb{P}^{\overline{\gamma},P^{\ast}}\left(P_{t}^{\overline{\gamma},P^{\ast}}\in \mathcal{F}_{b}\right)
\end{equation}
For $t\in\mathbb{T}_{+}$ define the sets $\mathcal{J}_{t}^{P^{\ast}}=\mathcal{S}_{t}^{P^{\ast}}\cap\mathcal{U}(\mathcal{F}_{b})$.
For $t\geq t_{\mathcal{F}_{b}}$ we have (see also (\ref{upper14}))
\begin{equation}
\label{tight10}
\mathbb{P}^{\overline{\gamma},P^{\ast}}\left(P_{t}^{\overline{\gamma},P^{\ast}}\in \mathcal{F}_{b}\right)=\sum_{\mathcal{R}\in\mathcal{J}_{t}^{P^{\ast}}}(1-\overline{\gamma})^{\pi(\mathcal{R})}\overline{\gamma}^{t-\pi(\mathcal{R})}\leq {t_{\mathcal{F}_{b}}\choose\ell(\mathcal{F}_{b})}(1-\overline{\gamma})^{\ell(\mathcal{F}_{b})}\leq {t_{\mathcal{F}_{b}}\choose\ell(\mathcal{F}_{b})}(1-\overline{\gamma})^{z}
\end{equation}
A familiar set of arguments as in Lemma~\ref{upper} yields the following from eqns.~(\ref{tight9},\ref{tight10}):
\begin{equation}
\label{tight11}
\mathbb{\mu}^{\overline{\gamma}}(K_{a}^{C})\leq {t_{\mathcal{F}_{b}}\choose\ell(\mathcal{F}_{b})}(1-\overline{\gamma})^{z},~~~\forall \overline{\gamma}
\end{equation}
from which we obtain
\begin{equation}
\label{tight12}
\limsup_{\overline{\gamma}\uparrow 1}-\frac{1}{\ln(1-\overline{\gamma})}\mathbb{\mu}^{\overline{\gamma}}(K_{a}^{C})\leq -z\leq -a
\end{equation}
Thus for every $a>0$ there exists a compact set $K_{a}$ such that (\ref{tight1}) is satisfied and the Lemma follows.
\end{proof}
}

\section{Proof of Lemma~\ref{compMDP}}
\label{proof_compMDP}
{\small

\begin{proof}
By Theorem~\ref{thm:mdp} it suffices to show that for all $M>0$,
\begin{equation}
\label{compMDP4}
\inf_{X\in B_{M}^{C}(P^{\ast})}I(X)=\iota(M)
\end{equation}
\begin{equation}
\label{compMDP5}
\inf_{X\in \overline{B}_{M}^{C}(P^{\ast})}I(X)=\iota_{+}(M)
\end{equation}
We prove (\ref{compMDP4}) only, the proof of (\ref{compMDP5}) being similar.

The class $\mathcal{A}$ assumption implies (through the same line of arguments as in Proposition~\ref{string_prop} Assertion (iii)) that
\begin{equation}
\label{compMDP6}
\mathcal{N}(\mathcal{R})\preceq f_{0}^{\pi(\mathcal{R})}(P^{\ast}),~~~\forall~\mathcal{R}\in\mathcal{S}^{P^{\ast}}
\end{equation}
By definition it follows that $f_{0}^{\iota(M)}(P^{\ast})\in B_{M}^{C}(P^{\ast})$
and hence
\begin{equation}
\label{compMDP8}
\inf_{X\in B_{M}^{C}(P^{\ast})}I(X)\leq \iota(M)
\end{equation}
Now assume on the contrary that (\ref{compMDP4}) does not hold. Then by (\ref{compMDP8}) we have
\begin{equation}
\label{compMDP9}
\inf_{X\in B_{M}^{C}(P^{\ast})}I(X)< \iota(M)
\end{equation}
Thus there exists $\mathcal{R}\in\mathcal{S}^{P^{\ast}}$, such that,
\begin{equation}
\label{compMDP10}
\mathcal{N}\left(\mathcal{R}\right)\in B_{M}^{C}(P^{\ast}),~~~\pi(\mathcal{R})\leq\iota(M)-1
\end{equation}
We then have from (\ref{compMDP6})
\begin{equation}
\label{compMDP11}
\mathcal{N}(\mathcal{R})\leq f_{0}^{\pi(\mathcal{R})}(P^{\ast})\leq f_{0}^{\iota(M)-1}(P^{\ast})
\end{equation}
which implies
\begin{equation}
\label{compMDP12}
\left\|f_{0}^{\iota(M)-1}(P^{\ast})-P^{\ast}\right\|\geq\left\|\mathcal{N}(\mathcal{R})-P^{\ast}\right\|\geq M
\end{equation}
This contradicts the definition of $\iota(M)$ which is the smallest non-negative integer $k$, such that, $\left\|f_{0}^{k}(P^{\ast})-P^{\ast}\right\|\geq M$.
We thus conclude that the claim in (\ref{compMDP4}) holds.
\end{proof}

\bibliographystyle{IEEEtran}
\bibliography{IEEEabrv,CentralBib}

\end{document}